\documentclass{amsart}
\usepackage[latin9]{inputenc}
\usepackage{geometry}
\usepackage{xcolor}
\usepackage{graphicx}
\usepackage{mathrsfs}
\usepackage[normalem]{ulem}
\usepackage{amstext}
\usepackage{amsthm}
\usepackage{amssymb}
\usepackage[unicode=true,
 bookmarks=false,
 breaklinks=false,pdfborder={0 0 1},backref=section,colorlinks=false]
 {hyperref}

\makeatletter

\usepackage{enumitem}

\author[Kao]{Lien-Yung Kao}
\address{George Washington University, Washington, D.C. 20052, lkao@gwu.edu}
\author[Martone]{Giuseppe Martone}
\address{Sam Houston State University, Huntsville, Texas 77341, gxm120@shsu.edu}
\thanks{Kao was partially supported by the Simons Foundation (Gift No.~956047) and by the National Science Foundation under Award No.~2441413. Also, parts of this work were done during his visit to the National Center for Theoretical Sciences (NCTS) in Taiwan. 
He would like to thank NCTS for their support and hospitality. Martone was partially supported by an
AMS-Simons Research Enhancement Grant for PUI faculty and by a New Faculty Investement grant from Sam Houston State University. He would like to thank the Mathematics Department at Yale University were he completed part of this work.}

\usepackage{enumitem}

\theoremstyle{plain}
\newtheorem{theorem}{Theorem}[section]
\newtheorem{thm}[theorem]{Theorem}
\newtheorem{prop}[theorem]{Proposition}
\newtheorem{assumptions}[theorem]{Assumptions}
\newtheorem{lemma}[theorem]{Lemma}
\newtheorem{lem}[theorem]{Lemma}

\newtheorem{cor}[theorem]{Corollary}
\newtheorem{definition}[theorem]{Definition}
\newtheorem{remark}[theorem]{Remark}
\newtheorem{rem}[theorem]{Remark}

\theoremstyle{definition}
\newtheorem*{remark*}{\textbf{Remark}}

\newtheorem{notation}[theorem]{Notation}
\newtheorem{thmx}{Theorem}

\numberwithin{equation}{section}
\numberwithin{figure}{section}

\newtheorem*{thm*}{\protect\theoremname}
\newtheorem*{lem*}{\protect\lemmaname}

\newcommand{\boxt}{I^2_{\xi,m}(t)}

\DeclareMathOperator{\Rb}{\mathbb{R}}

\newcommand{\wt}[1]{\widetilde{#1}}
\newcommand{\wh}[1]{\widehat{#1}}

\newcommand{\vertiii}[1]{{\left\vert\kern-0.25ex\left\vert\kern-0.25ex\left\vert #1 
    \right\vert\kern-0.25ex\right\vert\kern-0.25ex\right\vert}}
	
	\newcounter{notes}

\def\bb #1{\mathbb{#1}}
\def\bf #1{\mathbf{#1}}
\def\cal #1{\mathcal{#1}}

\def\rm #1{\mathrm{#1}}
\def\sf #1{\mathsf{#1}}
\def\wt #1{\widetilde{#1}}
\def\wh #1{\widehat{#1}}
\def\ul #1{\underline{#1}}

\newcommand{\eee}{\epsilon}
\newcommand{\Lip}{\text{Lip}}

\newcommand{\fbf}{\mathbf f}

\newcommand{\xbf}{\mathbf x}
\newcommand{\ybf}{\mathbf y}
\newcommand{\zbf}{\mathbf z}

\def\inn #1{\langle #1\rangle}

\makeatother

\providecommand{\lemmaname}{Lemma}

\providecommand{\theoremname}{Theorem}

\begin{document}
\title[Correlation number]{Correlation number for potentials with entropy gaps and cusped Hitchin
representations}
\begin{abstract} We introduce a correlation number for two strictly positive, locally H\"older continuous, independent potentials with strong entropy gaps at infinity on a topologically mixing countable state Markov shift with BIP. We define in this way a correlation number for pairs of cusped Hitchin representations. Furthermore,
we explore the connection between the correlation number and the Manhattan curve, along with several rigidity properties of this correlation number.
\end{abstract}

\maketitle
\setcounter{tocdepth}{1}
\tableofcontents

\section{Introduction} 

In this article, we introduce a correlation number for pairs of well-behaved potentials over countable-state Markov shifts, defined through the study of their simultaneous orbital distributions. These dynamical results let us define a correlation number for pairs of cusped Hitchin representations. This correlation number quantifies the similarities of their marked length spectra and captures some rigidity properties.  

Our dynamical results extend Lalley's earlier work \cite{Lalley:1987df} to non-compact settings. We derive the geometric counterparts through the symbolic coding of cusped Hitchin representations. A similar approach was successfully employed in previous work by Bray, Canary, and the authors \cite{BCKMcounting}, where we established counting and equidistribution results for cusped Hitchin representations. In this article, we obtain a  comparison of the growth rate of the marked length spectra of a pair of cusped Hitchin representations.  

The primary motivation for this work stems from (higher rank) Teichm\"uller theory, especially from the study of geometrically meaningful diverging sequences of (cusped) Hitchin representations.  Motivated by results of Dai and the second author \cite{Dai}, we wish to study the behavior of the correlation number along these sequences and interpret its limit. For example, given two sequences of hyperbolic structures converging to distinct points in the same stratum of the augmented Teichm\"uller space, as a consequence of Theorem \ref{thm:main application}, we can now define the correlation number of the limiting points and it would be interesting to compare it to the correlation number along the diverging sequences. 

In the following subsections, we present our dynamical results first, and then their applications to geometry.


\subsection*{General thermodynamical results}
See Section \ref{sec:background} for precise definitions. Let $\Sigma^{+}$ be a topologically mixing (one-sided) countable state Markov
shift with BIP and consider two strictly positive, locally H\"older continuous, and independent potentials $f,g\colon\Sigma^{+}\to\mathbb{R}$. Let $\cal A$ denote the countable alphabet for $\Sigma^+$. Following \cite{BCKMcounting}, we say that a potential $f$ has a {\em strong entropy gap at infinity} if the series
\[
Z_1(f,s)=\sum_{a\in\mathcal A}e^{-s\sup\{f(\ul x)\colon  \ul x\in\Sigma^+\text{ and } x_1=a\}}
\]
has a a finite critical exponent $d(f)>0$ and diverges when $s=d(f)$. For the rest of this introduction, assume that $f$ and $g$ have strong entropy gaps at infinity, $\|f-g\|_\infty$ is finite, and set ${\bf f}=(f,g)$.

We fix $\xi, m>0$: $\xi$ is a (small) precision in our estimates, while $m$ is an {\em admissible slope}, defined below in terms of the {\em Manhattan curve of ${\bf f}$}. For any $t>0$, let $I^2_{\xi, m}(t)$ denote the open rectangle $(t,t+\xi)\times (mt,mt+\xi)$ and for any $n\in\bb Z_{>0}$ define
\[
\cal M(n,t)=\mathcal{M}(n,t;{\bf f,\xi,m):=\{\ul x\in\Sigma^{+}\colon\ \ul x\in \rm {Fix}^n,\ S_{n}{\bf f(\ul x)\in I^2_{\xi,m}(t)\}}}
\]
where $\rm {Fix}^n$ is the set of {\em $n$-periodic words} of $\Sigma^+$ and $S_{n}{\bf f}$ denotes the pair of {\em
$n$-th ergodic sums} of $\bf f$, that is
\[
S_{n}{\bf f}(\ul x):=\left(\sum_{k=1}^{n}f(\sigma^{k-1}(\ul x)),\sum_{k=1}^{n}g(\sigma^{k-1}(\ul x))\right).
\]

The first goal of this article is to study the growth rate as $t$ goes to infinity of
\[
M(t;{\bf f,m,\xi):=\sum_{n}\frac{1}{n}\#\mathcal{M}(n,t;{\bf f,m,\xi)}}.
\]

\begin{thmx}\label{thm:weaker-main}
For any admissible slope $m$, $M(t;{\bf f},m,\xi)$ grows exponentially. Specifically, the limit 
\[
\alpha_{\fbf}(m) := \lim_{t\to\infty} \frac{1}{t} \log M(t;{\bf f},m,\xi)
\]
exists, and moreover, $0 < \alpha_{\fbf}(m) < \infty.$
\end{thmx}
In what follows, we examine properties of the exponential growth rate $\alpha_{\fbf}(m)$ of $M(t;{\bf f},m,\xi)$, which we call the {\em correlation number} of $\fbf$ and $m$. A {\em cylinder} $p$ is a subset of $\Sigma^+$ consisting of words sharing an initial string of length $|p|\in\mathbb N$. Our second result provides a more precise asymptotic growth estimate for
\[
M_{p}(t;{\bf f,m,\xi):=\sum_{n}\frac{1}{n}\#\mathcal{M}_{p}(n,t;{\bf f,m,\xi)}}
\]
where, for any cylinder $p$, we write
$\mathcal{M}_{p}(n,t;{\bf f,m,\xi)\colon=\mathcal{M}(n,t;{\bf f,m,\xi)\cap p}}$. Notably, we obtain a finer asymptotic estimate when we restrict to a cylinder.

\begin{thmx}\label{thm:local-counting}
For any admissible slope $m$ and any cylinder $p$, there exist constants $C_{1}(p)$ and $C_{2}(p)$ such that 
\[
C_{1}(p) \leq \lim_{t \to \infty} \frac{t^{\frac{3}{2}}}{e^{\alpha_{\fbf}(m)t}} M_{p}(t;{\bf f},m,\xi) \leq C_{2}(p),
\]
and 
\[
\lim_{|p| \to \infty} \frac{C_{1}(p)}{C_{2}(p)} = 1.
\]
\end{thmx}

By restricting our focus to the counting problem on cylinders, Theorem \ref{thm:local-counting} provides an estimate for the \textit{local} asymptotic expansion of the orbital distribution with respect to $\fbf$ and $m$. Unlike in the compact case, the \textit{global} asymptotic expansion does not immediately follow from the local version. While we expect that the global asymptotic expansion also holds, we would need a stronger version of Lemma \ref{lem:rough-bounds} to establish it. See Remark \ref{rem:obsticle} for more details. However, let us note that the local asymptotic expansion suffices for our main applications.

Next, we explicitly relate the growth rate $\alpha_\fbf(m)$ to the Manhattan curve, an important dynamical object associated with orbital distribution problems for $\fbf$, first introduced by Burger in \cite{Burger:1993wb}. Let $P$ denote the {\em topological pressure} associated with $\Sigma^+$. We recall (see Section \ref{sec:manhattan} for more details) that the {\em Manhattan curve of $\fbf$} is the curve
\[
\mathcal{C}(\fbf) = \{(a,b) \in \mathbb{R}^{2} \colon P(-af - bg) = 0,\ a \geq 0,\ b \geq 0,\ a + b > 0\}.
\]
By the analyticity properties of the pressure function $P$ and \cite[Theorem C]{BCKMcounting}, we know that the Manhattan curve $\mathcal{C}(\fbf)$ can be parametrized by $(s, q(s))$, where $q = q(s)$ is a real analytic function for $s \in [0, \delta_f]$. Here, $\delta_f$ is the unique value for which $P(-\delta_f f) = 0$. The slope of the normal to the Manhattan curve at the point $(s, q(s))$ is given by $m(s) = -1 / q'(s)$. We then say that $m = m(s)$ is an {\em admissible slope} and we denote by $\cal S(\fbf)$ the set of admissible slopes. In particular, note that $m_* = \delta_f / \delta_g \in \cal S(\fbf)$ (see Theorem \ref{thm:manhattan}).

Given an admissible slope \(m\), we denote by \((a_m, b_m)\) the point on \(\mathcal{C}(\mathbf{f})\) where the slope of the normal equals \(m\). With these notations, we can define the quantity 
\[
H_{\mathbf{f}}(m) := a_m + m \cdot b_m.
\]
The following result shows that \(H_{\mathbf{f}}(m)\) equals the correlation number of $\fbf$ and $m$.

\begin{thmx}\label{thmx:2} If $m\in\cal S(\fbf)$ is an admissible slope, then 
\[
\alpha_{\fbf}(m)=H_{\fbf}(m).
\] 
\end{thmx}

Using the convexity of the Manhattan curve and the above results, we derive a rigidity result for the special admissible slope $m_*=\delta_f/\delta_g$:
\begin{cor}[Correlation number rigidity]\label{cor:corr-rigidity}
For $m_*=\delta_f/\delta_g$, we have
\[
\alpha_\fbf(m_*)=H_{\fbf}(m_*)\leq\delta_f.
\]
Moreover, the equality holds if and only if $m_*S_{n}f(\ul x)=S_{n}g(\ul x)$ for all $\ul x\in \rm {Fix}^{n}.$
\end{cor}
Another interesting dynamical growth rate associated with $\fbf$ is $h_{(\alpha, \beta)}^{BS}({\bf f})$, the {\em $(\alpha,\beta)$-Bishop-Steger entropy}, defined for $\alpha, \beta > 0$ as
\[
h_{(\alpha, \beta)}^{BS}({\bf f}) := \lim_{t \to \infty} \frac{1}{t} \log \sum_{n} \frac{1}{n} \#\left\{ \ul x\in \rm{Fix}^n \colon \alpha S_{n} f(\ul{x}) + \beta S_{n} g(\ul{x}) \leq t \right\}.
\]
Our last result in the symbolic setting relates the Bishop-Steger entropy with the correlation number. 

\begin{thmx}  \label{Thm:BSentropy}  
Let $m \in \mathcal{S}(\fbf)$ be an admissible slope. Then,
\[
\frac{H_{{\bf f}}(m)}{\alpha + m\beta} \leq h_{(\alpha, \beta)}^{BS}({\bf f}),
\]
with equality if and only if $m$ is the unique slope satisfying $a_{m}/b_{m} = \alpha/\beta$.
\end{thmx}

\subsection*{Application to cusped Hitchin representations} We apply the results above to pairs of cusped Hitchin representations. See Section \ref{sec:cuspedHit} for details. Let $\Gamma$ be a torsion-free geometrically finite Fuchsian group which is not convex-cocompact. The recurrent portion of the geodesic flow on $T^1(\bb H^2/\Gamma)$ can be coded by a topologically mixing countable states Markov shift $\Sigma^+$ with BIP. In particular, there is a map $G\colon\text{Fix}^n\to \Gamma$ such that if $\gamma\in\Gamma$ is hyperbolic, then it is conjugated to $G(\ul x)$ for some $n$ and a unique up to shift $\ul x\in\text{Fix}^n$.

A representation $\rho\colon\Gamma\to\mathsf{SL}(d,\mathbb R)$ is {\em cusped Hitchin} if there exists a continuous, $\rho$-equivariant, positive map from the limit set of $\Gamma$ to the space of complete flags in $\mathbb R^d$. Using a Lie theoretic construction, a cusped Hitchin representation together with a choice of a nonzero positive linear combination $\phi$ of the simple roots of $\mathsf{SL}(d,\bb R)$ determines a length function $\ell^\phi_\rho\colon\Gamma\to\mathbb R_{\geq 0}$ which is constant on the conjugacy class $[\gamma]\in[\Gamma]$ of an element $\gamma\in\Gamma$. Then, for any cusped Hitchin representation $\rho\colon\Gamma\to\mathsf{SL}(d,\bb R)$ and any such $\phi$ there exists a strictly positive, locally H\"older continuous potential $\tau_\rho^\phi$ with a strong entropy gap at infinity and such that
\[
\tau_\rho^\phi(\ul x)=\ell^\phi_\rho(G(\ul x))
\] 
for any $\ul x\in\text{Fix}^n$. Moreover, given another cusped Hitchin representation $\eta\colon\Gamma\to\mathsf{SL}(d,\bb R)$, then $\|\tau_\rho^\phi-\tau_\eta^\phi\|_\infty$ is finite and $\tau_\rho^\phi$ and $\tau_\eta^\phi$ are independent unless $\eta=\rho,(\rho^{-1})^\top$, see Lemma \ref{lem:independence}. 

Recall that for a cusped Hitchin representation $\rho$ and a nonzero positive linear combination $\phi$ of the simple roots of $\mathsf{SL}(d, \mathbb{R})$, we can define $\delta_\phi(\rho)$, the {\em $\phi$-topological entropy}, as $\delta_{\tau^\phi_\rho}$, the entropy of the corresponding potential $\tau^\phi_\rho$. Furthermore, for a pair of cusped Hitchin representations $(\rho, \eta)$ and a given $\phi$, we can define the Manhattan curve $\mathcal{C}^\phi(\rho, \eta)$ by $\mathcal C(\tau^\phi_\rho, \tau^\phi_\eta)$. (See \cite[Cor.~1.3, Cor.~1.4]{BCKMcounting} for more details.)
Thus, we can apply our results from symbolic dynamics to obtain the following.
 
\begin{thmx}\label{thm:main application} Let $\rho,\eta\colon\Gamma\to\mathsf{SL}(d,\bb R)$ be cusped Hitchin representations such that $\eta\neq \rho,(\rho^{-1})^\top$ and let $\phi$ be a nonzero positive linear combination of the simple roots of $\mathsf{SL}(d,\bb R)$. Let $m$ be an admissible slope for the Manhattan curve $\mathcal C^\phi(\rho,\eta)$. Then
\[
\lim_{t\to \infty}\frac{1}{t}\log\left\{[\gamma]\in[\Gamma]\colon (\ell^\phi_\rho(\gamma),\ell^\phi_\eta(\gamma))\in I^2_{\xi, m}(t)\right\}=:H^\phi_{(\rho,\eta)}(m)=a_m+m\cdot b_m.
\] 
where $H^{\phi}_{(\rho,\eta)}(m)$ is the correlation number of $(\rho,\eta)$ and $\phi$, and $(a_m,b_m)$ is the point $\mathcal{C}^\phi(\rho,\eta)$ where the slope of the normal equals $m$.
\end{thmx}

The following rigidity result follows immediately from Theorem \ref{thm:main application}, Corollary \ref{cor:corr-rigidity}, and \cite[Cor.~1.5]{BCKMcounting}.

\begin{cor}[Correlation number rigidity II]
Under the same assumptions as in Theorem \ref{thm:main application}, for $m_* = \delta_\phi(\rho) / \delta_\phi(\eta)$, we have
\[
H^\phi_{(\rho,\eta)}(m_*) \leq \delta_\phi(\rho).
\]
Moreover, equality holds if and only if $m_* \ell^\phi_\rho(\gamma) = \ell^\phi_\eta(\gamma)$ for all $\gamma \in \Gamma$.
\end{cor}

Similarly, we can define the {\em $(\alpha, \beta)$-Bishop-Steger entropy} of $\phi$, $\rho$ and $\eta$ as $h_{(\alpha, \beta)}^{\phi}(\rho,\eta):=
h_{(\alpha, \beta)}^{BS}(\tau^\phi_\rho, \tau^\phi_\eta).$ Then, Theorem \ref{Thm:BSentropy} can be rewritten as follows:

\medskip
\noindent
$\mathbf{Theorem\ \ref{Thm:BSentropy}^*\scalebox{1.5}{.}}$ 
{\em 
Let $m$ be an admissible slope. Then,
\[
\frac{H^\phi_{(\rho,\eta)}(m)}{\alpha + m\beta} \leq h_{(\alpha, \beta)}^{\phi}(\rho,\eta),
\]
with equality if and only if $m$ is the unique slope satisfying $a_m / b_m = \alpha / \beta$.
}

\subsection*{Outline of the paper and main steps} 

The approach in this paper combines ideas from \cite{Lalley:1987df} and \cite{BCKMcounting}. 
The results in \cite{BCKMcounting} rely on a renewal theorem, which is not available when studying simultaneous orbital distribution problems. 
For this reason, we follow Lalley's strategy from the compact case \cite{Lalley:1987df} and use Fourier analysis to study the asymptotic behavior of the simultaneous orbital distribution of a pair of potentials.

Lalley's method involves two key steps. 
The first step consists of converting the counting problem into estimates on thermodynamical quantities. 
The second step uses Fourier analysis on these quantities to derive the relevant asymptotic expansions. The simultaneous orbital growth problem studied in this paper is a refinement of the single-orbital growth problem treated in \cite{BCKMcounting}. 
To extend Lalley's method to the non-compact setting, we build upon the framework in \cite{BCKMcounting} to establish more delicate estimates that we can then use to perform the two steps as above.

In what follows, we outline the main steps of each section. Section \ref{sec:background} gathers essential background results and definitions from the thermodynamic formalism and the entropy gaps at infinity introduced in \cite{BCKMcounting}. In Section \ref{sec:first results}, we begin our study of the correlation number \( H_\fbf(m) \) and establish the rigidity results given in Corollary \ref{cor:corr-rigidity} and Theorem \ref{Thm:BSentropy}.  
Section \ref{sec:prep} reformulates the orbit counting problem by linking it to the transfer operator, enabling the application of tools from thermodynamic formalism. In Section \ref{sec:globalgrowth}, we derive a priori estimates for the counting problem over countable Markov shifts, giving an upper bound for the simultaneous orbital growth. Section \ref{sec:local-estimate} uses Fourier analysis and the Saddle Point Method to establish a local asymptotic expansion, providing a lower bound for the simultaneous orbital growth and thereby proving Theorems \ref{thm:weaker-main}, \ref{thm:local-counting}, and \ref{thmx:2}. In Section \ref{sec:cuspedHit}, we review key aspects of the theory of cusped Hitchin representations and establish Lemma \ref{lem:independence}, which allows us to apply of our main dynamical results in this context, leading to the proof of Theorem \ref{thm:main application}. Finally, in Appendix \ref{sec:Proof-of-holo} and Appendix \ref{app:saddle-point}, we provide proofs of Theorem \ref{thm:tran-holo} and Proposition \ref{prop:saddle-point}, which are known to experts, but we were not able to find in the available literature.

\subsection*{Historical remarks}
Orbital distribution has long been a central theme in ergodic geometry. A key early milestone in this field is the Prime Orbit Theorem (see, for example, the foundational works of Huber, Margulis, Lalley, Parry, and Pollicott \cite{Huber:1959ed,Margulis:1970gj,Lalley:1989jm,Parry:1990tn}), which establishes a connection between the growth rate of closed orbits and topological entropy. Lalley's pioneering work \cite{Lalley:1987df,Lalley:1989eh} introduces probabilistic perspectives to study orbital distribution problems within compact settings. Building on this foundation, along with earlier contributions by the authors \cite{Kao:2018th,Kao:2019tr,BCKMcounting}, our work extends Lalley's results to a large class of countable Markov shifts.

Lalley's results in \cite{Lalley:1987df} initiated further investigations into correlation numbers, which Sharp \cite{Sharp:1998if} later explored through a different approach. Lalley's subsequent work \cite{Lalley:1989eh} focused on counting closed orbits within specific homology classes, drawing inspiration from Phillips and Sarnak \cite{Phillips87homology} and Katsuda and Sunada \cite{Katsuda88homology}. This problem was further studied by Pollicott and Sharp \cite{Pollicott91homology,Sharp93homology} using a zeta function approach. Babillot and Ledrappier \cite{Babillot:1998kh} later provided a unified framework for these two closely related counting problems in compact hyperbolic flows.

Schwartz and Sharp \cite{Schwartz:1993ty} used Lalley's work \cite{Lalley:1987df} to define and examine the correlation of length spectra of two hyperbolic structures on a closed surface. Pollicott and Sharp \cite{Pollicott06correlation,Pollicott13correlation} further analyzed a related correlation number, defined relative to word length. Glorieux \cite{Glorieux17} studied correlation numbers with slope from the perspective of GHMC Anti-De Sitter manifolds. Dai and the second author \cite{Dai} extended Schwartz and Sharp's approach to pairs of Hitchin representations. Chow and Oh \cite{chow2023jordan,chow2024multiplecor} recently developed a new approach and generalized these results to tuples of Anosov groups.

Cusped Hitchin representations are examples of positive representations in the sense of
Fock and Goncharov \cite{Fock:2006da}. Canary, Zhang, and Zimmer \cite{CZZcusped} establish several of their geometric properties, which extend results on Anosov representations. Bray, Canary, and the authors developed dynamical aspects of this theory in \cite{BCKMcounting, BCKMpressure}.
We refer to \cite{Canary:2024} for a survey on (cusped) Hitchin representations.

In the non-convex cocompact setting, constrained counting problems remain less developed. Epstein \cite{Epstein87homology} studied the counting of orbits in given homology classes for finite-volume hyperbolic manifolds, while Babillot and Peign\'e \cite{Babillot:2000fv} extended these results to certain infinite-volume hyperbolic manifolds with cusps. To the best of our knowledge, no other results on correlation numbers exist in non-convex cocompact settings, apart from Theorems \ref{thm:weaker-main} and \ref{thm:local-counting}. Moreover, these two results partially answer open question no.7 listed on \cite[p.11]{Pollicott:2012ud}.

Theorem \ref{thmx:2} generalizes \cite[Thm 1]{Sharp:1998if} and \cite[Theorem 3.25]{Glorieux17} (with slopes) for convex cocompact Fuchsian representations, as well as \cite[Thm 6.14]{Dai} for Hitchin representations. Corollary \ref{cor:corr-rigidity} draws inspiration from the Intersection Number Rigidity result in \cite[Cor 1.5]{BCKMcounting}. Finally, Theorem \ref{Thm:BSentropy} studies the relation between the correlation number and (generalized) Bishop-Steger entropy, extending \cite[Theorem 3.15]{Glorieux17} for convex cocompact Fuchsian representations to cusped Hichin representations.

\subsection*{Acknowledgements} The authors would like to thank Harry Bray, Dick Canary, Xian Dai, Fran\c{c}ois Ledrappier, and Hee Oh for helpful and insightful conversations. We thank the anonymous referee for helpful and thoughtful feedback on an earlier version of this manuscript.

\section{Background}\label{sec:background}
\subsection{Markov shifts}
A {\em countable state Markov shift} is the data of a countable {\em alphabet} $\mathcal A$, a {\em transition matrix} $\mathbb T=(t_{ab})\in \{0,1\}^{\mathcal A\times \mathcal A}$, the set of {\em words}
\[
\Sigma^+=\{\ul x=(x_n)\in\mathcal A^{\mathbb Z_{>0}}\colon t_{x_nx_{n+1}}=1\}
\]
and the {\em left-shift} $\sigma\colon\Sigma^+\to\Sigma^+$ defined by $\sigma((x_n))=(x_{n+1})$.  The countable state Markov shift is: (i) {\em topologically mixing} if for any letters $a,b\in\mathcal A$ there exists $N=N(a,b)$ such that for any $n> N$ there exists a word $\ul x\in\Sigma^+$ with $x_1=a$ and $x_n=b$; (ii) has {\em BIP}, short for {\em big images pre-images property}, if there exists a finite subset $\mathcal B$ of the alphabet $\mathcal A$ such that for any $a\in\mathcal A$ there exist letters $b_p,b_s\in\mathcal B$ such that $t_{b_pa}=t_{ab_s}=1$. For the rest of this section, we fix a topologically mixing, countable state Markov shift with BIP which, with a slight abuse of notation, we denote by $\Sigma^+$. Furthermore, we equip $\Sigma^+$ with the metric $d(\ul x,\ul y)=e^{-\inf \{n:\ x_n\neq y_n\}}$.

For any $k\in\mathbb Z_{>0}$, a {\em $k$-cylinder} is a non-empty subset $p\subset \Sigma^+$ defined by the property
\[
\ul x,\ul y\in p\iff x_i=y_i\text{ for all }i=1,\dots,k.
\] 
Denote by $\Lambda_k$ the set of $k$-cylinders.

\subsection{Potentials, entropy gaps, and pressure}
A {\em potential} is a continuous function $f\colon \Sigma^+\to\mathbb K$ where $\mathbb K=\mathbb R$ or $\mathbb C$ equipped with the norm $|\cdot |$.
The {\em $n$-th ergodic sum} of $f$ is the potential 
\[
S_nf(\ul x)=\sum_{i=0}^{n-1}f(\sigma^i(\ul x)).
\]

A potential $f\colon\Sigma^+\to\mathbb K$ is {\em locally H\"older continuous} (or $\beta$-{\em locally H\"older continuous with constant $A$}) if there exist $A,\beta>0$ such that 
\[
|f(\ul x)-f(\ul y)|<Ad(\ul x,\ul y)^\beta
\] 
whenever $x_1=y_1$. We denote by $\mathcal F_\beta$ (resp. $\mathcal F_\beta(\mathbb C)$) the space of $\beta$-locally H\"older continuous potentials valued in $\mathbb R$ (resp. $\mathbb C$) equipped with the norm 
\[
\|\phi\|_{\beta}=\|\phi\|_{\infty}+\Lip(\phi)\ 
\text{where}\ \Lip(\phi):=\sup\left\{\frac{|\phi(\ul x)-\phi(\ul y)|}{d(\ul x,\ul y)^\beta}\colon\ \ul x\neq \ul y,\ x_1=y_1\right\}.
\] 
Then, let $\mathcal F^b_\beta(\mathbb C)$ denote the subspace of bounded potentials and let $\mathcal B(\mathcal F^b_\beta(\mathbb C))$ be the space of bounded operators on $\mathcal F^b_\beta(\mathbb C)$ equipped with the operator norm
\[
\|\mathcal L\|_{\rm {op}}=\sup\left\{\frac{\|\mathcal L(\phi)\|_\beta}{\|\phi\|_\beta}:\ \phi\in \mathcal F^b_\beta(\mathbb C)\right\}.
\]
Note that $(\mathcal F^b_\beta(\mathbb C),\|\cdot\|_\beta)$ and $(\mathcal B(\mathcal F^b_\beta(\mathbb C)),\|\cdot\|_{\rm {op}})$ are Banach spaces.

Let $f\colon\Sigma^+\to\mathbb R$ be locally H\"older continuous and set $
S(f,a)=\sup\{f(\ul x)\colon \ul x\in\Sigma^{+},x_{1}=a\}.$
Then, recall from the introduction that $f$ has a {\em strong entropy gap at infinity} if the series
\[
Z_{1}(f,s)=\sum_{a\in\mathcal{A}}e^{-sS(f,a)}
\] 
has a finite critical exponent $d(f)>0$ and diverges when $s=d(f)$. 

Let $\mathcal M_\sigma$ denote the space of shift-invariant probability measures on $\Sigma^+$ and, for any $\mu \in\mathcal M_\sigma$, let $h_\sigma(\mu)$ denote the measure-theoretic entropy of $\sigma$ with respect to $\mu$. The {\em (topological) pressure} of the locally H\"older continuous potential $f$ is
\[
P(f)=\sup\left\{h_\sigma(\mu)+\int_{\Sigma^+}fd\mu\colon \mu\in\mathcal M_\sigma\text{ and }-\int_{\Sigma^+}f\ d\mu<\infty\right\}.
\]
An \emph{equilibrium state} for a potential $f$ is a $\sigma$-invariant Borel probability measure $\mu$ on $\Sigma^+$ such that 
\[
P(f) = h_\sigma(\mu) + \int_{\Sigma^+} f \, d\mu.
\]
A \emph{Gibbs state} for $f$ is a Borel probability measure $\nu$ on $\Sigma^+$ for which there exists a constant $Q > 1$ such that
\[
\frac{1}{Q} \leq \frac{e^{S_n f(\ul{x}) - nP(f)}}{\nu(p)} \leq Q
\]
for any $n$-cylinder $p$ and any $\ul{x}$ in $p$. 

The Gibbs states and equilibrium states of potentials with good regularity (if they exist) are closely related. Specifically, we have the following result:

\begin{thm}[{Mauldin-Urba\'nski \cite[Thm 2.2.9]{Mauldin:2003dn}, Sarig \cite[Thm 4.9]{Sarig:2009wta}}]\label{thm:gibbs and equilibrium} 
Let $\Sigma^+$ be a topologically mixing, countable state Markov shift with BIP. If $f$ is a real-valued, locally H\"older continuous potential with finite pressure, then $f$ admits a unique shift-invariant Gibbs state $\mu_f$. Moreover, if $-\int f \, d\mu_f < \infty$, then $\mu_f$ is the unique equilibrium state for $f$.
\end{thm}

Two locally H\"older continuous potentials $f$ and $g$ are {\em (Liv\v{s}ic) cohomologous} if there exists a locally H\"older continuous potential $h$ such that $f-g=h-h\circ \sigma$. The real-valued potential $f$ is \emph{strictly positive} if there exists a constant $B > 0$ such that $f(\ul{x}) > B$ for all $\ul{x} \in \Sigma^+$. 

\begin{rem}\normalfont Several of the results from \cite{BCKMcounting} used in this work hold more generally for {\em eventually positive} potentials. However, in this paper we will need the stronger hypothesis of strict positivity, as in \cite[Theorem C]{BCKMcounting}. We note that, by \cite[Lemma 3.2]{BCKMcounting}, any eventually positive locally H\"older continuous potential with a strong entropy gap at infinity is cohomologous to a strictly positive potential with the same properties.
\end{rem}

A potential $f\colon\Sigma^+\to\mathbb R$ is {\em arithmetic} if the additive subgroup of $\mathbb R$ generated by $\{S_nf(\ul x)\colon \ul x\in\text{Fix}^n,\ n\in\mathbb N\}$ is cyclic.
Two potentials are {\em independent} if $af+bg$ is arithmetic only if $a=b=0$. Note that if $f$ and $g$ are independent, then each one of them is non-arithmetic.

\begin{assumptions}\label{rmk:potentials}
From now on through the paper, we assume that $\Sigma^{+}$ is a topologically
	mixing, countable state Markov shift with BIP, and let $f,g\colon\Sigma^+\to\mathbb R$ denote strictly positive locally H\"older continuous potentials with strong entropy gaps at infinity such that $\|f-g\|_\infty$ is finite. We often write $\fbf$ for the pair $(f,g)$.
\end{assumptions}

\begin{remark}\normalfont 
\begin{enumerate}
    \item When clear from context, our notations will not highlight the dependence on the fixed pair of potentials $\fbf$.

    \item Assuming that $\|f-g\|_\infty$ is finite and $f$ has a strong entropy gap at infinity guarantees that the same property holds for $g$. However, we explicitly assume that $g$ has a strong entropy gap at infinity to ease exposition.
\end{enumerate}
\end{remark}

\subsection{Transfer operators} 
We are interested in certain weighted sums of $f$ and $g$ that have finite pressure. More specifically, we will consider weights in the set
\begin{align}\label{eq:D(f)}
\begin{split}
D=D(\fbf):=\Bigg(&\{(z_{1},z_{2})\in\mathbb{R}^{2}\colon z_1\leq 0,z_2\leq 0,z_1+z_2<0\}\\
&\cup \{(z_{1},z_{2})\in\mathbb{R}^{2}\colon z_1> 0,z_2<0,z_1c(f)+z_2(c(f)-A)<0\}\\
&\cup \{(z_{1},z_{2})\in\mathbb{R}^{2}\colon z_1<0,z_2>0,z_1(c(g)-A)+z_2c(g)<0\}\Bigg)\\
&\hspace{-1em}\cap\{(z_{1},z_{2})\in\mathbb{R}^{2}\colon d(-z_{1}f-z_{2}g)<1\}
\end{split}
\end{align}
where $A>0$ is such that $\|f-g\|_{\infty}=A<\infty$, $c(f)=\inf_{\Sigma^+}f>0$, and $c(g)=\inf_{\Sigma^+}g>0$. Notice that for any $(z_{1},z_{2})\in D$, the potential $-(z_{1}f+z_{2}g)$ is locally H\"older continuous, strictly positive, and with strong entropy gap at infinity. In particular, \cite[Lemma 3.3(1)]{BCKMcounting} implies that $D\subset \{(z_{1},z_{2})\in\mathbb{R}^{2}\colon P(z_{1}f+z_{2}g)<\infty\}$. 
We now list some important properties of the weights in $D$.
\begin{prop}\label{cor:gibbs and eq for zf}
	 For any $(z_{1},z_{2})\in D$, the following holds.
	\begin{enumerate}
\item The potiential $z_{1}f+z_{2}g$ has a unique equilibrium state $\mu_{z_{1}f+z_{2}g}$.
\item The potentials $f,g$ are in $L^{n}(\mu_{z_{1}f+z_{2}g})$ for all $n\in\mathbb{N}.$
\item $D$ is an open set in $\mathbb R^2$.
\end{enumerate}
\end{prop}
\begin{proof} 	
The first assertion follows from the assumption $d(-z_1f-z_2g)<1$ together with \cite[Lemma 3.4]{BCKMcounting}.
	For the second assertion, since $-(z_{1}f+z_{2}g)$ has
	a strong entropy gap at infinity, the series 
	\[
	Z_{1}(-z_{1}f-z_{2}g,s)=\sum_{a\in\mathcal{A}}e^{-sS(-z_{1}f-z_{2}g,a)}
	\]
	converges for all $s>d(-(z_{1}f+z_{2}g)).$ Since we assume $(z_1,z_2)\in D$, we can find $\epsilon>0$ small enough such that $1-\epsilon>d(-(z_{1}f+z_{2}g))$, and thus
	\[
	\sum_{a\in\mathcal{A}}e^{-(1-\epsilon)S(-z_{1}f-z_{2}g,a)}<\infty.
	\]
	Then,	
	\begin{align*}
		\infty>\sum_{a\in\mathcal{A}}e^{-(1-\epsilon)S(-z_{1}f-z_{2}g,a)} & =\sum_{a\in\mathcal{A}}e^{-S(-z_{1}f-z_{2}g,a)}e^{\epsilon S(-z_{1}f-z_{2}g,a)}\\
		& =\sum_{a\in\mathcal{A}}e^{-S(-z_{1}f-z_{2}g,a)}\sum_{n\geq1}\frac{\epsilon^{n}}{n!}\left(S(-z_{1}f-z_{2}g,a)\right)^{n}\\
		& =\sum_{n\geq1}\sum_{a\in\mathcal{A}}\frac{\epsilon^{n}}{n!}\left(S(-z_{1}f-z_{2}g,a)\right)^{n}e^{-S(-z_{1}f-z_{2}g,a)}
	\end{align*}
	where we can swap the series because $(z_1,z_2)\in D$, $-z_1f-z_2g$ is strictly positive, thus guaranteeing absolute convergence. It follows that for all $n\in\mathbb{N}$
	\[
	\sum_{a\in\mathcal{A}}\left(S(-z_{1}f-z_{2}g,a)\right)^{n}e^{-S(-z_{1}f-z_{2}g,a)}<\infty.
	\]
    Now, for all $a\in\mathcal A$, let $[a]$ denote the cylinder of words starting with the letter $a$ and, for any locally H\"older continuous potential $h$, let $I(h,a)$ denote the infimum of $\{h(\ul x)\colon \ul x\in[a] \}$.
	By the Gibbs property of $\mu_{z_1f+z_2g}$, we know that there exists $Q>1$ such that
	\[
	e^{-I(-z_{1}f-z_{2}g,a)}= e^{S(z_{1}f+z_{2}g,a)}\geq\frac{e^{P(z_{1}f+z_{2}g)}}{Q}\mu_{z_{1}f+z_{2}g}([a]).
	\]
Hence 
	\begin{align*}
	0<\int\left|-z_{1}f-z_{2}g\right|^{n}\ d\mu_{z_{1}f+z_{2}g}&\leq\frac{Q}{e^{P(z_{1}f+z_{2}g)}}\sum_{a\in\mathcal{A}}\left(S(-z_{1}f-z_{2}g,a)\right)^{n}e^{-I(-z_{1}f-z_{2}g,a)}\\
	&\leq \frac{Q}{e^{P(z_{1}f+z_{2}g)}}\sum_{a\in\mathcal{A}}\left(S(-z_{1}f-z_{2}g,a)\right)^{n}e^{-S(-z_{1}f-z_{2}g,a)+A}<\infty
	\end{align*}
	where we used that the potential $-z_1f-z_2g$ is locally H\"older continuous to see that
	\[
	|I(-z_{1}f-z_{2}g,a)-S(-z_{1}f-z_{2}g,a)|<A
	\]
	for some uniform constant $A>0$. From here we deduce that $f,g\in L^{n}(\mu_{z_{1}f+z_{2}g})$
	for all $n\in\mathbb{N}$.
    
The last assertion is an immdeiate consequence of the formula $$d(-z_1f-z_2g))=\frac{d(f)}{-(z_1+z_2)}$$
established in the proof of \cite[Theorem C*]{BCKMcounting}.
\end{proof}

Consider a locally H\"older continuous potential $h\colon\Sigma^+\to\mathbb K$. The {\em transfer operator} $\mathcal{L}_{h}$ of $h$
(also known as the {\em Ruelle-Perron-Frobenius operator}) is defined by 
\[
\mathcal{L}_{h}\phi(\underline{x})=\sum_{\ul y\in\sigma^{-1}(\ul x)}e^{h(\ul y)}\phi(\ul y)
\]
where $\phi\colon\Sigma^+\to\mathbb K$ is a \emph{bounded} locally H\"older continuous potential. The eigenvalues and eigenvectors of the transfer operator are related to the pressure and equilibrium states as follows.

\begin{thm}[{Mauldin-Urba\'nski \cite[Cor. 2.7.5]{Mauldin:2003dn}, Sarig \cite[Thm 4.9]{Sarig:2009wta}}]
\label{transfer fact}
If $u :\Sigma^+\to\mathbb R$ is locally H\"older continuous, \hbox{$P(u)<+\infty$}, and $\sup u<+\infty$ then there
exist unique probability measures $\mu_u$ and $\nu_u$ on $\Sigma^+$  and a positive function
$h_u:\Sigma^+\to\mathbb R$
so that
$$\mu_u=h_u\nu_u,\qquad \mathcal L_uh_u=e^{P(u)}h_u,\qquad\mathrm{and}\qquad
\mathcal L_u^*\nu_u=e^{P(u)}\nu_u.$$
Moreover, $h_u$ is bounded away from both $0$ and $+\infty$ and $\mu_u$ is an equilibrium state for $u$.
\end{thm}

We will be interested in complex analytic perturbations of the transfer operator. To this end, consider the complex domain $$
\overline{D}=\overline{D}(\fbf):=\{(z,w)\subset\mathbb{C}\times\mathbb{C}:\ (\Re{z},\Re{w})\in D\}$$
which contains $D=D(\fbf)$.

\begin{thm}
[Holomorphicity] \label{thm:tran-holo} The map $\overline D\to \mathcal{B}(\mathcal{F}_{\beta}^{b}(\mathbb{C}))$ defined by $(z,w)\mapsto\mathcal{L}_{zf+wg}$
is holomorphic.
\end{thm}

\begin{proof}
The proof follows from \cite[Prop. 2 (3)]{Sarig06phase} with some minor modifications and simplifications. For the sake of completeness and
readability, we present the argument in Appendix \ref{sec:Proof-of-holo}.
\end{proof}
We recall the well-known analytic perturbation theorem (cf. for example,
\cite{Kato95-pertubation}, \cite[Prop. 4.6]{Parry:1990tn}, \cite[Thm. 5.6]{Sarig:2009wta},
or \cite[Theorem]{Benoit19Perturb}):
\begin{thm}
[Analytic Perturbation Theorem] \label{thm:APT}  Let $X$ be a complex,
(resp. real) Banach space and $\mathcal{B}(X)$ be the space of bounded
linear operators acting on $X$ endowed with the operator norm. If
$\mathcal{L}_{0}\in\mathcal{B}(X)$ has a simple isolated eigenvalue,
then there is an open neighborhood $V$ of $\mathcal L_{0}$ such that every $\mathcal{L}\in V$ has an eigenvalue $\lambda_{\mathcal L}$ close to
$\lambda_{0}$. The map $\lambda\colon V\to\mathbb{C}\ (resp.\ \mathbb{R})$
is analytic, $\mathcal{L}$ does not have other eigenvalues near $\lambda_0$,
and there is another analytic map $u:V\to X$ such that $u_{\mathcal L}$ is
an eigenvector of $\mathcal{L}$ for $\lambda_{\mathcal{L}}$. 
\end{thm}

The following corollary is a consequence of the analytic perturbation theorem discussed above, which can be viewed as a complex version of the Ruelle-Perron-Frobenius theorem. See, for example, Sarig \cite[Thm 5.8]{Sarig:2009wta} for the standard Ruelle-Perron-Frobenius theorem for the unperturbed transfer operator in the countable setting. This corollary plays a key role in our local asymptotic estimates in Section \ref{sec:local-estimate}.
\begin{cor}
\label{thm:decay-cor} Let ${\bf z}=(z_{1},z_{2})\in D$ and for any
${\bf \Theta}=(\theta_1,\theta_2)\in\mathbb{R}^{2}$ let $w({\bf \Theta}):=(z_1+i\theta_1)f+(z_2+i\theta_2)g$.
\begin{enumerate}

\item The map ${\bf \Theta}\mapsto\mathcal{L}_{w({\bf \Theta})}$
is analytic.

\item There exists $\epsilon>0$ such that for all $\|{\bf \Theta}\|<\epsilon$
there are functions $P(w({\bf \Theta}))$, $h(w({\bf \Theta}))$,
and $\nu_{w({\bf \Theta})}$ satisfying
\[
\mathcal{L}_{w({\bf \Theta})}h(w({\bf \Theta}))=e^{P(w({\bf \Theta}))}h(w({\bf \Theta}))\ \text{ and } \mathcal{L}^*_{w({\bf \Theta})}\nu_{w({\bf \Theta})}=e^{P(w({\bf \Theta}))}\nu_w({\bf \Theta}) 
\]
where $\mathcal{L}^*_{w({\bf \Theta})}$ is the dual operator
of $\mathcal{L}_{w({\bf \Theta})}$ and $\nu_{w({\bf \Theta})}$
is a Radon complex measure.

\item For all $\|{\bf \Theta}\|<\epsilon$, the maps ${\bf \Theta}\mapsto P(w({\bf \Theta}))$
and ${\bf \Theta}\mapsto h(w({\bf \Theta}))$ are analytic, and ${\bf \Theta}\mapsto\nu_{w({\bf \Theta})}$
is weak{*} analytic, in the sense that for any $\phi\in\mathcal{F}_{\beta}^{b}(\mathbb{C})$, ${\bf \Theta}\mapsto\nu_{w({\bf \Theta})}(\phi)$
is analytic. 

\item For all $\|{\bf \Theta}\|<\epsilon$, there exist $K_{1}>0$, and $r_{1}\in(0,1)$
such that for any $\phi\in\mathcal{F}_{\beta}^{b}(\mathbb{C})$ 
\[
\left\|e^{-nP(w({\bf \Theta}))}\mathcal{L}_{w({\bf \Theta})}^{n}\phi-h(w({\bf \Theta}))\int \phi\ d\nu_{w({\bf \Theta})}\right\|_{{\rm \beta}}\leq K_{1}r_{1}^{n}\|\phi\|_{\beta}.
\]

\item If $f$ and $g$ are independent, then there exists $\delta>0$
such that for any $\phi\in\mathcal{F}_{\beta}^{b}(\mathbb{C})$
\[
\lim_{n\to\infty}(1+\delta)^n\|e^{-nP(w({\bf 0}))}\mathcal{L}_{w({\bf \Theta})}^{n}\phi\|_{{\rm \beta}}=0
\]
for all ${\bf \Theta}\neq(0,0)$ and this holds uniformly for $\bf\Theta$ in a compact set.
\end{enumerate}
\end{cor}
\begin{proof}
This corollary is standard once the analytic perturbation theorem and the 
complex Ruelle--Perron--Frobenius theorem are available 
(cf.\ \cite[Appendix~1]{Lalley:1987df} for the case of subshifts of finite type). 
The first assertion follows from Theorem~\ref{thm:tran-holo}. 
Since \({\bf z} \in D\), we know that \(\mathcal{L}_{w({\bf 0})}\) is a bounded linear operator 
with a simple isolated eigenvalue of \(e^{P(w({\bf 0}))}\) 
(cf.\ \cite[Thm~5.8]{Sarig:2009wta}). 
Therefore, the second and third assertions are direct consequences of 
Theorem~\ref{thm:APT}. 
The fourth assertion follows from the proof of \cite[Prop.~5]{Lalley:1987df}. 
Finally, the last assertion is established using the proof of \cite[Prop.~6]{Lalley:1987df}, 
where the complex Ruelle--Perron--Frobenius theorem is provided by 
\cite[Thm~2.14]{Kessebohmer:2017ke} in this context.
\end{proof}

The (\textit{asymptotic}) \textit{covariance }of $f$
and $g$ with respect to the shift-invariant measure $\mu$ is given
by 
\[
\mathrm{Cov}(f,g,\mu):=\lim_{n\to\infty}\frac{1}{n}\int S_{n}\left(f-\int f\ d\mu\right)S_{n}\left(g-\int g\ d\mu\right)\ d\mu
\]
and the (\textit{asymptotic}) \textit{variance} of $f$ with respect to $\mu$ is $\text{Var}(f,\mu):=\text{Cov}(f,f,\mu).$

\begin{cor}\label{cor:analy+postive-def}
For any ${\bf z}=(z_{1},z_{2})\in D$, let $\mu=\mu_{z_{1}f+z_{2}g}$
be the unique equilibrium state of $z_{1}f+z_{2}g$. Then \normalfont
\begin{enumerate}
\item $(z_{1},z_{2})\mapsto P(z_{1}f+z_{2}g):=\mathbb{P}({\bf z})$ is analytic
in $D$;
\item $\frac{\partial}{\partial z_{1}}\mathbb{P}({\bf z})=\int f\ d\mu$ and
$\frac{\partial}{\partial z_{2}}\mathbb{P}({\bf z})=\int g\ d\mu$;
\item $\frac{\partial^{2}}{\partial z_{1}^{2}}\mathbb{P}({\bf z})=\mathrm{Var}(f,\mu)$,
$\frac{\partial^{2}}{\partial z_{2}^{2}}\mathbb{P}({\bf z})=\mathrm{Var}(g,\mu)$,
and $\frac{\partial^{2}}{\partial z_{1}\partial z_{2}}\mathbb{P}({\bf z})=\frac{\partial^{2}}{\partial z_{2}\partial z_{1}}\mathbb{P}({\bf z})=\mathrm{Cov}(f,g,\mu)$;
\item if $f$ and $g$ are independent, then $\nabla^{2}\mathbb{P}$ is positive
definite in $D$.
\end{enumerate}
\end{cor}

\begin{proof}
The first assertion follows from the Analytic Perturbation Theorem (Theorem \ref{thm:APT}). Specifically, for any ${\bf z} \in D$, we know that the transfer operator $\mathcal{L}_{z_{1}f+z_{2}g}$ is a bounded linear operator on $\mathcal{B}(\mathcal{F}_{\beta}^{b})$, with a simple, isolated eigenvalue $e^{\mathbb{P}({\bf z})}$ (cf. \cite[Thm 5.8]{Sarig:2009wta}). By applying Theorem \ref{thm:tran-holo}, we see that the map ${\bf z} \mapsto \mathcal{L}_{z_{1}f+z_{2}g}$ is real-analytic for ${\bf z} \in D$, and the analyticity of ${\bf z} \mapsto \mathbb{P}({\bf z})$ then follows directly from Theorem \ref{thm:APT}.

The second and third assertions are consequences of Proposition \ref{cor:gibbs and eq for zf} parts (1) and (2) and the formulas for the first and second derivatives of the pressure given in \cite[Prop. 2.6.13, Prop. 2.6.14]{Mauldin:2003dn}.

For the final assertion, we need to show that $\frac{\partial^{2}}{\partial z_{1}^{2}}P(z_{1}f+z_{2}g)>0$ and that $\det(\nabla^{2}\mathbb{P})>0$. Indeed, $\frac{\partial^{2}}{\partial z_{1}^{2}}P(z_{1}f+z_{2}g)>0$ follows from \cite[Thm 3 \& Rmk(2) on p. 635]{Sarig06phase}. Specifically, since $f \in L^{2}(\mu)$ by Proposition \ref{cor:gibbs and eq for zf}(2), \cite[Thm 3 \& Rmk(2) on p. 635]{Sarig06phase} states that $\mathrm{Var}(f,\mu)=0$ if and only if $f$ is cohomologous to a constant. However, since $f$ and $g$ are independent, we know that neither $f$ (nor $g$) can be cohomologous to a constant. Therefore, $\text{Var}(f,\mu)>0$. Now, $\det(\nabla^{2}\mathbb{P})>0$ follows from the non-degeneracy of the variance using a standard argument from probability theory. Namely, the positivity of $\det(\nabla^{2}\mathbb{P})$ follows, via a short computation, from the positivity of the variance of the function $g - \frac{\text{Cov}(f,g,\mu)}{\text{Var}(f,\mu)}f$.
\end{proof}

\subsection{The Legendre transform}\label{sec:legendre}

Recall (see Assumptions \ref{rmk:potentials}) that $f$ and $g$ are strictly positive locally H\"older continuous potentials with strong entropy gaps at infinity and $\|f-g\|_\infty$ is finite. We assume further that $f$ and $g$ are independent. For ${\bf f}=(f,g)$, we consider the domain $D$ defined in Equation \ref{eq:D(f)}.

Part (4) in Corollary \ref{cor:analy+postive-def} implies that the map ${\mathbb{P}: D \to \mathbb{R}}$ is convex. The \textit{Legendre transform} of ${\mathbb{P}}$ is defined for all ${\bf x\in\nabla\mathbb{P}(}D)$ as
\[
\mathbb{P^{*}}({\bf x})=\sup_{{\bf z\in}D}\left(\langle{\bf x},{\bf z}\rangle-\mathbb{P}({\bf z})\right)
\]
where $\langle\cdot,\cdot\rangle$ denotes the standard inner product on $\mathbb R^2$.
\begin{prop}
	\label{prop:legen-trans} For any ${\bf x\in\nabla\mathbb{P}(}D)$, the following properties hold.
	
	\begin{enumerate}
		
		\item $\mathbb{P}^{*}$ is analytic on $\nabla\mathbb{P}(D)$;
		
		\item $\nabla\mathbb{P}^{*}\circ\nabla\mathbb{P}=id$ on $D$;
		
		\item $\nabla\mathbb{P}\circ\nabla\mathbb{P}^{*}=id$ on $\nabla\mathbb{P}(D)$;
		
		\item\label{item:legendre 4} $\nabla\mathbb{P}({\bf z})={\bf x}$ $\iff$ $\nabla\mathbb{P}^{*}({\bf x})={\bf z}\iff\nabla^{2}\mathbb{P}^{*}({\bf x})=\left(\nabla^{2}\mathbb{P}({\bf z})\right)^{-1}$;
		
		\item\label{item:legendre 5} $\nabla\mathbb{P}({\bf z})={\bf x}$ $\implies$ $\mathbb{P}^{*}({\bf x})+\mathbb{P}({\bf z})=\langle{\bf x},{\bf z}\rangle;$
		
		\item $\nabla\mathbb{P}({\bf z})\neq{\bf x}$ $\implies$ $\mathbb{P}^{*}({\bf x})+\mathbb{P}({\bf z})>\langle{\bf x},{\bf z}\rangle;$
		
		\item $\nabla\mathbb{P}({\bf z})={\bf x}$ $\implies$ $\mathbb{P}^{*}({\bf x})=-h_{\mu_{\langle{\bf z},{\bf f}\rangle}}(\sigma).$
		
	\end{enumerate}
\end{prop}

\begin{proof}
	Since $\mathbb{P}$ is analytic and strictly convex in $D$, all but the last assertion are classical properties of the Legendre transform
	(cf. \cite[Ex. 11.9]{Rockafellar98_Legendre} for a proof). 
	
	The last
	assertion follows from the derivative formula for the pressure, namely Corollary
	\ref{cor:analy+postive-def}(2). 
    Recall that $\mu=\mu_{z_{1}f+z_{2}g}=\mu_{\langle{\bf z},{\bf f}\rangle}$ is the equilibrium state of $z_{1}f+z_{2}g$, thus $$P(z_{1}f+z_{2}g)=h_{\mu}(\sigma)+\int {(z_{1}f+ z_{2}g)}\,  d\mu.$$
    Then, since by assumption 
	\[
	{\bf x}=\nabla\mathbb{P}({\bf z})=(\partial_{z_{1}}P(z_{1}f+z_{2}g),\partial_{z_{2}}P(z_{1}f+z_{2}g))=\left(\int f\ d\mu,\int g\ d\mu\right)
	\]
	we apply part (5) and compute
	\begin{align*}
		\mathbb{P}^{*}({\bf x})= \langle{\bf x},{\bf z}\rangle-\mathbb{P}({\bf z})
		=  \int z_{1}f\ d\mu+\int z_{2}g\ d\mu-P(z_{1}f+z_{2}g)
		= -h_{\mu}(\sigma).
	\end{align*}
\end{proof}

\begin{lem}
	\label{lem:cov-legen} For any ${\bf x\in\mathbb{R}^{2}}
	$ and any $t>0$ such that $\xbf/t\in\nabla\bb P(D)$ the following holds.\begin{enumerate}\normalfont
		
		\item $\frac{d}{dt}(-t\mathbb{P}^{*}({\bf x}/t))=\mathbb{P}(\nabla\mathbb{P^{*}}({\bf x}/t));$
		
		\item $\frac{d^{2}}{dt^{2}}(-t\mathbb{P}^{*}({\bf x}/t))=-\frac{1}{t}\left(\nabla^{2}\mathbb{P}^{*}({\bf x}/t)\right)\langle{\bf x}/t, {\bf x}/t\rangle<0$. 
		
	\end{enumerate}
	In particular, the map
	$t\mapsto-t\mathbb{P}^{*}({\bf x}/t)$ is concave down.
\end{lem}

\begin{proof}
	For part (1), we compute
	\begin{align*}
		\frac{d}{dt}(-t\mathbb{P}^{*}({\bf x}/t)) & =-\mathbb{P}^{*}({\bf x}/t)+(-t)\langle\nabla\mathbb{P}^{*}({\bf x}/t),-{\bf x}/t^{2}\rangle\\
		& =-\mathbb{P}^{*}({\bf x}/t)+\langle\nabla\mathbb{P}^{*}({\bf x}/t),{\bf x}/t\rangle\\
		& =\mathbb{P}(\nabla\mathbb{P}^{*}({\bf x}/t))
	\end{align*}
	where the last equality follows from Proposition \ref{prop:legen-trans} (5).
		
	For (2), we take another derivative of the above equation. Then
	\begin{align*}
		\frac{d^{2}}{dt^{2}}(-t\mathbb{P}^{*}({\bf x}/t)) & =\frac{d}{dt}\mathbb{P}(\nabla\mathbb{P}^{*}({\bf x}/t))\\
		& =\nabla\mathbb{P}(\nabla\mathbb{P}^{*}({\bf x}/t))\cdot\nabla^{2}\mathbb{P}^{*}({\bf x}/t)\cdot (-{\bf x}/t^2)
        =-\frac{1}{t}\left(\nabla^{2}\mathbb{P}^{*}({\bf x}/t)\right)\langle{\bf x}/t, {\bf x}/t\rangle.
	\end{align*}
	The second derivative is smaller than zero because $\nabla^{2}\mathbb{P}^{*}$
	is positive definite (as by Corollary \ref{cor:analy+postive-def}
	$\nabla^{2}\mathbb{P}$ is positive definite). 
\end{proof}

\subsection{The Manhattan curve}\label{sec:manhattan}

The {\em Manhattan curve} of $\fbf$ is given by
\[
\mathcal C(\fbf)=\mathcal{C}(f,g)=\{(a,b)\in\mathbb{R}_{\geq 0}^{2}:\ P(-af-bg)=0,\ a+b>0\},
\]
and the {\em extended Manhattan curve of $\fbf$} is 
\[
\mathcal C^*(\fbf)=\{(a,b)\in \mathcal D(\fbf) :\ P(-af-bg)=0\}
\]
where $\mathcal D(\fbf)=-D(\fbf)$ and $D(\fbf)$ is defined in Equation \ref{eq:D(f)}.

\begin{remark}{\normalfont
We will focus primarily on the Manhattan curve $\mathcal C(\fbf)$, however it will be convenient at times to consider an open neighborhood of $\mathcal C(\fbf)$ in $\mathcal C^*(\fbf)$.}
\end{remark}

Before we list important properties of the Manhattan curve we recall that under our assumptions on the potentials $f$ and $g$ (see Assumptions \ref{rmk:potentials}), there exist $\delta_f,\delta_g>d(f)=d(g)>0$
which are the unique numbers satisfying Bowen's formula $P(-\delta_f f)=P(-\delta_g g)=0$ (cf. \cite[Cor 1.2]{BCKMcounting}). Recall that we are assuming that $f$ and $g$ are independent.
\begin{thm}
	[Bray-Canary-Kao-Martone, \cite{BCKMcounting} Thm C*] \label{thm:manhattan}Let
	$f,g:\Sigma^{+}\to\mathbb{R}$ be strictly positive locally H\"older continuous potentials with strong entropy gaps at infinity and such that $\|f-g\|_\infty$ is finite. Then \begin{enumerate}
		
		\item \label{thm:cri-mann}$(\delta_{f},0)$ and $(0,\delta_{g})$
		are on $\mathcal{C}({\bf f})$;
		
		\item $\mathcal{C}({\bf f})$ is a closed subcurve of the analytic
		curve $\mathcal C^*(\fbf)$;
            
        \item $\mathcal{C}^*({\bf f})$ is convex and strictly convex unless 
            \[
S_{n}f(\ul x)=\frac{\delta_g}{\delta_f}S_{n}g(\ul x)
\]
for all $\ul x\in \rm{Fix}^n.$ In particular, $\mathcal{C}^*({\bf f})$ is strictly convex provided $f$, and $g$
		are independent;
		
		\item \label{thm:slope-normal}The slope of the normal to $\mathcal{C}^*(\fbf)$
		at $(a,b)\in\mathcal{C}^*({\bf f})$ is 
		\[
		m=\frac{\int g\ d\mu_{-af-bg}}{\int f\ d\mu_{-af-bg}}
		\]
		where $\mu_{-af-bg}$ is the equilibrium state of the potential $-af-bg$. 
	\end{enumerate}
\end{thm}

By Theorem \ref{thm:manhattan}, we can parametrize
$\mathcal{C}^*({\bf f})$ by $(s,q(s))$ for $q=q(s)$ an analytic function and $s$ in an open interval containing $[0,\delta_f]$.  We define the set of {\em admissible slopes} as
\[
\mathcal{S}(\fbf):=\{m\colon m\text{ is the slope of the normal to $\mathcal{C}^*({\bf f})$
	for some point $(a,b)\in\mathcal{C}^*({\bf f})$}\}.
\]

\begin{notation}\label{not:ambm} When the potentials are independent, by the strict convexity of $\mathcal{C}^*({\bf f)}$, for any $m\in\mathcal{S}({\bf f)}$
there exists a unique point $(a_m,b_m)\in\mathcal{C}^*({\bf f)}$ such
that the slope of the normal at $(a_m,b_m)$ is $m$. Then, define
\[
t_{m}:=\left(\int f\ d\mu_{-a_{m}f-b_{m}g}\right)^{-1}\qquad {\bf m}:=(1,m)\qquad {\bf x}_{m}:={\bf m}/t_{m} \qquad\bf z_m:=(-a_m,-b_m).
\]
\end{notation}

\begin{cor}
	\label{cor:man-legen} The set of admissible slopes $\mathcal{S}(\bf f)$ is a non-empty open interval. Moreover, for any $m\in\mathcal S(\fbf)$ the following holds.
\begin{enumerate}
		\item $\mathbb{P}({\bf z}_{m})=0$, and in particular, $\bf z_m \in D$.
		
		\item 
		$
		\nabla\mathbb{P^{*}}({\bf x}_{m})={\bf z}_{m}\text{ and }\nabla\mathbb{P}({\bf z}_{m})={\bf x}_{m}.
		$
		
		\item $\mathbb{P}^{*}({\bf x}_{m})=-h_{\mu_{-a_{m}f-b_{m}g}}(\sigma)$.
		
		\item $\mathbb{P}\left(\nabla\mathbb{P^{*}}(\xbf_m)\right)=0$,
		thus $t=t_{m}$ realizes the maximum of $t\mapsto-t\mathbb{P}\left({\bf m}/t\right)$.
	\end{enumerate}
\end{cor}

\begin{proof} Theorem \ref{thm:manhattan}(2) implies that $\mathcal S(\fbf)$ is an open interval containing the interval with extremes given by the slopes of the normal at the point $(\delta_f,0)$ and $(0,\delta_g)$.
Item (1) follows from the definition of $\mathcal{C}^*(\fbf)$ and the observation (cf.~\cite[Lem 3.3 (2)]{BCKMcounting}) that $P(\langle \mathbf{z}_m, \mathbf{f} \rangle) = 0$ implies $d(-\langle \mathbf{z}_m, \mathbf{f} \rangle) < 1$.
	For (2), we notice that ${\bf x}_{m}=(1/t_{m},m/t_{m})=(\int f\ d\mu,\int g\ d\mu)$
	where $\mu=\mu_{-a_{m}f-b_{m}g}.$ Thus, by part (4) in Proposition \ref{prop:legen-trans},
	it is sufficient to show 
	\[
	\nabla\mathbb{P}(-a_{m},-b_{m})=\left(\int f\ d\mu,\int g\ d\mu\right).
	\]
	However, this follows from the formula for the derivative of the pressure
	given in Corollary \ref{cor:analy+postive-def}(2). The third assertion follows from
	Proposition \ref{prop:legen-trans} and the second assertion. The
	last assertion is an immediate consequence of assertions (1) and (2) and part (2) in Lemma
	\ref{lem:cov-legen}. 
\end{proof}

\section{The correlation number}\label{sec:first results} Let $\Sigma^+$ be a topologically mixing countable state Markov shift with BIP. Fix a pair $\fbf=(f,g)$ of strictly positive locally H\"older continuous independent potentials on $\Sigma^+$ with strong entropy gaps at infinity such that $\|f-g\|_\infty$ is finite.

\begin{definition}
	Let $m\in\mathcal S(\fbf)$ be an admissible slope for the Manhattan curve $\mathcal{C}(\fbf)$.
	The {\em correlation number} of ${\bf f}$
	and $m$ is
	\[
	H_{{\bf f}}(m)\colon=a_{m}+b_{m}\cdot m
	\]
	where $(a_{m},b_{m})\in \mathcal{C}(\fbf)$ is the unique point 
	such that the slope of the normal to $\mathcal{C}(\fbf)$ at $(a_{m},b_{m})$
	is $m$.
\end{definition}

\begin{rem}\label{rem:sus-flow}\
	{\normalfont\begin{enumerate}
		\item It is unclear at this point why we call $H_{{\bf f}}(m)$
		the correlation number. Theorem \ref{thmx:2} justifies this definition by showing that $H_\fbf(m)$ equals the exponential growth rate $\alpha_\fbf(m)$ of $M(t;\fbf,m,\xi)$. 
		
		\item There is a useful interpretation of $H_{{\bf f}}(m)$ via
		the suspension flow $\phi^f:\Sigma^{f}\to\Sigma^{f}$, which we do not
		use in this paper but is still worth mentioning. 
  Suppose $\Psi^g:\Sigma^{f}\to\mathbb{R}$
		is the symbolic reparametrization function of $g$ with respect to
		the suspension flow $\phi^f:\Sigma^{f}\to\Sigma^{f}$ (cf. \cite[Def 3.9]{Kao:2019tr}
		or \cite[Thm 10.3]{BCK-2019}). Then, for any $m\in\mathcal{S}({\bf f})$, one can check that 
		\[
		H_{\fbf}(m)=h_{\hat{\mu}}(\phi^f)
		\]
		where $\hat{\mu}$ is the equilibrium state of $-b_{m}\Psi^g$ and $h_{\hat{\mu}}(\phi^f)$
		is the measure-theoretic entropy of $\hat{\mu}$.		Indeed, observe that $P_{\phi^f}(-b_m\Psi^g)=a_m$, where $P_{\phi^f}$ is the pressure function with respect to the suspension flow, and thus
		\[
		a_m=P_{\phi^f}(-b_m\Psi^g)=h_{\hat{\mu}}(\phi^f)-b_m\int\Psi^g\ d\hat{\mu}.
		\]
		Then, recall that in this case $\hat{\mu}$ is the lift of $\mu=\mu_{-a_mf-b_mg}$
		(cf. \cite[Cor. 10.2]{BCK-2019}) and hence $\int\Psi^g\ d\hat{\mu}=\left(\int g\ d\mu\right)/\left(\int f\ d\mu\right)=m.$
	\end{enumerate}}
\end{rem}

\subsection{Properties of the correlation number} In this subsection, we establish some of the properties of $H_{\fbf}(m)$ including Corollary \ref{cor:corr-rigidity} and Theorem \ref{Thm:BSentropy} from the introduction.

\medskip
\noindent\textbf{Corollary \ref{cor:corr-rigidity}.}(Correlation number rigidity) {\em For $m_*=\delta_f/\delta_g$, we have
\begin{equation}
H_{\fbf}(m_*)\leq\delta_f.\label{eq:corr-rigidity}
\end{equation}
Moreover, the equality holds if and only if $m_*S_{n}f(\ul x)=S_{n}g(\ul x)$ for all $\ul x\in \rm{Fix}^n.$}

\begin{proof}
Since the Manhattan curve
$\mathcal{C}(\fbf)$ is convex and analytic, the point $(a,b)\in\mathcal{C}(\fbf)$ is below or on the secant passing through the points
$(\delta_f,0)$ and $(0,\delta_g)$.  Hence,
\[
a\delta_g+b\delta_f\leq\delta_f\delta_g
\]
 and the equality holds if and only if $\mathcal{C}(\fbf)$ is a straight
line. See Figure \ref{fig:Manhattan}.

\begin{figure}[h!]
    \centering \includegraphics[width=0.35\linewidth]{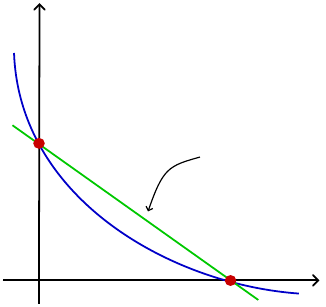}
    \put(-62,74){\small $x\delta_g+y\delta_f=\delta_f\delta_g$}
    \put(-50,19){\small $\delta_f$}
    \put(-140,80){\small $\delta_g$}
    \caption{The Manhattan curve and the secant between its intercepts with the axes.}
    \label{fig:Manhattan}
\end{figure}

We obtain Equation (\ref{eq:corr-rigidity}) by specializing at the point $(a_{m_*},b_{m_*})\in\mathcal{C}(\fbf)$ since
\begin{align*}
H_{\fbf}(m_*) & =a_{m_*}+m_*b_{m_*}\leq\frac{\delta_f\delta_g}{\delta_g}=\delta_f.
\end{align*}
The rigidity statement follows from Theorem \ref{thm:manhattan}(3) as $\mathcal{C}(\fbf)$
is a straight line segment if and only if 
\[
S_{n}f(\ul x)=\frac{\delta_g}{\delta_f}S_{n}g(\ul x)
\]
for all $\ul x\in \rm{Fix}^n.$
\end{proof}
Recall that for any $\alpha, \beta > 0$, the $(\alpha, \beta)$-Bishop-Steger entropy of ${\bf f}$ is defined as
\[
h_{(\alpha, \beta)}^{BS}({\bf f}) := \lim_{T \to \infty} \frac{1}{T} \log \sum_{n} \frac{1}{n} \#\left\{ \ul{x} \in \text{Fix}^{n}\colon \alpha S_{n} f(\ul{x}) + \beta S_{n} g(\ul{x}) \leq T \right\}.
\]

\smallskip
\noindent\textbf{Theorem \ref{Thm:BSentropy}.} {\em
Let $m \in \mathcal{S}(\fbf)$ be an admissible slope. Then
\[
\frac{H_{{\bf f}}(m)}{\alpha + m\beta} \leq h_{(\alpha, \beta)}^{BS}({\bf f}),
\]
and the equality holds if and only if $m$ is the unique slope satisfying $a_{m}/b_{m} = \alpha/\beta$.}
\begin{proof}
We ease notation by writing $h_{BS}$ for $h_{(\alpha, \beta)}^{BS}({\bf f})$. From \cite[Thm A]{BCKMcounting}, we know that $P(-h_{BS}(\alpha f + \beta g)) = 0$, which implies $(a_{BS}, b_{BS}) := (\alpha h_{BS}, \beta h_{BS}) \in \mathcal{C}({\bf f})$. Observe that $a_{BS}/b_{BS} = \alpha/\beta$ and denote by $m_{BS}$ the slope of the normal to $\mathcal{C}({\bf f})$ at $(a_{BS}, b_{BS})$. We claim that for any $m\in\cal S(\bf f)$,
\[
\frac{a_{m} + m b_{m}}{\alpha + m\beta} \leq \frac{a_{m_{BS}} + m_{BS} b_{m_{BS}}}{\alpha + m_{BS} \beta} = h_{BS}.
\]
The last equality follows directly from the definition of $h_{BS}$. To prove the inequality, note that  $\mathcal{C}({\bf f})$ can be parametrized as $(s, q(s))$ where $q(s)$ is an analytic function, and we denote the slope of the normal to $\mathcal{C}({\bf f})$ at $(s, q(s))$ by $m(s)$ so that $q'(s) = -1/m(s)$. Now consider the function
\[
F(s) = \frac{s + m(s) q(s)}{\alpha + m(s) \beta} = \frac{s - q(s)/q'(s)}{\alpha -\beta / q'(s)} = \frac{s q'(s) - q(s)}{\alpha q'(s) - \beta}.
\]
Differentiating, we get
\[
F'(s) = \frac{q''(s)\left(\alpha q(s) - \beta s\right)}{\left(\alpha q'(s) - \beta\right)^{2}}.
\]
Since $\mathcal{C}({\bf f})$ is strictly convex and real analytic, we know that $q''(s) > 0$. Moreover, $F'(0)>0$ while $F'(\delta_f)<0$, so $F(s)$ achieves its maximum when $\alpha q(s) = \beta s$, which implies $s / q(s) = \alpha / \beta$. By the strict convexity of $\mathcal{C}({\bf f})$, this solution is unique, occurring at $(a_{BS}, b_{BS}) = (\alpha h_{BS}, \beta h_{BS})$. This completes the proof of the claim and the theorem.
\end{proof}

We conclude this section by establishing technical properties of $H_\fbf(m)$ which will be useful for the estimates carried out in the next two sections.
\begin{lem}
	\label{lem:swarp_f_g}Let $f,g\colon\Sigma^{+}\to\bb R$ be strictly positive locally
	H\"older continuous potentials with strong entropy gaps at infinity such that $\|f-g\|_\infty$ is finite. Set $\fbf=(f,g)$.
	Then, for every $m\in\mathcal S(\fbf)$
	\begin{enumerate}
		
		\item $H_{\fbf}(m)=m\cdot H_{(g,f)}(\frac{1}{m})$.
		
		\item $H_{{\bf f}}(m)=-t_{m}\mathbb{P}^{*}(\xbf_m)$.
		
		\item $({\bf x}_{m})^{\top}\cdot\nabla^{2}\mathbb{P}^{*}({\bf x}_{m})\cdot({\bf x}_{m})>0$
		where $\mu=\mu_{-a_{m}f-b_{m}g}.$
	\end{enumerate}
\end{lem}
\begin{proof} For the first assertion, given $m\in\mathcal{S}(\fbf)$ consider
	$(a_{m},b_{m})\in\mathcal{C}(\fbf)$, so that $P(-a_{m}f-b_mg)=0$,
	$H_{\fbf}(m)=a_{m}+b_{m}m$, and 
	\[
	m=\frac{\int g\ d\mu_{-a_{m}f-b_{m}g}}{\int f\ d\mu_{-a_{m}f-b_{m}g}}.
	\]
	Similarly, for $m'\in\mathcal{S}(g,f)$, there is $(b_{m'},a_{m'})\in\mathcal{C}(g,f)$,
	that is $P(-b_{m'}g-a_{m'}f)=0$, and
	\[
	m'=\frac{\int f\ d\mu_{-a_{m'}f-b_{m'}g}}{\int g\ d\mu_{-a_{m'}f-b_{m'}g}}.
	\]
	Moreover, $
	H_{(g,f)}(m')=b_{m'}+a_{m'}m'.$
	By Theorem \ref{thm:manhattan}(4), we see that $m\in\mathcal{S}(\fbf)$ if and only if $1/m\in\mathcal{S}(g,f)$,
	and that when $m'=1/m$ we have $(a_{m'},b_{m'})=(b_m,a_m)$. Thus
	\[
	H_{(g,f)}\left(\frac{1}{m}\right)=b_m+a_{m}\frac{1}{m}=\frac{1}{m}(a_m+mb_m)=\frac{1}{m}H_{\fbf}(m).
	\]
	For item (2), by Corollary \ref{cor:man-legen}(2), we know
	$\nabla\mathbb{P}(-a_{m},-b_{m})={\bf x}_m$ and 
	\[
	\mathbb{P}^{*}({\bf x}_m)=-h_{-a_{m}f-b_{m}g}(\sigma).
	\]
	Hence, using the fact that $\mu_{-a_{m}f-b_{m}g}$ is the equilibrium state of $-a_{m}f-b_{m}g$, we compute
	\begin{alignat*}{1}
		-t_{m}\mathbb{P}^{*}({\bf x}_m) & =\frac{h_{-a_{m}f-b_{m}g}(\sigma)}{\int f\ d\mu_{-a_{m}f-b_{m}g}}\\
		& =\frac{P(-a_{m}f-b_{m}g)-\int\left(-a_{m}f-b_{m}g\right)\ d\mu_{-a_{m}f-b_{m}g}}{\int f\ d\mu_{-a_{m}f-b_{m}g}}\\
		& =a_{m}+b_{m}\frac{\int g\ d\mu_{-a_{m}f-b_{m}g}}{\int f\ d\mu_{-a_{m}f-b_{m}g}}=a_{m}+b_{m}m\\
		& =H_{{\bf f}}(m).
	\end{alignat*}
	
	The last assertion follows from Proposition \ref{prop:legen-trans}(4) as
	$\nabla\mathbb{P}^{*}({\bf x}_{m})=(-a_{m},-b_{m})$ and
	\[
	\nabla^{2}\mathbb{P}^{*}({\bf x}_{m})=\left(\nabla^{2}\mathbb{P}(-a_{m},-b_{m})\right)^{-1}=\frac{1}{\mathrm{det\nabla^{2}\mathbb{P}}}\left(\begin{array}{cc}
		\mathrm{Var}(g,\mu) & \mathrm{Cov}(f,g,\mu)\\
		\mathrm{Cov}(f,g,\mu) & \mathrm{Var}(f,\mu)
	\end{array}\right)^{-1}
	\]
	is positive definite (as by Corollary \ref{cor:analy+postive-def}(4)
	$\nabla^{2}\mathbb{P}$ is positive definite). 
\end{proof}

\begin{rem}\normalfont\label{rem:P''}
	Using the same suspension flow interpretation as in Remark \ref{rem:sus-flow}, one can relate the term in Lemma \ref{lem:swarp_f_g}(3) with the second derivative of the pressure over the suspension flow. Explicitly
	\[
\overline{\mathbb{P}}_{m}'':=t_{m}^{3}\cdot({\bf x}_{m})^\top\cdot\nabla^{2}\mathbb{P}^{*}({\bf x}_{m})\cdot {\bf x}_{m}=\left(\frac{d^{2}}{dz^{2}}\left.P_{\phi^f}\right|_{z=-b_{m}}(z\Psi^g)\right).
	\] 
 The convenient notation $\overline{\mathbb{P}}_{m}''$ will be used in Section \ref{sec:localgrowth}.
\end{rem}

\section{Preparing to count}\label{sec:prep}
Before presenting the setup and results, we give a brief outline of this section. 
The goal here is to convert the given orbital counting problem into estimates on thermodynamical quantities. In particular, Lalley's method counts periodic orbits by reducing this problem to studying preimages of sample points in the shift space.



To start the setup of this section, we fix a topologically mixing, countable state Markov shift $\Sigma^+$ with BIP and let $f,g\colon\Sigma^+\to \mathbb R$ be strictly positive locally H\"older continuous potentials with strong entropy gaps at infinity such that $\|f-g\|_\infty$ is finite. We let $\fbf=(f,g)$ and use the notations from Sections \ref{sec:background} and \ref{sec:first results}.

For an admissible slope $m\in\mathcal S(\fbf)$ and a precision $\xi>0$, we set $I_{\xi,m}^2(t)=(t,t+\xi)\times(mt,mt+\xi)$ and define
	\[
	M(n,t):=M(n,t;{\bf f},m,\xi)=\#\{\ul x\in\text{Fix}^{n}\colon S_{n}\fbf(\ul x)\in I_{\xi,m}^2(t)\}
	\]
	where $S_n\fbf(\ul x):=\left(S_nf(\ul x),S_ng(\ul x)\right)$. The goal of this section is to obtain some a priori estimates on $M(n,t)$ by relating it to the transfer operator $\cal L_{-a_mf-b_mg}$. Our strategy is similar to the one from Section 5 in  \cite{BCKMcounting}, with some key differences which we single out below.

\subsection{Preimages of sample points and a priori estimates}

Fix an admissible slope $m\in{\cal S(\fbf)}$, $k\in\bb Z_{>0}$, and a non-periodic
word $\ul z_{p}\in\Sigma^+$ in each $k$-cylinder $p\in\Lambda_{k}$. For every $n\in\mathbb Z_{>0}$ and $\xi>0$ define
\begin{gather*}
{\cal {W}}(n,p,t;\fbf,m,\xi)=p\cap\sigma^{-n}(\ul z_{p})\cap\{\ul x\in\Sigma^+\colon S_{n}\fbf(\ul x)\in\boxt\},\text{ and} \\
W(n,p,t)=W(n,p,t;\fbf,m,\xi):=\#\cal {W}(n,p,t;\fbf,m,\xi).
\end{gather*}
Propositions \ref{prop: epsilonk} and \ref{prop:n<k} below relate $M(n,t)$ to $W(n,p,t)$, which is advantageous because $W(n,p,t)$ is itself related to the transfer operator via the {\em Laplace-Fourier
transform} (defined below).

First, we introduce the following convenient notations. For any $\ybf\in\mathbb R^2$, define 
	\[
	M(n,\ybf):=M(n,{\bf y};\fbf)=\sum_{\ul x\in\text{Fix}^{n}}{\bf 1_{\{\ul x\colon S_{n}\fbf(\ul x)=\ybf\}}(\ul x)}
	\]
	and for any open bounded rectangle $U\subset \mathbb R^2$ let
	\[
	M(n,U)=\#\{\ul x\in\text{Fix}^{n}\colon S_{n}\fbf(\ul x)\in U\}		=\int_{U}M(n,\ybf)\ d\ybf.
	\]
	Note that with this notation $M(n,t)=M(n,\boxt)$.
Define $W(n,p,\ybf)$ and
$W(n,p,U)$ analogously and let 
\[
W(n,U)=\sum_{p\in\Lambda_{k}}W(n,p,U).
\]
Finally, for any open rectangle $U=(a,b)\times(c,d)$ and sufficiently small $\eee>0$,
consider the open rectangles
\[
U_{\pm\eee}:=(a\mp\eee,b\pm\eee)\times(c\mp\eee,d\pm\eee).
\]
\begin{prop}[{{\cite[Lemma 5.1 (ii)-(iv)]{BCKMcounting}}}]
	\label{prop: epsilonk} Suppose that $\Sigma^{+}$ is a topologically
	mixing, countable state Markov shift with BIP, $f,g\colon\Sigma^{+}\to\bb R$
	are strictly positive locally H\"older continuous potentials with strong entropy gaps at infinity such that $\|f-g\|_\infty$ is finite.
	\begin{enumerate}
		\item For any $p\in\Lambda_{k}$ and $n\geq k$ there exists a bijection
		$\Psi_{p}^{n}\colon\text{Fix}^{n}\cap p\to\sigma^{-n}(\ul z_{p})\cap p.$ 
		\item There exists a sequence $(\eee_{k})_{k=1}^{\infty}$ such that $\lim_{k}\eee_{k}=0$
		and if $\ul y\in\text{Fix}^{n}\cap p$ and $n\geq k$, then 
		\[
		\|S_{n}\fbf(\ul y)-S_{n}\fbf(\Psi_{p}^{n}(\ul y))\|\leq\eee_{k}.
		\]
		\item For all $n\geq k$ and all open rectangles $U\subset\bb R^{2}$
		\begin{equation}
			W(n,U_{-\eee_{k}})\leq M(n,U)\leq W(n,U_{\eee_{k}}).\label{eq:W and N-1}
		\end{equation}
	\end{enumerate}
\end{prop}

\begin{proof}
	Item (1) is \cite[Lemma 5.1 (ii)]{BCKMcounting}, item (2) is an immediate consequence of \cite[Lemma 5.1 (iii)]{BCKMcounting}, and item (3) follows
from item (2). 
\end{proof}
We will need to complement the results in Proposition \ref{prop: epsilonk}
with an exponential control of the quantities $\sum_{p\in\Lambda_{k}}W(n,p,t)$
and $M(n,t)$ for $k\in\bb Z_{>0}$ and $n<k$, similar to \cite[Lemma 5.1 (iv)]{BCKMcounting}.
However, in the current setting we will need a different approach which
relies on Lemmas \ref{lem:H>d} and \ref{lem:m>1} below.

\begin{lem}
	\label{lem:H>d} Let $f,g\colon\Sigma^{+}\to\bb R$ be strictly positive
	locally H\"older continuous potentials with strong entropy
	gaps at infinity such that $\|f-g\|_\infty$ is finite. If $m\geq1$ is an admissible slope, then 
	\[
	d(f)<H_{\fbf}(m).
	\]
\end{lem}

\begin{remark}\normalfont There exist potentials $f$ and $g$ for which every admissible slope is strictly less than 1. 
\end{remark} 

\begin{proof} By \cite[Lemma 3.3]{BCKMcounting}, it suffices to show that $P(-H_{\fbf}(m)f)<\infty$. Recall that, by definition, there exists a point $(a_{m},b_{m})$ on the Manhattan curve $\mathcal{C}(\fbf)$ such that $H_{\fbf}(m)=a_{m}+m\cdot b_{m}$.
	Since $m\geq1$ and the pressure function is increasing by Corollay \ref{cor:analy+postive-def}(2), we see that
	\[
	P(-H_{\fbf}(m)f)=P(-a_{m}f-m\cdot b_{m}f)\leq P(-a_{m}f-b_{m}f).
	\]
	Then, since $\|f-g\|_\infty=C<\infty$ and $P$ is increasing and linear on constants, we see that 
	\[
	P(-a_{m}f-b_{m}f)\leq P(-a_{m}f-b_{m}g+b_{m}C)=P(-a_{m}f-b_{m}g)+b_{m}C=b_{m}C<\infty
	\]
	where we used the fact that $P(-a_{m}f-b_{m}g)=0$ in the last equality.
\end{proof}

\begin{lem}
	\label{lem:m>1} Let $f,g\colon\Sigma^{+}\to\bb R$ be strictly positive
	locally H\"older continuous potentials with strong entropy gaps
	at infinity such that $\|f-g\|_\infty$ is finite. Suppose $m\geq1$. Then for all $k\in\mathbb{N}$
	and $b\in(d(f),H_{\fbf}(m))$, there exists $C=C(k,b,\xi)>0$ such that
	for all $n<k$ and $t>0$, we have
	\[
	\sum_{p\in\Lambda_{k}}W(n,p,t)\leq Ce^{bt}
    \qquad\text{ and }\qquad 
	M(k,t)\leq Ce^{bt}.
	\]
\end{lem}

\begin{rem}\normalfont
    In this proof we will use \cite[Lemma 5.1 (iv)]{BCKMcounting}. In that result, one is only interested in
	the case $b<\delta_{f}$. However, its proof holds more generally for any $b>d(f)$ (see also \cite[Lemma 3.1]{BCKMcounting}). 
\end{rem}

\begin{proof}
	We notice that, by \cite[Lemma 5.1 (iv)]{BCKMcounting}, for all $b>d(f)$ there exists $C=C(k,b,\xi)$ such that
	\begin{alignat*}{1}
		\sum_{p\in\Lambda_{k}}W(n,p,t) & \leq\sum_{p\in\Lambda_{k}}\sum_{\ul y\in\sigma^{-n}(\ul z_{p})}\mathbf{1}_{p}(\ul y)\mathbf{1}_{\{S_{n}f(\ul y)\leq t+\xi\}}(\ul y)\leq Ce^{bt}.
	\end{alignat*}
	As $m\geq1$,
	we have that $H_{\fbf}(m)>d(f)$ by Lemma \ref{lem:H>d}, and hence,
	we may choose any $b\in(d(f),H_{\fbf}(m)).$ The same argument shows that for all $b\in(d(f),H_{\fbf}(m))$ there exists $C=C(k,b,\xi)$ such that 
	\[
	\text{ }M(k,t)\leq Ce^{bt}.\qedhere
	\]
\end{proof}
We are now ready to bound $\sum_{p\in\Lambda_{k}}W(n,p,t)$
and $M(k,t)$ for any $k\in\bb Z_{>0}$ and $n<k$.
\begin{prop}[Negligible parts]\label{prop:n<k}Let $f,g\colon\Sigma^{+}\to\bb R$
	be strictly positive locally H\"older continuous potentials with
	strong entropy gaps at infinity such that $\|f-g\|_\infty$ is finite. Then
	for any $k>0$, there exist constants $\zeta=\zeta(m), C=C(k,m,\xi)>0$ such that for all $n<k$ 
	\[
	\sum_{p\in\Lambda_{k}}W(n,p,t)\leq Ce^{\left(H_{\fbf}(m)-\zeta\right)t}
    \qquad\text{ and }\qquad
	M(k,t)\leq Ce^{\left(H_{\fbf}(m)-\zeta\right)t}.
	\]
\end{prop}

\begin{proof}
	For $m\geq1$, a stronger version is given by Lemma \ref{lem:m>1}.
	Thus, we assume $m<1$ and consider $H_{(g,f)}(m')$ with $m'=1/m>1$.
	Applying Lemma \ref{lem:m>1} to the pair $(g,f)$, if we pick $\zeta$
	small enough such that $H_{(g,f)}(m')-m'\zeta\in(d(g),H_{(g,f)}(m'))$,
	then we know that there exists $C=C(k,m,\xi)$ such that
	\begin{align*}
		\sum_{p\in\Lambda_{k}}W(n,p,\widetilde{t};(g,f),m',\xi)=\#\left(p\cap\sigma^{-n}(\ul z_{p})\cap\{\ul x\colon\ S_{n}(g,f)(\ul x)\in I_{\xi,m'}^{2}(\wt t)\}\right)\leq Ce^{(H_{(g,f)}(m')-m'\zeta)\widetilde{t}}
	\end{align*}
	for all $n<k$ and $\wt t\geq0$. Therefore, setting $\widetilde{t}=mt$
	we have 
	\begin{alignat*}{1}
		\sum_{p\in\Lambda_{k}}W(n,p,t;\fbf,m,\xi) & =\#\left(p\cap\sigma^{-n}(\ul z_{p})\cap\{\ul x\colon S_{n}\fbf(\ul x)\in\boxt\}\right)\\
		& =\sum_{p\in\Lambda_{k}}W\left(n,p,mt;(g,f),\frac{1}{m},\xi\right)\\
		& \leq Ce^{(H_{(g,f)}(\frac{1}{m})-\frac{1}{m}\zeta)mt}\\
		& =Ce^{(mH_{(g,f)}(\frac{1}{m})-\zeta)t}\\
		& =Ce^{(H_{\fbf}(m)-\zeta)t}
	\end{alignat*}
	where the last equality follows from Lemma \ref{lem:swarp_f_g}(1). Applying
	the same argument to $M(k,t)$, we obtain 
	\[
	M(k,t)\leq Ce^{\left(H_{\fbf}(m)-\zeta\right)t}
	\]
	which ends the proof. 
\end{proof}

The {\em Laplace-Fourier transform} of a continuous function $F\colon\bb R^{2}\to\bb R$
is
\begin{equation}
\wh F(\zbf)=\int_{\bb R^{2}}e^{\inn{\zbf,\xbf}}F(\xbf)\ d{\bf x}.\label{eq:laplace-fourier}
\end{equation}
Write $\wh W(n,p,\zbf)$, $\wh W(n,\zbf)$, and $\wh M(n,\zbf)$ for
the Laplace-Fourier transforms of the functions $\xbf\mapsto W(n,p,\xbf)$,
$\xbf\mapsto W(n,\xbf)$, $\xbf\mapsto M(n,\xbf)$, respectively.
The key observation is that $\wh W(n,p,\zbf)$ is related to the transfer operator as follows:
\begin{align}
\widehat{W}(n,p,\mathbf{z}) & =\int_{\mathbb{R}^{2}}e^{\langle\mathbf{z},\mathbf{x}\rangle}\ W(n,p,\mathbf{x})\ d{\bf x}\nonumber \\
 & =\int_{\mathbb{R}^{2}}e^{\langle\mathbf{z},\mathbf{x}\rangle}\sum_{\underline{y}\in\sigma^{-n}(\ul z_{p})}{\bf 1}_{p}(\underline{y}){\bf 1}_{\{\underline{y}\colon S_{n}\mathbf{f}(\underline{y})=\mathbf{x}\}}(\underline{y})\ d\mathbf{x}\nonumber \\
 & =\sum_{\ul y\in\sigma^{-n}(\ul z_{p})}e^{\langle\mathbf{z},S_{n}\mathbf{f}(\underline{y})\rangle}{\bf 1}_{p}(\underline{y})\nonumber \\
 & =\left(\mathcal{L}_{\langle\mathbf{z},\mathbf{f}\rangle}^{n}{\bf 1}_{p}\right)(\ul z_{p}).\label{eq:w-hat}
\end{align}
Assume $\zbf\in D(\fbf)$ so that $\langle\mathbf{z},\mathbf{f}\rangle$ is a locally H\"older continuous potential
with finite pressure. By Proposition \ref{cor:gibbs and eq for zf}(1), $\mu_{\langle\mathbf{z},\mathbf{f}\rangle}$ is the unique invariant
Gibbs measure and equilibrium state for $\langle\mathbf{z},\mathbf{f}\rangle$ and we denote by $Q>1$ its Gibbs constant.

\begin{lem}
	\label{lem:n>k-hatW} Consider $k\in\mathbb N$ and $\bf z\in D(\fbf)$. For all $n\geq k,$ 
	and $p\in\Lambda_{k}$ we have that 
	\[
	\widehat{W}(n,p,\mathbf{z})\leq Q\mu_{\langle\mathbf{z},\mathbf{f}\rangle}(p)e^{n\mathbb{P}(\mathbf{z})}
	\]
	where $\mu_{\langle\mathbf{z},\mathbf{f}\rangle}$ is the equilibrium
	state for $\langle\mathbf{z},\mathbf{f}\rangle$ with Gibbs constant $Q$. 
\end{lem}
\begin{proof}
	Notice
	that 
	\begin{align*}
	e^{-n\mathbb{P}(\mathbf{z})}\cdot\widehat{W}(n,p,\mathbf{z})=e^{-n\mathbb{P}(\mathbf{z})}\left(\mathcal{L}_{\langle\mathbf{z},\mathbf{f}\rangle}^{n}{\bf 1}_{p}\right)(\ul z_{p}) & =\sum_{\underline{y}\in\sigma^{-n}(\ul z_{p})}e^{\langle\mathbf{z},S_{n}\mathbf{f}(\underline{y})\rangle-n\mathbb{P}(\mathbf{z})}{\bf 1}_{p}(\underline{y})\\
		& =\sum_{\underline{y}\in\sigma^{-n}(\ul z_{p})\cap p}e^{\langle\mathbf{z},S_{n}\mathbf{f}(\underline{y})\rangle-n\mathbb{P}(\mathbf{z})}.
	\end{align*}
	Since $n\geq k$
	we have $\sigma^{-n}(\ul z_{p})\cap p\neq\emptyset$. Then, for any $\ul y\in\sigma^{-n}(\ul z_{p})\cap p$, there exists an $n$-cylinder $p_{\ul y}$ such that $\ul y\in p_{\ul y}\subset p$. Moreover, $p_{\ul y_1}=p_{\ul y_2}$ implies $\ul y_1=\ul y_2$. Thus, we can repeatedly use the Gibbs property
 to see that
	\[
	\sum_{\underline{y}\in\sigma^{-n}(\ul z_{p})\cap p}e^{\langle\mathbf{z},S_{n}\mathbf{f}(\underline{y})\rangle-n\mathbb{P}(\mathbf{z})}\leq Q\sum_{\underline{y}\in\sigma^{-n}(\ul z_{p})\cap p}\mu_{\langle\mathbf{z},\mathbf{f}\rangle}(p_{\ul y})\leq Q\mu_{\langle\mathbf{z},\mathbf{f}\rangle}(p)
	\]
	which concludes the proof.
\end{proof}

Finally, we apply Lemma \ref{lem:n>k-hatW} to estimate the quantity
\[
W(n,U):=\sum_{p\in\Lambda_{k}}W(n,p,U).
\]
\begin{lem}
	\label{lem:prop1}Assume $\zbf\in D(\fbf)$, $U$ is an open bounded rectangle in $\mathbb R^2$, and $n\geq k$. Let $\xbf\in\mathbb R^2$ be such that $\langle\mathbf{z},\mathbf{y}\rangle\geq\langle\mathbf{z},\mathbf{x}\rangle$
	for all $\mathbf{y}\in U$. Then 
	\[
	W(n,U)\leq Qe^{n\mathbb{P}(\mathbf{z})-\langle\mathbf{z},\mathbf{x}\rangle}.
	\]
\end{lem}
\begin{proof}
	The statement follows from the proof of \cite[Prop.  1]{Lalley:1987df}. Indeed,
    \begin{align*}
		W(n,U) &\leq\sum_{p\in\Lambda_{k}}\int_{U}e^{\langle\mathbf{z},\mathbf{y}\rangle-\langle\mathbf{z},\mathbf{x}\rangle}W(n,p,\mathbf{y}) d\bf y \\&\leq\sum_{p\in\Lambda_{k}}e^{-\langle\mathbf{z},\mathbf{x}\rangle}\int_{\mathbb{R}^{2}}e^{\langle\mathbf{z},\mathbf{y}\rangle}W(n,p,\mathbf{y}) d\bf y \\
		& \leq\sum_{p\in\Lambda_{k}}\widehat{W}(n,p,\mathbf{z})e^{-\langle\mathbf{z},\mathbf{x}\rangle}\leq Qe^{n\mathbb{P}(\mathbf{z})-\langle\mathbf{z},\mathbf{x}\rangle}
	\end{align*}
    where we apply Lemma \ref{lem:n>k-hatW} to obtain the last inequality.
\end{proof}


\section{The global growth rate control}\label{sec:globalgrowth}

The first goal of this section is to establish Proposition~\ref{prop:a-priori estimate}, that is we find an {\em a priori} upper bound on the number of preimages of a fixed non-periodic word $\underline{z}_p$ in a cylinder $p$ for which the $n$-th ergodic sum of $\fbf$ is comparable to $t(1,m)$. It is important for our purposes that these upper bounds are global, in the sense that they do not depend on $p$.
The second goal is to establish Lemma~\ref{lem:a-priori estimate-2}, which is a type of large-deviation estimate. There we single out a set of preimages of the $\underline{z}_p$'s whose growth is relatively slow and thus negligible.

To begin the setup of this section, we fix a topologically mixing, countable state Markov shift $\Sigma^+$ with BIP and strictly positive locally H\"older continuous potentials $f,g\colon\Sigma^+\to\mathbb R$ with strong entropy gaps at infinity such that $\|f-g\|_\infty$ is finite. Set $\fbf=(f,g)$ and fix $\xi>0$. Recall that for an admissible slope $m\in\mathcal S(\fbf)$ in Notation \ref{not:ambm} we defined the quantities $(a_m,b_m)=-\zbf_m$, $t_m$, $\bf m$, $\xbf_m={\bf m}/t_m$, and that we define the correlation number as $H_\fbf(m)=a_m+b_mm$. Fix $k\in\mathbb N$, and for each $k$-cylinder $p\in\Lambda_k$ choose a non-periodic word $\ul z_{p}\in p$. 

In this section we obtain growth rate estimates which are independent on the choice of a cylinder.
 
\begin{prop}
	[A priori estimate]\label{prop:a-priori estimate} For any admissible slope $m\in\mathcal S(\fbf)$ there exists a
	constant $C=C(k,m,\xi)$ such that for all $p\in\Lambda_{k}$ we have 
	\[
	{\displaystyle \sum_{n\geq1}\sum_{\ul y\in\sigma^{-n}(\ul z_{p})}\mathbf{1}_{p}(\ul y)\mathbf{1}_{\{t\leq S_{n}f(\ul y)\leq t+\xi\}}(\ul y)\mathbf{1}_{\{mt\leq S_{n}g(\ul y)\leq mt+\xi\}}(\ul y)\leq Ce^{tH_{{\bf {\bf f}}}(m)}\mu_{-a_{m}f-b_{m}g}(p).}
	\]
	where $\mu_{-a_{m}f-b_{m}g}$ is the equilibrium state for $-a_{m}f-b_{m}g$. 
\end{prop}

Proposition \ref{prop:a-priori estimate} follows from Lemma \ref{lem:entropy bound}, which is a similar a priori estimate for one potential.

\begin{lemma}[Technical lemma]\label{lem:entropy bound} Suppose $f\colon\Sigma^+\to\mathbb R$ is a strictly positive, locally H\"older continuous potential with a strong entropy gap at infinity. Let $\delta_f>d(f)$ be the unique constant such that $P(-\delta_ff)=0$. Then, there exists a constant $D>0$ such that for all $p\in\Lambda_{k}$ 
	\begin{equation}
		\sum_{n\geq1}\sum_{\ul y\in\sigma^{-n}(\ul z_{p})}{\bf 1}_{p}(y){\bf 1}_{\{S_{n}f(\ul y)\le t\}}(\ul y)\le De^{t\delta}\mu_{-\delta f}(p)\label{eq:updated54}
	\end{equation} where $\mu_{-\delta f}$ is the equilibrium state of $-\delta f$.
\end{lemma}
\begin{proof} By \cite[Lem. 5.1 (iv)]{BCKMcounting} it is enough
	to consider $n>k$. Notice that by the Gibbs property of $\mu_{-\delta f}$,
	we know that
	\begin{alignat*}{1}
		\sum_{n>k}\sum_{\ul y\in\sigma^{-n}(\ul z_{p})}{\bf 1}_{p}(\ul y){\bf 1}_{\{S_{n}f(\ul y)\le t\}}(\ul y) & =\sum_{n>k}\sum_{\ul y\in\sigma^{-n}(\ul z_{p})}{\bf 1}_{p}(\ul y){\bf 1}_{\{S_{n-k}f(\sigma^{k}(\ul y))\le t-S_{k}f(\ul y)\}}(\ul y)\\
		& \leq\sum_{n>k}\sum_{\ul y\in\sigma^{-n}(\ul z_{p})}{\bf 1}_{p}(\ul y){\bf 1}_{\{S_{n-k}f(\sigma^{k}(\ul y))\le t+\frac{\log\mu_{-\delta f}(p)+\log Q}{\delta}\}}(\ul y)\\
		& \leq\sum_{n>k}\sum_{\ul w\in\sigma^{k-n}(\ul z_{p})}{\bf 1}_{\{S_{n-k}f(\ul w)\le t+\frac{\log\mu_{-\delta f}(p)+\log Q}{\delta}\}}(\ul w)\\
		& =\sum_{m\geq1}\sum_{\ul w\in\sigma^{-m}(\ul z_{p})}{\bf 1}_{\{S_{m}f(\ul w)\le t+\frac{\log\mu_{-\delta f}(p)+\log Q}{\delta}\}}(\ul w)\\
		& \leq \widetilde Ce^{\delta(t+\frac{\log\mu_{-\delta f}(p)+\log Q}{\delta})}=\widetilde CQe^{\delta t}\mu_{-\delta f}(p)
	\end{alignat*}
	where $Q$ is the Gibbs constant of $\mu_{-\delta_f}$ and the last inequality follows from \cite[Lem. 4.3]{BCKMcounting}. 
\end{proof}

\begin{proof}[Proof of Proposition \ref{prop:a-priori estimate}]
Let us simplify notation by writing $a=a_{m}$, $b=b_{m}$, $H=H_{{\bf f}}(m)=a+bm$, and $\mu=\mu_{-af-bg}$. We first claim that there exists $E$ such that
\[
{\displaystyle \sum_{n\geq1}\sum_{\ul y\in\sigma^{-n}(\ul z_{p})}\mathbf{1}_{p}(\ul y)\mathbf{1}_{\{S_{n}f(\ul y)\leq t\}}(\ul y)\mathbf{1}_{\{|S_{n}(g-mf)(\ul y)|\leq\xi\}}(\ul y)\leq Ce^{tH}\mu(p)}.
\]

Let $\varphi=\frac{1}{H}(af+bg)$ and observe that $\varphi$ is a strictly
positive locally H\"older continuous potential satisfying $P(-H\varphi)=0$. Using Lemma
\ref{lem:entropy bound} we know that there exists a constant
$D>0$ such that for all $p\in\Lambda_{k}$ 
\begin{equation}\label{eq:5.4-1}
	{\displaystyle \sum_{n\geq1}\sum_{\ul y\in\sigma^{-n}(\ul z_{p})}\mathbf{1}_{p}(\ul y)\mathbf{1}_{\{S_{n}\varphi(\ul y)\leq t\}}(\ul y)\leq De^{tH}\mu(p)}
\end{equation}
where we note that $\mu$ is the equilibrium state for $-H\varphi=-af-bg$. Since $a=H-bm$, we have 
\[
\varphi=\frac{1}{H}(af+bg)=f+\frac{b}{H}(g-mf).
\]
Hence, we know that for any $n$ and $\ul y\in\Sigma^{+}$
\begin{equation}
	\mathbf{1}_{\{S_{n}f(\ul y)\leq t\}}(\ul y)\mathbf{1}_{\{|S_{n}(g-mf)(\ul y)|\leq\xi\}}(\ul y)\leq \mathbf{1}_{\{S_{n}f(\ul y)+\frac{b}{H}S_{n}(g-mf)(\ul y)\leq t+\frac{b}{H}\xi\}}(\ul y)=\mathbf{1}_{\{S_{n}\varphi(\ul y)\leq t+\frac{b}{H}\xi\}}(\ul y).\label{eq:5.4-2}
\end{equation}
By equation (\ref{eq:5.4-1}) and (\ref{eq:5.4-2}), we finally obtain
\[
{\displaystyle \sum_{n\geq1}\sum_{\ul y\in\sigma^{-n}(\ul z_{p})}\mathbf{1}_{p}(\ul y)\mathbf{1}_{\{S_{n}f(\ul y)\leq t\}}(\ul y)\mathbf{1}_{\{|S_{n}(g-mf)(\ul y)|\leq\xi\}}(\ul y)\leq De^{\frac{b\xi}{H}}e^{tH}\mu(p)=:Ee^{tH}\mu(p)}
\]
which completes the proof of the claim. Then, we get
\begin{alignat*}{1}
	&\sum_{n\geq1}\sum_{\ul y\in\sigma^{-n}(\ul z_p)}\mathbf{1}_{p}(\ul y)\mathbf{1}_{\{t\leq S_{n}f(\ul y)\leq t+\xi\}}(\ul y)\mathbf{1}_{\{mt\leq S_{n}g(\ul y)\leq mt+\xi\}}(\ul y)\\
	 &\leq \sum_{n\geq1}\sum_{\ul y\in\sigma^{-n}(\ul z_p)}\mathbf{1}_{p}(\ul y)\mathbf{1}_{\{t\leq S_{n}f(\ul y)\leq t+\xi\}}(\ul y)\mathbf{1}_{\{|S_{n}g(\ul y)-mS_{n}f(\ul y)|\leq(m+1)\xi\}}(\ul y)\\
	&\leq \sum_{n\geq1}\sum_{\ul y\in\sigma^{-n}(\ul z_p)}\mathbf{1}_{p}(\ul y)\mathbf{1}_{\{S_{n}f(\ul y)\leq t+\xi\}}(\ul y)\mathbf{1}_{\{|S_{n}g(\ul y)-mS_{n}f(\ul y)|\leq(m+1)\xi\}}(\ul y)\\
	&\leq \underset{C}{\underbrace{De^{\left(\frac{b(m+1)\xi}{H}+H\xi\right)}}}e^{tH}\mu(p)
\end{alignat*}
which finishes the proof of the proposition.
\end{proof}

Recall the notation $I_{\xi,m}^{2}(t)=(t,t+\xi)\times(mt,mt+\xi)$ and consider the set 
\begin{gather*}
	\mathcal{W}(p,t):=\{\ul y\in\Sigma^{+}\colon\ul y\in p,\ \sigma^{n}(\ul y)=\ul z_{p},\ S_{n}\fbf(\ul y)\in I_{\xi,m}^{2}(t)\ \text{\text{for some }\ensuremath{n\geq1}}\}.
\end{gather*}
Since we assume that $\ul z_p$ is not periodic, if $\ul y\in\mathcal W(p,t)$ then there exists a unique positive integer $n(\ul y)$ such that $\sigma^{n(\ul y)}(\ul y)=\ul z_{p}$. Now, for any $\eee>0$ define
\begin{gather*}
	\mathcal{W}(p,t,\geq\epsilon):=\left\{ \ul y\in\mathcal{W}(p,t)\colon\left|\frac{n(\ul y)}{t}-t_{m}\right|\geq\epsilon\right\} \quad\text{ and }\quad\mathcal{W}(p,t,<\epsilon):=\left\{ \ul y\in\mathcal{W}(p,t)\colon\left|\frac{n(\ul y)}{t}-t_{m}\right|<\epsilon\right\} 
\end{gather*}
Observe that these sets are related to the quantities $W(n,p,t)$ introduced in Section \ref{sec:prep}, and in particular
\[
\#\mathcal W(p,t)=\sum_{n\geq 1}W(n,p,t).
\]
We now derive a type of large deviation result which shows that the contribution to the growth of $\mathcal{W}(p,t)$ by the elements $\ul y$ for which $n(\ul y)/t$ is far from $t_m$ is negligible. 

\begin{lem}
	[Negligible part]\label{lem:a-priori estimate-2} Suppose $m\in\mathcal S(\fbf)$ is an admissible slope for $\fbf$. For any $\epsilon>0$ there exist constants
	$C_{2}=C_2(k,m,\xi,\eee),\eta=\eta(k,m,\xi,\eee)>0$ such that 
	\[
	\sum_{p\in\Lambda_{k}}\mathcal{W}(p,t,\geq\eee)=\sum_{p\in\Lambda_{k}}{\displaystyle \sum_{n\colon\left|\frac{n}{t}-t_{m}\right|\geq\eee}W(n,p,t)\leq C_{2}e^{t\left(H_{\fbf}(m)-\eta\right)}}.
	\]
\end{lem}

\begin{proof}
    When $n<k$, the statement follows from a repeated application of  Proposition
	\ref{prop:n<k}. Thus, assume $n\geq k$ and let ${\bf m}=(1,m)$. By Lemma \ref{lem:cov-legen}(2) we know that the function
	$t\mapsto-t\mathbb{P}^{*}(\frac{\bf m}{t})$ is concave down 
    and realizes
	its maximum at $t=t_{m}$. Let $t^{1}=t_{m}+\eee$, $t^{0}=t_{m}-\eee$, and $\mathbf{z}^{i}=\nabla\mathbb{P^{*}}(\frac{\mathbf{m}}{t^{i}})$
	for $i=0,1$. By Proposition \ref{prop:legen-trans}(4), we have
	\[
	\mathbb{P}(\mathbf{z}^{1})<0=\mathbb P\left(\nabla\mathbb{P}^{*}\left(\frac{\mathbf{m}}{t_m}\right)\right)<\mathbb{P}(\mathbf{z}^{0})
	\]
	and
	$t^{i}\mathbb{P}(\mathbf{z}^{i})-\langle\mathbf{m},\mathbf{z}^{i}\rangle<H_{\fbf}(m)$ by Proposition \ref{prop:legen-trans}(6) and Lemma \ref{lem:swarp_f_g}(2).
    
	Consider $U_{t}=\left(1,1+\frac{\xi}{t}\right)\times\left(m,m+\frac{\xi}{t}\right)$, so that
	$t\cdot U_{t}=I_{\xi,m}^{2}(t)$. Recall that Lemma \ref{lem:prop1} shows
	that for all $n\ge k$, and for a fixed $\mathbf{z}\in D$, if $\mathbf{x}\in\mathbb{R}^{2}$ is such that $\langle\mathbf{z},\mathbf{y}\rangle\geq\langle\mathbf{z},\mathbf{x}\rangle$ for all $\mathbf{y}\in U_{t}$, then
	\[
	W(n,U_{t})\leq Qe^{n\mathbb{P}(\mathbf{z})-\langle\mathbf{z},\mathbf{x}\rangle}.
	\]
	Notice that the linear functional $\mathbf{y}\mapsto\langle\mathbf{z},\mathbf{y}\rangle$
	for $\mathbf{y}\in \overline{U}_{t}$ reaches its extreme values at a vertex
	of $\overline{U}_{t}$. 
    
    We let $\mathbf{x}^{i}(t)$ be the vertex of $\overline{U}_{t}$
	such that $\langle\mathbf{z}^{i},\mathbf{y}\rangle\geq\langle\mathbf{z}^{i},\mathbf{x}^{i}(t)\rangle$
	for all $\mathbf{y}\in U_{t}$. We also notice that $|\mathbf{x}^{i}(t)-\mathbf{m}|\leq\sqrt{2}\frac{\xi}{t}$.
	With this notation, we have
	\begin{align*}
		W(n,I_{\xi,m}^{2}(t)) & \leq Qe^{n\mathbb{P}(\mathbf{z}^{1})-t\langle\mathbf{z}^{1},\mathbf{x}^{1}(t)\rangle}
	\end{align*}    
	and given that $\mathbb{P}(\mathbf{z}^{1})<0$, we obtain
	\begin{align*}
		\sum_{n\geq tt^{1}}W(n,I_{\xi,m}^{2}(t)) & \leq\sum_{n\geq tt_{1}}Qe^{n\mathbb{P}(\mathbf{z}^{1})-t\langle\mathbf{z}^{1},\mathbf{x}^{1}(t)\rangle}\ \ \\
		& \leq Q\cdot\frac{e^{tt^{1}\mathbb{P}(\mathbf{z}^{1})-t\langle\mathbf{z}^{1},\mathbf{x}^{1}(t)\rangle}}{1-e^{\mathbb{P}(\mathbf{z}^{1})}}=Q'e^{t\left(t^{1}\mathbb{P}(\mathbf{z}^{1})-\langle\mathbf{z}^{1},\mathbf{x}^{1}(t)\rangle\right)}\\
		& \leq \overline Q'e^{t(H_{\fbf}(m)-\eta_{1})}.
	\end{align*}
	Similarly, 
	\begin{align*}
		\sum_{n\leq tt^{0}}W(n,I_{\xi,m}^{2}(t)) & \leq\sum_{n\leq tt^{0}}Qe^{n\mathbb{P}(\mathbf{z}^{0})-t\langle\mathbf{z}^{0},\mathbf{x}^{0}(t)\rangle}\\
		& \leq Q\cdot\frac{e^{\mathbb{P}(\mathbf{z}^{0})}\left(e^{tt^{0}\mathbb{P}(\mathbf{z}^{0})}-1\right)e^{-t\langle\mathbf{z}^{0},\mathbf{x}^{0}(t)\rangle}}{e^{\mathbb{P}(\mathbf{z}^{0})}-1}= Q''e^{t\left(t^{0}\mathbb{P}(\mathbf{z}^{0})-\langle\mathbf{z}^{0},\mathbf{x}^{0}(t)\rangle\right)}\\
		& \leq \overline Q''e^{t(H_{\fbf}(m)-\eta_{2})}.\qedhere
	\end{align*}
\end{proof}

\begin{lem}
\label{lem:rough-bounds} There exists a constant $C_3=C_3(k,m,\xi)>0$ independent
of $p$ such that 
\[
\sum_{p\in\Lambda_{k}}\sum_{n\geq1}\frac{1}{n}W(n,p,t)\leq C_3\frac{e^{H_\fbf(m)t}}{t}
\]
where $W(n,p,t)=\#\left(p\cap\sigma^{-n}(\ul z_{p})\cap\{\ul y\colon S_{n}\fbf(\ul y)\in I_{\xi,m}^{2}(t)\}\right)$.
\end{lem}

\begin{proof}
This follows from Proposition \ref{prop:a-priori estimate} and Lemma \ref{lem:a-priori estimate-2} via the following estimate
\begin{align*}
\sum_{p\in\Lambda_{k}}\sum_{n\geq1}\frac{1}{n}W(n,p,t) & = \sum_{p\in\Lambda_{k}}\sum_{\ul y\in\mathcal{W}(p,t)}\frac{1}{n(\ul y)}\\
 & = \sum_{p\in\Lambda_{k}}\sum_{y\in\mathcal{W}(p,t,\geq\eee)}\frac{1}{n(\ul y)} + \sum_{p\in\Lambda_{k}}\sum_{\ul y\in\mathcal{W}(p,t,<\eee)}\frac{1}{n(\ul y)}\\
 & \leq C_{2}e^{t\left(H_\fbf(m)-\eta\right)} + \sum_{p\in\Lambda_{k}}\sum_{\ul y\in\mathcal{W}(p,t,<\eee)}\frac{1}{t(t_{m}-\eee)}\\
 & \leq C_{2}e^{t\left(H_\fbf(m)-\eta\right)} + \frac{Ce^{tH_{{\bf {\bf f}}}(m)}}{t(t_m-\epsilon)}\sum_{p\in\Lambda_{k}}\mu_{-a_mf-b_mg}(p)\\
 & \leq  \frac{C_{3}e^{tH_{{\bf {\bf f}}}(m)}}{t},
\end{align*}
where \(C_{2}\) is the constant from Lemma~\ref{lem:a-priori estimate-2}, 
\(C\) is the constant from Proposition~\ref{prop:a-priori estimate}, 
and \(\frac{C}{t_{m}-\varepsilon}\) is uniformly bounded for $t$ sufficiently large (for instance, when \(t_{m}>2\varepsilon\)). 
Hence a constant \(C_{3}\) in the last inequality exists.
\end{proof}

\section{Local counting, asymptotic estimates, and Theorems \ref{thm:weaker-main}, \ref{thm:local-counting}, and \ref{thmx:2}}\label{sec:local-estimate}

We begin by outlining the structure of this section. 
Our goal is to combine the global estimates derived in Section~\ref{sec:globalgrowth} 
with the local asymptotic expansion formula in Theorem~\ref{thm:loc-est} 
to establish the main results of the paper, namely 
Theorems~\ref{thm:weaker-main}, \ref{thm:local-counting}, and~\ref{thmx:2}. 

Let $\Sigma^+$ be a topologically mixing, countable state Markov shift with BIP and let $f,g\colon\Sigma^+\to\mathbb R$ be strictly positive locally H\"older continuous potentials with strong entropy gaps at infinity such that $\|f-g\|_\infty$ is finite. Set $\fbf=(f,g)$. Fix an admissible slope $m\in\mathcal S(\fbf)$ and recall Notation \ref{not:ambm} for $(a_m,b_m)=-\zbf_m$, $t_{m}$, ${\bf m}$, and ${\bf x}_{m}={\bf m}/t_m$. Furthermore, as in Theorem \ref{transfer fact}, we let $h_{-a_m f - b_m g}$ and $\nu_{-a_m f - b_m g}$ denote the eigenfunction and eigenmeasure of the transfer operator $\mathcal{L}_{-a_m f - b_m g}$, respectively. Throughout this section, we fix $k\in\mathbb N$, a $k$-cylinder $p\in\Lambda_{k}$ and a non-periodic word $\ul z_p\in p$.

The main goal of this section is to establish the following local estimate. 

\begin{thm}[Local estimate]\label{thm:loc-est} Let $m\in\mathcal{S}({\bf f})$ and consider $(a_{m},b_{m})\in\mathcal{C}(\fbf)$. Then, as $t\to\infty$ 
	\[
	\sum_{n\geq1}\frac{1}{n}W(n,p,t)\sim\frac{e^{tH_{\mathbf{f}}(m)}}{t^{\frac{3}{2}}}\left(2\pi\overline{\mathbb{P}}''_{m}\right)^{-\frac{1}{2}}\cdot C_{p}(m)\cdot \int_{0}^{\xi}e^{a_{m}t}\ dt\cdot \int_{0}^{\xi}e^{b_{m}t}\ dt
	\]
	where $\overline{\mathbb{P}}''_{m}:=t_{m}^{3}\cdot({\bf x}_{m})^\top\cdot\nabla^{2}\mathbb{P}^{*}({\bf x}_{m})\cdot {\bf x}_{m}$ and $C_{p}(m)=h_{-a_{m}f-b_{m}g}(\ul z_{p})\nu_{-a_{m}f-b_{m}g}(p)$.
\end{thm}
\begin{remark*}
    The notation $\overline{\mathbb{P}}''_{m}$ is used to simpilify the right-hand side in Theorem \ref{thm:loc-est}. One can also see Remark \ref{rem:P''} for another interpretaion of $\overline{\mathbb{P}}''_{m}$ as the second derivative of the pressure over the suspension flow. 
\end{remark*}
\subsection{The proofs of Theorems \ref{thm:weaker-main}, \ref{thm:local-counting}, and \ref{thmx:2}}\label{sec:abc proofs}

Theorem \ref{thm:loc-est} is the last ingredient needed to establish Theorems \ref{thm:weaker-main}, \ref{thm:local-counting}, and \ref{thmx:2} from the introduction. In this subsection we record the proofs of these results while assuming Theorem \ref{thm:loc-est}.
\begin{thm}[Theorem \ref{thm:weaker-main} and Theorem \ref{thmx:2}] For any $m\in\mathcal{S}({\bf f})$, we have
	\[
	\lim_{t\to\infty}\frac{1}{t}\log M(t;{\bf f,m,\xi)}=H_{\fbf}(m).
	\]
\end{thm}
\begin{proof} Fix $k\in\mathbb N$. For $n\geq k$, we apply Proposition \ref{prop: epsilonk}(3) with $U=I^2_{\xi,m}(t)$ to deduce that for any $p\in \Lambda_k$
	\begin{equation}
	W(n,p,t;\fbf,m,\xi-\eee_k)\leq \#\mathcal{M}_p(n,t;{\bf f,m,\xi)}\leq W(n,p,t;\fbf,m,\xi+\eee_k) \label{eq:WMbounds1}
	\end{equation}
    where $(\eee_k)_k$ is the sequence from Proposition \ref{prop: epsilonk}(2). Then, we apply Proposition \ref{prop:n<k} to see that there exists $\zeta=\zeta(m)>0$ and $C_{1}=C_1(k,m,\xi)>0$ such that for any $n<k$
	\begin{equation}
	    \#\mathcal{M}(n,t;{\bf f},m,\xi)\leq C_{1}e^{t\left(H_{\fbf}(m)-\zeta\right)}.\label{eq:WMbounds2}
	\end{equation}
Moreover, Lemma \ref{lem:rough-bounds} and Theorem \ref{thm:loc-est} show that there exists $C_{2}=C_2(m,p,\xi)$ and $C_{3}=C_3(k,m,\xi)$ such that 
	\[
	\frac{e^{tH_{\fbf}(m)}}{t^{\frac{3}{2}}}(C_2+o(1))= \sum_{n\geq 1} \frac{1}{n} W(n,p,t)\leq \sum_{p\in\Lambda_k}\sum_{n\geq 1} \frac{1}{n} W(n,p,t)\leq C_{3}\frac{e^{tH_{\fbf}(m)}}{t}.
	\]
	Thus, recalling that $M(t;\fbf,m,\xi)=\sum_{p\in\Lambda_k}\sum_{n\geq 1}\frac{1}{n}\#\mathcal M_p(n,t;\fbf,m,\xi)$, we obtain
	\[
	\lim_{t\to\infty}\frac{1}{t}\log M(t;{\bf f,m,\xi)}=H_{\fbf}(m).\qedhere
	\]
\end{proof}

\begin{rem} \label{rem:obsticle} \normalfont
We expect the asymptotic growth rate of $\sum_{p\in\Lambda_k}\sum_{n\geq 1} \frac{1}{n} W(n,p,t)$ to be ${e^{tH_{\fbf}(m)}}/{t^{\frac{3}{2}}}$, up to a constant, as suggested by Theorem \ref{thm:local-counting}. The key missing step for establishing the global version of Theorem \ref{thm:local-counting} is to improve the current upper bound in Lemma \ref{lem:rough-bounds} from ${e^{tH_{\fbf}(m)}}/t$ to ${e^{tH_{\fbf}(m)}}/{t^{\frac{3}{2}}}$.

\end{rem}

\begin{proof}[Proof of  Theorem \ref{thm:local-counting}] 
	By Theorem \ref{thm:loc-est} we know 
	\[
	\sum_{n\geq1}\frac{1}{n}W(n,p,t;{\bf f,m,\xi)=\frac{e^{tH_{\fbf}(m)}}{t^{\frac{3}{2}}}C(m,p,\xi)(1+o(1))}
	\] where $C(m,p,\xi)=\left(2\pi\overline{\mathbb{P}}''_{m}\right)^{-\frac{1}{2}}\cdot C_{p}(m)\cdot\int_{0}^{\xi}e^{a_{m}t}\ dt\cdot\int_{0}^{\xi}e^{b_{m}t}\ dt.$
	Moreover, by equations (\ref{eq:WMbounds1}) and (\ref{eq:WMbounds2}) we get
	\[
	\frac{e^{tH_{\fbf}(m)}}{t^{\frac{3}{2}}}C(m,p,\xi-\eee_k)(1+o(1))\leq\sum_{n\geq1}\frac{1}{n}\#\mathcal{M}_p(n,t;{\bf f,m,\xi)\leq\frac{e^{tH_{\fbf}(m)}}{t^{\frac{3}{2}}}C(m,p,\xi+\eee_k)(1+o(1))}.
	\] 
	To conclude the proof recall that $\epsilon_k\to0$ as $k\to\infty$ by Proposition \ref{prop: epsilonk}(2).
\end{proof}

\subsection{Proof of Theorem \ref{thm:loc-est}}\label{sec:localgrowth}
This proof consists of three steps.
First, for any $\epsilon>0$ we break the sum into two parts: 
\[
\sum_{n\geq1}\frac{1}{n}W(n,p,t)=\sum_{n:\ \left|\frac{n}{t}-t_{m}\right|\ge\epsilon}\frac{1}{n}W(n,p,t)+\sum_{n:\ \left|\frac{n}{t}-t_{m}\right|<\epsilon}\frac{1}{n}W(n,p,t).
\]
By Lemma \ref{lem:a-priori estimate-2}, for any $\epsilon>0$
\[
\limsup_{t\to\infty}e^{-tH_{{\bf f}}(m)}\cdot t^{\frac{3}{2}}\left(\sum_{n:\ \left|\frac{n}{t}-t_{m}\right|\geq\epsilon}\frac{1}{n}W(n,p,t)\right)\leq\limsup_{t\to\infty}C_{2}t^{\frac{3}{2}}e^{-t\eta}=0.
\]
Therefore, for any $\epsilon>0$, it suffices to study the asymptotic behavior for $t\to\infty$ of
\[
\displaystyle\sum_{n:\ \left|\frac{n}{t}-t_{m}\right|<\epsilon}\frac{1}{n}W(n,p,t).
\]

Since the characteristic function of the rectangle $(0,\xi)\times(0,\xi)$ can be approximated by a nonnegative smooth function $u$
with compact support, it is enough to show that:
\begin{align}\label{eq:xx}
\lim_{\epsilon\to0}\lim_{t\to\infty}
e^{-tH_{{\bf f}}(m)}\, t^{\frac{3}{2}}
\left(
\sum_{n:\,\left|\frac{n}{t}-t_{m}\right|<\epsilon}
\frac{1}{n}
\int_{\mathbb{R}^{2}}
u({\bf y}-t{\bf m}) W(n,p,{\bf y})\, d{\bf y}
\right)
 \\
=
\left(2\pi\overline{\mathbb{P}}''_{m}\right)^{-\frac{1}{2}}
C_{p}(m)
\int_{\mathbb{R}^{2}}
u({\bf y})\, e^{-\langle \zbf_{m}, {\bf y} \rangle}\, d{\bf y}.
\nonumber
\end{align}

We will need Lemma \ref{lem:estimate} below, whose proof is postponed to Section \ref{ssec:lem estimate}. For $N\in\mathbb N$, denote by $C_c^N(\mathbb R^2)$ the space of $N$-differentiable functions on $\mathbb R^2$ with compact support.

\begin{lem}\label{lem:estimate} Let $K$ be a compact neighborhood of $\bf x_m$ in $\nabla\mathbb P(D)$.  Suppose ${\bf x=\nabla}\mathbb{P}({\bf z)}\in K$
and $u\text{\ensuremath{\in C_{c}^{N}(\mathbb{R}^{2})}}$ is nonnegative, then as $n\to\infty$
\[
\int u({\bf y-n{\bf x})W(n,p,\ybf)\ d \ybf\sim\frac{e^{-n\mathbb{P}^{*}({\bf x})}}{2\pi n}\left(\det\nabla\mathbb{P}^{*}({\bf x})\right)^{1/2}\cdot C_{p}(\zbf)\cdot\int_{\mathbb{R}^{2}}u({\bf y)}e^{-\langle{\bf z},{\bf y}\rangle}\ d{\bf y},}
\]
where $C_{p}(\zbf)=h_{\inn{\zbf,\fbf}}(\ul z_{p})\nu_{\inn{\zbf,\fbf}}(p)$.
Furthermore, the above convergence is uniform in $K$. 
\end{lem}

\begin{remark*}
Let us outline the strategy for obtaining equation~\eqref{eq:xx} assuming Lemma~\ref{lem:estimate}. 
The first step is to derive equation~\eqref{eq:***} below, which provides the asymptotic expansion for the left-hand side of (\ref{eq:xx}) for points near \({\bf x}_{m}\). 
Next, we expand the left-hand side of \eqref{eq:***}, denoted by \((*)\), and terms such as \(\overline{\mathbb{P}}''_{m}\) appearing in equation~\eqref{eq:xx} arise naturally. 
Using this expansion together with Lemma~\ref{lem:asymptotic}, a Gaussian-type distribution result, we can rewrite and bound the left-hand side of \eqref{eq:***}, i.e.\ \((*)\), in terms of \(\overline{\mathbb{P}}''_{m}\) and the other quantities in equation~\eqref{eq:xx}. 
Equation~\eqref{eq:xx} then follows by taking the limits of the upper and lower bounds for \((*)\).    
\end{remark*}

We now begin with the first step. We plan to apply Lemma~\ref{lem:estimate} to points 
\({\bf x}\) near \({\bf x}_{m} = {\bf m}/t_{m}\) (cf.\ Notation~\ref{not:ambm}), 
in particular to those of the form 
\({\bf x} = {\bf m}/(n/t)\) whenever \(|n/t - t_{m}| < \epsilon\) for small \(\epsilon\).
To do this, we need to ensure that these points lie in 
\(\nabla\mathbb{P}(D)\).
To see this, note first that \(D\) is an open set 
(by Proposition~\ref{cor:gibbs and eq for zf}(3)), and that 
\(\nabla\mathbb{P}\) is a homeomorphism on \(D\).  
The latter follows from the Inverse Function Theorem and because 
\(\nabla^{2}\mathbb{P}\) is positive definite (Corollary~\ref{cor:analy+postive-def}(4)).
Consequently, \(\nabla\mathbb{P}(D) \subset \mathbb{R}^{2}\) is open.

Since \({\bf x}_{m}\) lies in the open set \(\nabla\mathbb{P}(D)\), 
there exists \(\epsilon > 0\) sufficiently small such that whenever 
\(|n/t - t_{m}| < \epsilon\), the point \({\bf m}/(n/t)\) lies in a neighborhood 
of \({\bf x}_{m}\) contained in \(\nabla\mathbb{P}(D)\).
Thus there exists $
{\bf z}^{*} := \nabla\mathbb{P}^{*}\!\left(\frac{{\bf m}}{n/t}\right) \in D.
$

Then summing the estimate from Lemma \ref{lem:estimate} over the set $\{n:\left|\frac{n}{t}-t_{m}\right|<\epsilon\}$, for $t\to\infty$, we have
\begin{equation}
\sum_{\left|\frac{n}{t}-t_{m}\right|<\epsilon}\frac{1}{n}\int u({\bf y}-t{\bf m})W(n,p,{\bf y})\ d{\bf y}\sim\underset{:=~(*)}{\underbrace{\sum_{\left|\frac{n}{t}-t_{m}\right|<\epsilon}\frac{e^{-n\mathbb{P}^{*}(t{\bf m}/n)}}{2\pi n^{2}}\cdot\left(\det\nabla^{2}\mathbb{P}^{*}(t{\bf m}/n)\right)^{1/2}\cdot\mathbf{Q}_{p}\left(t{\bf m/n}\right)}}\label{eq:***}
\end{equation}
where $\mathbf{Q}_{p}\left(t{\bf m}/n\right)\colon=C_p(\bf z^*)\int_{\mathbb R^2} u(\ybf)e^{-\langle{\bf z}^*,\ybf\rangle}$. Now, the last step in the proof of Theorem \ref{thm:loc-est} consists of using a standard probability trick, Lemma \ref{lem:asymptotic} below, to simplify the
right hand side of the above equation. 

\begin{lem}[Lalley \cite{Lalley:1987df}, Lemma 6]
	\label{lem:asymptotic} Suppose $F$ is continuous on a closed interval
	$[t_{1},t_{2}]$ and $t_{1}<t_{m}<t_{2}$. Let $\sigma>0$. Then,
	for any $\eee>0$ such that $(t_m-\eee,t_m+\eee)\subset[t_{1},t_{2}]$,
	we have
	\begin{equation}
	\lim_{t\to\infty}\sum_{n\colon\left|\frac{n}{t}-t_m\right|<\eee}\frac{1}{t^{1/2}}F\left(\frac{n}{t}\right)e^{-\frac{t\left(\frac{n}{t}-t_m\right)^{2}}{2\sigma^{2}}}=F(t_m)(2\pi)^{1/2}\sigma.\label{eq:asymptotic lemma}
	\end{equation}
	Moreover, Equation (\ref{eq:asymptotic lemma}) holds uniformly for
	$t_m$ and $\sigma$ in a compact subset of $(0,\infty)$ and $F$
	in a compact subset of $C([t_{1},t_{2}])$ with $t_{1}<t_m<t_{2}$. 
\end{lem}
Recall that by Lemma \ref{lem:cov-legen}(2) and Corollary \ref{cor:man-legen}(4), $s\mapsto-s\mathbb{P^{*}}({\bf m}/s)$
is concave down for $s>0$ and achieves its maximum uniquely at $s=t_{m}$. Expanding
$-s\mathbb{P^{*}}({\bf m}/s)$ near $t_{m}$, we have
\begin{alignat*}{1}
	-\frac{n}{t}\mathbb{P}^{*}\left({\bf m}\frac{t}{n}\right)= & -t_{m}\mathbb{P}^{*}\left({\bf m}/t_{m}\right)-t_{m}^{-1}\left({\bf m}/t_{m})^\top\cdot\nabla^{2}\mathbb{P}^{*}({\bf m}/t_{m})\cdot({\bf m}/t_{m}\right)\cdot\frac{\left|\frac{n}{t}-t_{m}\right|^{2}}{2}+R(\zeta,m)\left|\frac{n}{t}-t_{m}\right|^{3}\\
	= & \underset{H_{{\bf f}}(m)}{\underbrace{-t_{m}\mathbb{P}^{*}({\bf x}_{m})}}-\underset{\text{main\ contribution}}{\underbrace{t_{m}^{-1}({\bf x}_{m})^\top\cdot\nabla^{2}\mathbb{P}^{*}({\bf x}_{m})\cdot{\bf x}_{m}\cdot\frac{\left|\frac{n}{t}-t_{m}\right|^{2}}{2}}}+\underset{\text{negligible term}}{\underbrace{R(\zeta,m)\left|\frac{n}{t}-t_{m}\right|^{3}}}
\end{alignat*}
where \(R(\zeta, m)\) is the remainder term from the Taylor expansion for some 
\(\zeta \in \left( \frac{n}{t},\, t_{m} \right)\). Since the function $s\mapsto-s\mathbb{P^{*}}({\bf m}/s)$
is analytic around $s=t_{m}$, there exists a constant $C_{m,\epsilon}>0$
such that $|R(\zeta,x)|\leq C_{m,\epsilon}$. 
Recalling that
\[
t_{m}^{-1}({\bf x}_{m})^\top\cdot\nabla^{2}\mathbb{P}^{*}({\bf x}_{m})\cdot{\bf x}_{m}=t_{m}^{-4}\overline{\mathbb{P}}''_{m},
\]
for $n$ such that $\left|\frac{n}{t}-t_{m}\right|<\epsilon$, we estimate
\begin{equation}
(*)\leq\sum_{\left|\frac{n}{t}-t_{m}\right|<\epsilon}\underset{=e^{tH_{{\bf f}}(m)}\cdot e^{-\frac{t\left|\frac{n}{t}-t_{m}\right|^{2}}{2\sigma_{+}^{2}}}}{\underbrace{e^{tH_{{\bf f}}(m)-tt_{m}^{-4}\overline{\mathbb{P}}''_{m}\cdot\frac{\left|\frac{n}{t}-t_{m}\right|^{2}}{2}+tC_{m,\epsilon}\epsilon\left|\frac{n}{t}-t_{m}\right|^{2}}}}\cdot\underset{=\frac{1}{t^{2}}F(n/t)}{\underbrace{\frac{1}{2\pi n^{2}}\cdot\left(\det\nabla^{2}\mathbb{P}^{*}(t{\bf m}/n)\right)^{1/2}\cdot\mathbf{Q}_{p}(t{\bf m}/n)}}\label{eq:6.4-11}
\end{equation}
\[
(*)\geq\sum_{\left|\frac{n}{t}-t_{m}\right|<\epsilon}\underset{=e^{tH_{{\bf f}}(m)}\cdot e^{-\frac{t\left|\frac{n}{t}-t_{m}\right|^{2}}{2\sigma_{-}^{2}}}}{\underbrace{e^{tH_{{\bf f}}(m)-tt_{m}^{-4}\overline{\mathbb{P}}''_{m}\cdot\frac{\left|\frac{n}{t}-t_{m}\right|^{2}}{2}-tC_{m,\epsilon}\epsilon\left|\frac{n}{t}-t_{m}\right|^{2}}}}\cdot\underset{=\frac{1}{t^{2}}F(n/t)}{\underbrace{\frac{1}{2\pi n^{2}}\cdot\left(\det\nabla^{2}\mathbb{P}^{*}(t{\bf {\bf m}}/n)\right)^{\frac{1}{2}}\cdot\mathbf{Q}_{p}(t{\bf m}/n)}}
\]
where $F(s)=\frac{1}{2\pi s^{2}}\left(\det\nabla^{2}\mathbb{P}^{*}({\bf m}/s)\right)^{\frac{1}{2}}\cdot\mathbf{Q}_{p}({\bf m}/s)$
and $\sigma_{\pm}=\left(t_{m}^{-4}\overline{\mathbb{P}}''_{m}\mp2\epsilon C_{m,\epsilon}\right)^{-\frac{1}{2}}$.

By Lemma \ref{lem:asymptotic}, we get
\begin{alignat}{1}
\lim_{t\to\infty}\frac{1}{F(t_{m})}{\displaystyle \sum_{\left|\frac{n}{t}-t_{m}\right|<\epsilon}t^{-1/2}e^{-\frac{t\left|\frac{n}{t}-t_{m}\right|^{2}}{2\sigma_{+}^{2}}}\cdot F(n/t)=\sqrt{2\pi}\sigma_{+}=\sqrt{2\pi}\left(t_{m}^{-4}\overline{\mathbb{P}}''_{m}-2\epsilon C_{m,\epsilon}\right)^{-\frac{1}{2}}}\label{eq:6.4-1}
\end{alignat}
and similarly 
\begin{alignat*}{1}
\lim_{t\to\infty}\frac{1}{F(t_{m})}{\displaystyle \sum_{\left|\frac{n}{t}-t_{m}\right|<\epsilon}t^{-1/2}e^{-\frac{t\left|\frac{n}{t}-t_{m}\right|^{2}}{2\sigma_{-}^{2}}}\cdot F(n/t)=\sqrt{2\pi}\sigma_{-}=\left(t_{m}^{-4}\overline{\mathbb{P}}''_{m}+2\epsilon C_{m,\epsilon}\right)^{-\frac{1}{2}}.}
\end{alignat*}

Finally, we put all the estimates above together and obtain, for any
$\epsilon$ small enough (depending on $m$),
\begin{align*}
 & \limsup_{t\to\infty}e^{-tH_{{\bf f}}(m)}\cdot t^{\frac{3}{2}}\left(\sum_{\left|\frac{n}{t}-t_{m}\right|\leq\epsilon}\frac{1}{n}\int u({\bf y}-t{\bf m})\cdot W(n,p,{\bf y)\ d{\bf y}}\right)\\
 & =\limsup_{t\to\infty}e^{-tH_{{\bf f}}(m)}\cdot t^{\frac{3}{2}}\left(*\right)\cdot (1+o(1)) & (\text{by~eq.~}(\ref{eq:***}))\\
 & \leq\lim_{t\to\infty}e^{-tH_{{\bf f}}(m)}\cdot t^{\frac{3}{2}}\cdot\left((1+o(1)) \cdot t^{-\frac{3}{2}}e^{tH_{{\bf f}}(m)}\left({\displaystyle \sum_{\left|\frac{n}{t}-t_{m}\right|<\epsilon}t^{-\frac{1}{2}}e^{-\frac{t\left|\frac{n}{t}-t_{m}\right|^{2}}{2\sigma_{+}^{2}}}\cdot F(n/t)}\right)\right) & (\text{by~eq.~}(\ref{eq:6.4-11}))\\
 &= F(t_{m})\sqrt{2\pi}\sigma_{+}. & (\text{by~eq.~}(\ref{eq:6.4-1}))
\end{align*}

Similarly, we have 
$$
 \liminf_{T\to\infty}e^{-tH_{{\bf f}}(m)}\cdot t^{\frac{3}{2}}\left(\sum_{\left|\frac{n}{t}-t_{m}\right|\leq\epsilon}\frac{1}{n}\int u({\bf y}-t{\bf m})\cdot W(n,p,\bf y)\ d{\bf y}\right)
	\geq F(t_{m})\sqrt{2\pi}\sigma_{-}.
$$
Then equation (\ref{eq:xx}) follows by taking the limit
\[
\lim_{\epsilon\to0}\sigma_{\pm}=\lim_{\epsilon\to0}\left(t_{m}^{-4}\overline{\mathbb{P}}''_{m}\mp2\epsilon C_{m,\epsilon}\right)^{-\frac{1}{2}}=t_{m}^{-2}\overline{\mathbb{P}}''_{m}
\]
which concludes the proof of Theorem \ref{thm:loc-est}.

\subsection{The proof of Lemma \ref{lem:estimate}}\label{ssec:lem estimate} In this subsection, we aim to prove the following key asymptotic estimate:

\medskip
\noindent\textbf{Lemma \ref{lem:estimate}.} {\em Let $K$ be a compact neighborhood of $\bf x_m$ in $\nabla\mathbb P(D)$.  Suppose ${\bf x=\nabla}\mathbb{P}({\bf z)}\in K$
and $u\text{\ensuremath{\in C_{c}^{N}(\mathbb{R}^{2})}}$ is nonnegative, then as $n\to\infty$
\[
\int u({\bf y-n{\bf x})W(n,p,\ybf)\ d \ybf\sim\frac{e^{-n\mathbb{P}^{*}({\bf x})}}{2\pi n}\left(\det\nabla\mathbb{P}^{*}({\bf x})\right)^{1/2}\cdot C_{p}(\zbf)\cdot\int_{\mathbb{R}^{2}}u({\bf y)}e^{-\langle{\bf z},{\bf y}\rangle}\ d{\bf y},}
\]
where $C_{p}(\zbf)=h_{\inn{\zbf,\fbf}}(\ul z_{p})\nu_{\inn{\zbf,\fbf}}(p)$.
Furthermore, the above convergence is uniform in $K$.}
\smallskip

In general, verifying Lemma \ref{lem:estimate} across the entire space $C_{c}^{N}(\mathbb{R}^{2})$ is challenging. However, as noted in Babillot-Ledrappier \cite[Lemma 2.4]{Babillot:1998kh}, it suffices to check Lemma \ref{lem:estimate} on a smaller, well-behaved function class $\mathcal{H}$. Specifically, the function class $\mathcal{H}$ is the linear span of functions of the form $\left\{ e^{i\langle {\bf z}, \cdot \rangle} h(\cdot) : h \in \mathcal{H}^{+}, {\bf z} \in \mathbb{R}^{2} \right\}$, where a function $h: \mathbb{R}^{2} \to \mathbb{C}$ belongs to the class $\mathcal{H}^{+}$ if:
\begin{enumerate}
    \item $h$ is real-valued and non-negative, and
    \item the Laplace--Fourier transform (see equation~\eqref{eq:laplace-fourier}) ${\bf \Theta} \mapsto \widehat{h}(-i{\bf \Theta})$ belongs to $C_{c}^{N}(\mathbb{R}^{2})$.
\end{enumerate}
\begin{remark*}
The statement of Babillot--Ledrappier \cite[Lemma~2.4]{Babillot:1998kh} is formulated 
in terms of convergence of measures on \(\mathbb{R}^{2}\). As explained immediately 
before \cite[Lemma~2.4]{Babillot:1998kh}, by considering the Borel measures 
\(\nu_{n}^{\mathbf{x}}\) on \(\mathbb{R}^{2}\), defined (via the Riesz representation theorem) by
\[
\nu_{n}^{\mathbf{x}}(G)
:= 
\frac{2\pi n\, e^{n\mathbb{P}^{*}(\mathbf{x})}}
     {C_{p}(\mathbf{z}) \left(\det\nabla\mathbb{P}^{*}(\mathbf{x})\right)^{1/2}}
\int_{\mathbb{R}^{2}} 
G(\mathbf{y}-n\mathbf{x})\, W(n,p,\mathbf{y}) \, d\mathbf{y},
\]
for any continuous compactly supported function \(G\) on \(\mathbb{R}^{2}\), and by noting that 
the limiting measure \(\lambda\) (in the notation of \cite[Lemma~2.4]{Babillot:1998kh}) is given by
\[
\lambda(G)
=
\int_{\mathbb{R}^{2}} 
G(\mathbf{y})\, e^{-\langle \mathbf{z}, \mathbf{y} \rangle} \, d\mathbf{y},
\]
the proof of \cite[Lemma~2.4]{Babillot:1998kh} applies verbatim in our setting.
\end{remark*}

The remainder of this section is devoted to deriving the following lemma:

\begin{lem}
\label{lem:lemma2.4_B-L} 
Let $K$ be a compact neighborhood of $\bf x_m$ in $\nabla\mathbb P(D)$.  Suppose ${\bf x=\nabla}\mathbb{P}({\bf z)}\in K$
and $u\in\cal H$, then as $n\to\infty$ 
\[
\int u({\bf y-n{\bf x})W(n,p,\ybf)\ d \ybf\sim\frac{e^{-n\mathbb{P}^{*}({\bf x})}}{2\pi n}\left(\det\nabla\mathbb{P}^{*}({\bf x})\right)^{1/2}\cdot C_{p}(\zbf)\cdot\int_{\mathbb{R}^{2}}u({\bf y)}e^{-\langle{\bf z},{\bf y}\rangle}\ d{\bf y},}
\]
where $C_{p}(\zbf)=h_{\inn{\zbf,\fbf}}(\ul z_{p})\nu_{\inn{\zbf,\fbf}}(p)$.
Furthermore, the above convergence is uniform in $K$.
\end{lem}

\begin{remark*}
The proof of Lemma~\ref{lem:lemma2.4_B-L} is given in the next subsection. 
Roughly speaking, the proof relies on a perturbation approach together with the saddle-point method. 
We also use the complex Ruelle--Perron--Frobenius theorem to bound and estimate the terms in the perturbation.
\end{remark*}

\subsubsection{The saddle-point method and the proof of Lemma~\ref{lem:lemma2.4_B-L}}
This subsection aims to prove Lemma~\ref{lem:lemma2.4_B-L}. 
Before beginning the proof, we first introduce one of our main tools: the saddle-point method.

Let $F:\Omega\colon=\{\|\mathbf{\Theta}\|<\epsilon\}\subset\Rb^{2}\to\Rb$ be an analytic function such that $\nabla F(\bf 0)=0$ and $\nabla^{2}F(\bf 0)$ is positive definite. As $F$ is analytic $\Omega\subset\mathbb{R}^{2}$, we can extend $F$ to a holomorphic function on $\Omega_{\mathbb{C}}=\{{\bf \Theta\in\mathbb{C}^{2}:\ }\|\mathbf{\Theta}\|<\epsilon\}$. Here we will abuse notation and continue to call the extended holomorphic function by $F:\Omega_{\mathbb{C}}\to\mathbb{C}$. 

To fix our notation, we write $M_{u}=\sup\{|u(x)|\}$, $v(t)=O_{F}(t)$
if and only if there exist $t_{0}>0$ and two constants $C_{1}$ and $C_{2}$
only depending on $F$ such that $C_{1}t\leq v(t)\leq C_{2}t$ for all $t>t_{0}$,
and $L_{v}$ is the Lipschitz constant of a Lipschitz function $v,$
i.e., $|v(z)-v(w)|\leq L{}_{v}|z-w|.$ 
\begin{prop}[Saddle-point method]
\label{prop:saddle-point} If $G\colon \mathbb{C}^{2}\to\mathbb{C}$ is
a Lipschitz function with compact support, then
\[
\int_\Omega G(i\mathbf{\Theta})e^{nF(i\mathbf{\Theta})}\ d\mathbf{\Theta}=\frac{e^{nF(\bf 0)}G(\bf 0)}{n}\cdot\frac{2\pi}{\sqrt{\det \nabla^{2}F(\bf 0)}}+\frac{e^{nF(\bf 0)}}{n}\left(L_{G}\cdot O_{F}\left(\frac{1}{\sqrt{n}}\right)+M_{G}\cdot O_{F}\left(\frac{1}{\sqrt{n}}\right)\right).
\]
\end{prop}
\begin{proof} We include a proof of this statement in Appendix \ref{app:saddle-point} because we were not able to find an appropriate reference. Our proof is adapted from \cite[Sec. 5.2]{Anantharaman:2000jq}.
\end{proof}

We establish Lemma \ref{lem:lemma2.4_B-L} in two steps. First, Lemma \ref{lem:cond-saddle-pt} allows us to use the saddle-point method to derive one side of the statement of Lemma \ref{lem:lemma2.4_B-L}. The second step is Lemma \ref{lem:relate-with-RPF}, which relies on the complex Ruelle-Perron-Frobenius theorem (i.e., Corollary \ref{thm:decay-cor}). 

\begin{lem}
\label{lem:cond-saddle-pt} Let $\bf z\in D$ and $\bf x=\nabla \mathbb P (\bf z)$. Suppose $u \in \mathcal{H}$ and define 
\[
F({\bf y}):=\langle{\bf x},{\bf y}-{\bf z}\rangle+\mathbb{P}({{\bf z}-{\bf y}})
\]
and 
\[
G_{p}({\bf y}):=\widehat{u}({\bf y}-{\bf z})C_{p}({\bf z}-{\bf y})
\]
where $C_{p}({\bf z}-{\bf y})=h_{\langle{\bf z}-{\bf y},{\bf f}\rangle}(\ul z_{p})\nu_{\inn{{\bf z}-{\bf y},\fbf}}(p)$.
Then there exists $\epsilon>0$ such that $F$ is analytic in $\Omega=\{\bf{\Theta}\colon\|{\bf\Theta}\|<\eee\}$,
$F(\bf 0)=-\mathbb{P}^{*}({\bf x})$, $\nabla F(\bf 0)=0$ and $\nabla^{2}F(\bf 0)=\nabla^{2}\mathbb{P}(\mathbf{z})=\left(\nabla^{2}\mathbb{P}^{*}(\mathbf{x})\right)^{-1}$
is positive definite. Moreover, $G_{p}$ is a
bounded Lipschitz continuous function in $\Omega_{\mathbb{C}}$. 
\end{lem}

\begin{proof}
The analyticity and the formulas for the derivatives of $F$ follow from Corollary \ref{cor:analy+postive-def}.
As $u\in\mathcal{H}$ we know $\widehat{u}\in C_{c}^{N}(\mathbb R^2)$. Corollary \ref{thm:decay-cor} gives us the analyticity of ${\bf y}\mapsto C_{p}({\bf z}-{\bf y})$
provided ${\bf y}\in\Omega_{\mathbb{C}}$. Hence, $G_{p}$ is bounded
Lipschitz in $\Omega_{\mathbb{C}}$.
\end{proof}
An immediate consequence of Proposition \ref{prop:saddle-point} and
Lemma \ref{lem:cond-saddle-pt} is:
\begin{cor}
\label{cor:apply-saddle-pt} Let $\Omega=\left\{ \|\mathbf{\Theta}\|<\epsilon\right\} \subset\Rb^{2}$, and let $F$ and $G_p$ be as in Lemma \ref{lem:cond-saddle-pt}. Then,
as $n\to\infty$, 
\begin{equation}
\int_{\Omega}G_{p}(i\mathbf{\Theta})e^{nF(i\mathbf{\Theta})}\ d\mathbf{\Theta}\sim e^{-n\mathbb{P}^{*}({\bf x})}\left(\frac{2\pi}{n}\right)(\det\nabla^{2}\mathbb{P}^{*}({\bf z}))^{\frac{1}{2}}\widehat{u}(-{\bf z})\cdot C_{p}({\bf z}).\label{eq:apply-saddle-pt}
\end{equation}
\end{cor}

The following lemma shows that the asymptotic growth of $\int_{\Rb^{2}}u({\bf y}-n{\bf x})W(n,p,{\bf y})\ d{\bf y}$
is captured by integrating the transfer operator around zero after applying the Laplace-Fourier transform. The proof is based on estimates
of the transfer operator. 
\begin{lem}
\label{lem:relate-with-RPF} Let $\bf z \in D$ and $\bf x=\nabla\mathbb{P}(\zbf)$. Suppose $u \in \mathcal{H}$, then as $n\to\infty$
\[
\int_{\Omega}\widehat{u}(i\mathbf{\Theta}-{\bf z})C_{p}({\bf z}-i\mathbf{\Theta})e^{n\langle{\bf x},i\mathbf{\Theta}-{\bf z}\rangle+n\mathbb{P}({\bf z}-i\mathbf{\Theta})}\ d\mathbf{\Theta}\sim(2\pi)^{2}\int_{\mathbb R^{2}}u({\bf y}-n{\bf x})W(n,p,{\bf y})\ d{\bf y}
\]
where $C_{p}({\bf z}-i\mathbf{\Theta})=h_{\langle{\bf z}-i\mathbf{\Theta},{\bf f}\rangle}(\ul z_{p})\cdot\nu_{\langle{\bf z}-i\mathbf{\Theta},{\bf f}\rangle}(p)$. Moreover, the convergence above is uniform for $\bf x\in\nabla\mathbb P(D)$ in any compact neighborhood of ${\bf x}$.
\end{lem}

\begin{proof} First, recall the Pancherel-Parseval's identity
\begin{equation}\label{eq:Parseval}
(2\pi)^2\int_{\bb R^2}U(\ybf)V(\ybf)d\ybf=\int_{\bb R^2}\widehat U(i\bf \Theta)\widehat V(-i\bf\Theta)d\bf\Theta.
\end{equation}
and that the Laplace-Fourier transform of $W(n,p,\bf y)$ is related to the transfer operator by Equation (\ref{eq:w-hat}). Now, we apply the Equation \ref{eq:Parseval} to the functions 
\[
U(\ybf)=u(\ybf-n\xbf)e^{-\langle\ybf,\zbf\rangle}\qquad\text{ and }\qquad V(\ybf)=W(n,p,\ybf)e^{\langle\ybf,\zbf\rangle}
\]
to get
\begin{align*}
(2\pi)^{2} & \int_{\mathbb R^{2}}u({\bf y}-n{\bf x})W(n,p,\bf y)d{\bf y}
=  \int_{\mathbb R^{2}}\widehat{u}(i\mathbf{\Theta}-{\bf z})e^{n\langle{\bf x},i{\bf\Theta}-{\bf z}\rangle}\widehat{W}(n,p,\mathbf{z}-i\mathbf{\Theta})\ d\mathbf{\Theta}\\
= & \int_{\mathbb R^{2}}\widehat{u}(i\mathbf{\Theta}-{\bf z})e^{n\langle{\bf x},i\bf\Theta-{\bf z}\rangle}\left(\mathcal{L}_{\langle{\bf z}-i\mathbf{\Theta},\mathbf{f}\rangle}^{n}{\bf 1}_{p}\right)(\ul z_{p})\ d\mathbf{\Theta}.
\end{align*}
By Corollary \ref{thm:decay-cor}(5), there exists $\delta>0$ such that
for any $\bf \Theta\neq \bf 0$
\begin{equation}
\lim_{n\to\infty}(1+\delta)^{n}e^{-n\mathbb{P}({\bf z})}\|\mathcal{L}_{\langle{\bf z}-i\mathbf{\Theta},\mathbf{f}\rangle}^{n}{\bf 1}_{p}\|_{\infty}=0\label{eq:CRPF-1}
\end{equation}
so we can find $n$ large enough
such that 
\[
\left\Vert e^{-n\mathbb{P}({\bf z})}(1+\delta)^{n}\left(\mathcal{L}_{\langle{\bf z}-i\mathbf{\Theta},\mathbf{f}\rangle}^{n}{\bf 1}_{p}\right)(\ul{z}_{p})\right\Vert \leq1.
\]
By Corollary \ref{thm:decay-cor}(4), we know that if $\bf\Theta$ is such that $\|\mathbf{\Theta}\|<\epsilon$,
there exist $R>0$ and $\eta\in(0,1)$ independent
of $\bf \Theta$ such that 
\begin{equation}
\left\Vert e^{-n\mathbb{P}({\bf z}-i\mathbf{\Theta})}\mathcal{L}_{\langle{\bf z}-i\mathbf{\Theta},\mathbf{f}\rangle}^{n}{\bf 1}_{p}-C_{p}({\bf z}-i\mathbf{\Theta})\right\Vert _{\infty}\leq R\eta^{n}.\label{eq:CRPF-2}
\end{equation}
Now, we split the integral as
\begin{align*}
 (2\pi)^{2}\int_{\mathbb R^{2}}u({\bf y}-n{\bf x})W(n,p,{\bf y})\ d{\bf y}
&=  \int_{\mathbb R^{2}}\underset{U(\mathbf{\Theta})}{\underbrace{\widehat{u}(i\mathbf{\Theta}-{\bf z})e^{n\langle{\bf x},i{\bf \Theta}-{\bf z}\rangle}\left(\mathcal{L}_{\langle{\bf z}-i\mathbf{\Theta},\mathbf{f}\rangle}^{n}{\bf 1}_{p}\right)(\ul z_{p})}}\ d\mathbf{\Theta}\\
&=  \int_{\Omega}U(\mathbf{\Theta})\ d\mathbf{\Theta}+\int_{\Omega^{c}}U(\mathbf{\Theta})\ d\mathbf{\Theta}.
\end{align*}

Recall that $\mathbb{P}^{*}({\bf x})=\langle{\bf x},{\bf z}\rangle-\mathbb{P}({\bf z})$ by Proposition \ref{prop:legen-trans}(4),
so we have
\begin{align}
\bigg|\int_{\Omega^c}U(\mathbf{\Theta})\ d\mathbf{\Theta}\bigg| & \leq\int_{K\cap\Omega^{c}}\widehat{u}(i{\bf \Theta}-\zbf)e^{n\left(-\langle{\bf x},{\bf z}\rangle+\mathbb{P}({\bf z})\right)}(1+\delta)^{-n}\left|e^{-n\mathbb{P}({\bf z})}(1+\delta)^{n}\left(\mathcal{L}_{\langle{\bf z}-i\mathbf{\Theta},\mathbf{f}\rangle}^{n}{\bf 1}_{p}\right)(\ul z_{p})\right|\ d\mathbf{\Theta}\nonumber \\
 & \leq M\int_{K\cap\Omega^{c}}e^{-n\mathbb{P}^{*}({\bf x})}(1+\delta)^{-n}\ d{\bf \Theta}\nonumber \\
 & =Me^{-n\mathbb{P}^{*}({\bf x})}(1+\delta)^{-n}\text{Leb}(K)\label{eq:small-outside-support}
\end{align}
where $\left\Vert \widehat{u}\right\Vert _{\infty}\leq M$ and $\widehat u$ vanishes for ${\bf \Theta}$ outside of $K$. On the other
hand, we can compute
\begin{align*}
\int_{\Omega}U(\mathbf{\Theta})\ d\mathbf{\Theta}= & \int_{\Omega}\widehat{u}(i\mathbf{\Theta}-{\bf z})e^{n\langle{\bf x},i{\bf \Theta}-{\bf z}\rangle}\left(\mathcal{L}_{\langle{\bf z}-i\mathbf{\Theta},\mathbf{f}\rangle}^{n}{\bf 1}_{p}\right)(\ul z_{p})\ d\mathbf{\Theta}\\
= & \underset{(*)}{\underbrace{\int_{\Omega}\widehat{u}(i\mathbf{\Theta}-{\bf z})e^{n\langle{\bf x},i{\bf \Theta}-{\bf z}\rangle+n\mathbb{P}({\bf z}-i\mathbf{\Theta})}C_{p}({\bf z}-i\mathbf{\Theta})\ d\mathbf{\Theta}}}\\
 & +\int_{\Omega}\widehat{u}(i\mathbf{\Theta}-{\bf z})e^{n\langle{\bf x},i{\bf \Theta}-{\bf z}\rangle}R({\bf z},{\bf f},\mathbf{\Theta})\ d\mathbf{\Theta}
\end{align*}
where $R({\bf z},{\bf f},\mathbf{\Theta})=\mathcal{L}_{\langle{\bf z}-i\mathbf{\Theta},\mathbf{f}\rangle}^{n}{\bf 1}_{p}(\ul z_p)-e^{n\mathbb{P}({\bf z}-i\mathbf{\Theta})}C_{p}({\bf z}-i\mathbf{\Theta})$.
Observe that by eq. $($\ref{eq:CRPF-2}$)$, we have

\begin{align}
\left|\int_{\Omega}\widehat{u}(i\mathbf{\Theta}-{\bf z})e^{n\langle{\bf x},i{\bf \Theta}-{\bf z}\rangle}R({\bf z},{\bf f},\mathbf{\Theta})\ d\mathbf{\Theta}\right| & \leq M\left|\int_{\Omega}e^{n\langle{\bf x},i{\bf \Theta}-{\bf z}\rangle}e^{n\bb P({\bf z}-i{\bf \Theta})}e^{-n\bb P({\bf z}-i{\bf \Theta})}R({\bf z},{\bf f},\mathbf{\Theta})\ d\mathbf{\Theta}\right|\label{eq:small-away-RPF}\\
 & \leq MR\eta^{n}\left|\int_{\Omega}e^{n\langle{\bf x},i{\bf \Theta}-{\bf z}\rangle}e^{n\bb P({\bf z}-i{\bf \Theta})}\ d\mathbf{\Theta}\right|\nonumber \\
 & \leq M'R\eta^{n}\int_{\Omega}e^{n\langle{\bf x},-{\bf z}\rangle}e^{n\bb P({\bf z})}\ d\mathbf{\Theta}\nonumber \\
 & \leq\pi\epsilon^{2}M'R\eta^{n}e^{-n\mathbb{P}^{*}({\bf x})}\nonumber 
\end{align}
where we used Corollary \ref{thm:decay-cor}(5) to obtain the third inequality. We notice that by Corollary \ref{cor:apply-saddle-pt} we know the
exponential growth rate of 
\[
(*):=\int_{\Omega}\widehat{u}(i\mathbf{\Theta}-{\bf z})e^{n\langle{\bf x},i{\bf \Theta}-{\bf z}\rangle+n\mathbb{P}({\bf z}-i\mathbf{\Theta})}C_{p}({\bf z}-i\mathbf{\Theta})\ d\mathbf{\Theta}
\]
is $\mathbb{P}^{*}({\bf x})$. Meanwhile, eq. (\ref{eq:small-outside-support})
and (\ref{eq:small-away-RPF}) show that the growth rate of these two
terms are strictly less then $\mathbb{P}^{*}({\bf x})$. Hence, 
\[
\lim_{n\to\infty}\frac{(2\pi)^{2}\int_{\mathbb R^{2}}u({\bf y}-n{\bf x})W(n,p,{\bf y})\ d{\bf y}}{\int_{\Omega}\widehat{u}(i\mathbf{\Theta}-{\bf z})e^{n\langle{\bf x},i{\bf \Theta}-{\bf z}\rangle+n\mathbb{P}({\bf z}-i\mathbf{\Theta})}C_{p}({\bf z}-i\mathbf{\Theta})\ d\mathbf{\Theta}}=1.\qedhere
\]
\end{proof}
We are now ready to present the proof of Lemma \ref{lem:lemma2.4_B-L}, which follows from Corollary \ref{cor:apply-saddle-pt} and Lemma \ref{lem:relate-with-RPF}.
\begin{proof}[Proof of Lemma \ref{lem:lemma2.4_B-L}]
By Lemma \ref{lem:relate-with-RPF}, we know

\[
\lim_{n\to\infty}\frac{(2\pi)^{2}\int_{\mathbb R^{2}}u({\bf y}-n{\bf x})W(n,p,{\bf y})\ d{\bf y}}{\int_{\Omega}\hat{u}(i\mathbf{\Theta}-{\bf z})e^{n\langle{\bf x},i{\bf \Theta}-{\bf z}\rangle+n\mathbb{P}({\bf z}-i\mathbf{\Theta})}C_{p}({\bf z}-i\mathbf{\Theta})\ d\mathbf{\Theta}}=1.
\]
Moreover, Corollary \ref{cor:apply-saddle-pt} shows that
\[
\lim_{n\to\infty}\frac{\frac{e^{-n\mathbb{P}^{*}({\bf x})}}{2\pi n}\left(\det\nabla^{2}\bb P^{*}(\xbf)\right)^{1/2}\wh u(-\zbf)C_{p}(\zbf)}{\int_{\Omega}\widehat{u}(i\mathbf{\Theta}-{\bf z})e^{n\langle{\bf x},i{\bf \Theta}-{\bf z}\rangle+n\mathbb{P}({\bf z}-i\mathbf{\Theta})}C_{p}({\bf z}-i\mathbf{\Theta})\ d\mathbf{\Theta}}=1.
\]
Combining these two equations, we obtain
\[
\lim_{n\to\infty}\frac{\frac{e^{-n\mathbb{P}^{*}({\bf x})}}{2\pi n}\left(\det\nabla^{2}\bb P^{*}(\xbf)\right)^{1/2}\wh u(-\zbf)C_{p}(\zbf)}{(2\pi)^{2}\int_{\mathbb R^{2}}u({\bf y}-n{\bf x})W(n,p,{\bf y})\ d{\bf y}}=1.\qedhere
\]
\end{proof}

\section{Application to cusped Hitchin representations: Theorem \ref{thm:main application}}\label{sec:cuspedHit}

In this section we recall results of Bray, Canary and the authors \cite{BCKMcounting} and combine them with Theorem \ref{thm:weaker-main} and Theorem \ref{thmx:2} to deduce Theorem \ref{thm:main application} from the introduction.

Let $S=\bb H^2/\Gamma$ be a geometrically finite hyperbolic surface with limit set $\Lambda(\Gamma)\subseteq \bb H^2$. Then, $\Gamma$ is a torsion-free geometrically finite Fuchsian group in $\sf{PSL}(2,\bb R)$ which we assume is not convex cocompact. We connect these objects to the previous sections by coding the recurrent portion of the geodesic flow of $S$.

\begin{thm}[Dal'Bo-Peign\'e \cite{Dalbo:1996vs}, Ledrappier-Sarig \cite{Ledrappier:2008wq}, Stadlbauer \cite{Stadlbauer:2004va}. See also Theorem 9.3 in \cite{BCKMcounting}]\label{thm:coding}
Let $\Gamma$ denote a torsion-free geometrically finite Fuchsian group which is not convex cocompact. There exists a topologically mixing countable state Markov shift $\Sigma^+$ with BIP which codes the recurrent portion of the geodesic flow on $T^1(\mathbb H^2/\Gamma)$. Moreover there exists a map $G\colon \mathit{Fix}^n\to\Gamma$ such that if $\gamma\in\Gamma$ is hyperbolic, then there exists $n\in\bb Z_{>0}$ and a unique (up to shift) $\ul x\in\text{Fix}^n$ such that $\gamma$ is conjugate to $G(\ul x)$.
\end{thm}

Note that depending on whether $S$ has finite area or not, the construction of the shift $\Sigma^+$ is different, however Theorem \ref{thm:coding} holds in both settings. See \cite[Section 9.4]{BCKMcounting} for a more detailed discussion. In particular, when $\Gamma$ is convex cocompact, Bowen and Series code the geodesic flow via a finite states Markov shift, and then the statement of Theorem \ref{thm:main application} follows from \cite{Lalley:1987df, Schwartz:1993ty,Sharp:1998if,Glorieux17,Dai,chow2024multiplecor}. 

\subsection{Cusped Hitchin representations} In order to define cusped Hitchin representations we will first need to discuss a notion of positivity on the space of complete flags.

A matrix $U\in\mathsf{SL}(d,\bb R)$ is {\em unipotent and totally positive} with respect to a basis $u=(u_1,\dots,u_d)$ of $\mathbb R^d$ if in the basis $u$, the matrix $U$ is unipotent, upper triangular and the minors of $U$ are strictly positive, unless they are forced to be zero due to the shape of the matrix. Recall that a (complete) flag $F$ in $\mathbb R^d$ is a maximal nested sequence of subspaces of $\mathbb R^d$. That is, $F=(F^0,F^1,\dots, F^{d-1},F^d)$ with $\dim F^i=i$ and $F^{i-1}\subset F^{i}$ for all $i=1,\dots,d$. We denote by $\mathcal F_d$ the space of flags in $\mathbb R^d$. Two flags $F,G$ are {\em transverse} if $F^i\cap G^{d-i}=\{0\}$ for all $i=0,\dots, d$. A basis $u=(u_1,\dots, u_d)$ is {\em consistent} with a pair of transverse flags $F,G$ if $u_i\in F^i\cap G^{d-i+1}$ for all $i=1,\dots, d$. Then, a $k$-tuple of flags $(F_1,\dots, F_k)$ is {\em positive} if there exist a basis $u=(u_1,\dots, u_d)$ consistent with $F_1$ and $F_k$ and unipotent and totally positive matrices $U_2,\dots, U_{k-1}$ such that $F_j=U_j\cdot \ldots\cdot U_2\cdot F_2$. 

Given a torsion-free geometrically finite Fuchsian group, a representation $\rho\colon \Gamma\to\mathsf{SL}(d,\mathbb R)$ is {\em cusped Hitchin} if there exists a continuous, $\rho$-equivariant map $\xi\colon \Lambda(\Gamma)\to\mathcal F_d$ which sends any positive tuple flags in $\Lambda(\Gamma)\subset \partial \mathbb H^2\cong \mathcal F_2$ to a positive tuple of flags in $\mathcal F_d$. 

\subsection{Length functions and potentials} Choose
\[
\mathfrak a=\{x\in\bb R^d\colon x_1+\dots+x_d=0\}\qquad\text{ and }\qquad\mathfrak a^+=\{x\in\mathfrak a\colon x_1\geq \dots\geq x_d\}
\] 
as a Cartan subspace and a closed positive Weyl chamber for the Lie algebra of $\mathsf{SL}(d,\mathbb R)$, respectively. The Jordan projection is the map
\[
\lambda\colon\mathsf{SL}(d,\mathbb R)\to\mathfrak a^+
\]
which records the logarithms of the moduli of eigenvalues of elements of $\mathsf{SL}(d,\mathbb R)$ in decreasing order. If $\gamma\in\Gamma$ is hyperbolic and $\rho$ is a cusped Hitchin representation, then $(\lambda\circ\rho)(\gamma)$ lies in the interior of $\mathfrak a^+$ (see for example \cite[Corollary 9.2]{Fock:2006da} and \cite[Theorem 1.4]{CZZcusped}).  Thus, if $\Delta$ is the set of nonzero positive linear combinations of the simple roots for $\mathfrak a^+$, for any $\phi\in\Delta$ we define the {\em $\phi$-length function of $\rho$} as
\[
\ell^\phi_\rho(\gamma)=(\phi\circ\lambda\circ\rho)(\gamma)\geq 0
\]
and observe that if $\gamma$ is hyperbolic, then $\ell^\phi_\rho(\gamma)>0$.

Theorem D in \cite{BCKMcounting} constructs a potential on the shift space from Theorem \ref{thm:coding} associated to a cusped Hitchin representation $\rho$ and a choice of length function.

\begin{thm}[Bray-Canary-Kao-Martone \cite{BCKMcounting}]\label{thm:Hitchinpotentials} Let $\Gamma$ be a torsion-free, geometrically finite Fuchsian group which is not convex cocompact, $\rho:\Gamma\to\mathsf{SL}(d,\mathbb R)$ is a cusped Hitchin representation 
and $\phi\in\Delta$. Then there exists a strictly positive locally H\"older continuous potential $\tau_\rho^\phi\colon\Sigma^+\to\mathbb R$ with strong entropy gap at infinity such that for every $\ul x\in\text{Fix}^n$
\[
S_n\tau^\phi_\rho(\ul x)=\ell^\phi_\rho(G(\ul x)).
\] 
Moreover, if $\eta\colon\Gamma\to\mathsf{SL}(d,\mathbb R)$ is another cusped Hitchin representation, then $\|\tau_\rho^\phi-\tau^\phi_\eta\|_\infty$ is finite.
\end{thm}
\begin{proof} The proof follows from Theorem D and Lemma 3.2 in \cite{BCKMcounting}.
\end{proof}

In this setting, we can characterize algebraically when a pair of potentials is independent.

\begin{lem}\label{lem:independence} If $\rho,\eta$ are cusped Hitchin representations and $\eta\neq \rho,(\rho^{-1})^\top$, then $\tau^\phi_\rho$ and $\tau^\phi_\eta$ are independent.
\end{lem}
\begin{proof} The Zariski closures of $\rho(\Gamma)$ and $\eta(\Gamma)$ are simple, center-free and connected Lie groups by a theorem of Sambarino (See Corollary 1.5 and Remark 6.5 in \cite{Sambarino:2024za}). Then, we can apply the proof of Lemma 2.11 in \cite{Dai} (see also \cite[Lemma 6.12]{Dai}) to obtain that the potentials $\tau^\phi_\rho$ and $\tau^\phi_\eta$ are independent.
\end{proof}

We are now ready to prove Thoerem \ref{thm:main application}.

\begin{proof}[Proof of Theorem \ref{thm:main application}]
Theorem \ref{thm:coding}, Theorem \ref{thm:Hitchinpotentials} and Lemma \ref{lem:independence} together imply that we can apply Theorems \ref{thm:weaker-main} and \ref{thmx:2} to the potentials $\tau_\rho^\phi$ and $\tau_\rho^\eta$. 
\end{proof}

\appendix

\section{Proof of Theorem \ref{thm:tran-holo}}\label{sec:Proof-of-holo}

\smallskip
\noindent\textbf{Theorem \ref{thm:tran-holo}.} {\em The map $(z,w)\mapsto\mathcal{L}_{zf+wg}$
is holomorphic in $\mathcal{B}(\mathcal{F}_{\beta}^{b}(\mathbb{C}))$
for all $(z,w)\in\overline{D}$.}
\smallskip

\begin{proof}
By \cite[Lemma 2.6.1]{Mauldin:2003dn}, we know $z\mapsto\mathcal{L}_{zf+wg}$
is holomorphic provided $z\mapsto\mathcal{L}_{zf+wg}$ is continuous.
Using Hartog's theorem for Banach spaces (cf. \cite[Thm 14.27]{Chae_Hartog_Thm}),
that is separate holomorphicity implies joint holomorphicity, it is
sufficient to show $z\mapsto\mathcal{L}_{zf+wg}$ is continuous for
any $(z,w)\in\overline{D}$, as the same argument shows that $w\mapsto\mathcal{L}_{zf+wg}$
is continuous for any $(z,w)\in\overline{D}$.

To see this, we fix $(z,w)\in\overline{D}$ and consider $z_{n}=z+\epsilon_{n}$
such that $\overline{D}\ni(z_{n},w)\to(z,w)$ as $n\to\infty$ in
$\mathbb{C}^{2}$, i.e., $\lim_{n\to\infty}|\epsilon_{n}|=0$. We
want to show that as $n\to\infty$
\[
\|\mathcal{L}_{z_{n}f+wg}-\mathcal{L}_{zf+wg}\|_{\rm {op}}\to0
\]
where $\|\cdot\|_{{\rm {op}}}$ is the operator norm on $\mathcal{B}(\mathcal{F}_{\beta}^{b}(\mathbb{C}))$.
The proof of this claim follows (with some minor modifications/simplifications) from Sarig \cite[Prop. 2 (3)]{Sarig06phase}. We give a proof below for the sake of completeness.

Let $\mathcal{A}_{x_{0}}=\{a\in\mathcal{A}:\ t_{ax_{0}}=1\}$. For
any $\phi\in\mathcal{F}_{\beta}^{b}(\mathbb{C})$, fix $\ul x$, $\ul y$
such that $x_{0}=y_{0}$. Write $U(\ul x)=\left(zf+wg\right)(\ul x)$, $G_{n}(\ul x)=1-e^{\epsilon_{n}f(\ul x)},$ and
\[
R_{n}\phi(\ul x):=\mathcal{L}_{zf+wg}\phi(\ul x)-\mathcal{L}_{z_{n}f+wg}\phi(\ul x)=\sum_{a\in\mathcal{A}_{x_{0}}}e^{U(a\ul x)}G_{n}(a\ul x)\phi(a\ul x).
\]
Hence,
\begin{align*}
|R_{n}\phi(\ul x)|\leq & \|\phi\|_{\mathcal{\beta}}\sum_{a\in\mathcal{A}_{x_{0}}}|e^{U(a\ul x)}||G_{n}(a\ul x)|\\
\leq & \|\phi\|_{\mathcal{\beta}}\sum_{a\in\mathcal{A}_{x_{0}}}e^{(\Re z\cdot f+\Re w\cdot g)(a\ul x)}|1-e^{\epsilon_{n}f(a\ul x)}|\\
\leq & \|\phi\|_{\mathcal{\beta}}\sum_{a\in\mathcal{A}_{x_{0}}}e^{(\Re z\cdot f+\Re w\cdot g)(a\ul x)+|\epsilon_{n}|f(a\ul x)}|e^{-|\epsilon_{n}|f(a\ul x)}-e^{(\epsilon_{n}-|\epsilon_{n}|)f(a\ul x)}|\\
= & \|\phi\|_{\mathcal{\beta}}\sum_{a\in\mathcal{A}_{x_{0}}}e^{\left((\Re z+|\epsilon_{n}|)\cdot f+\Re w\cdot g\right)(a\ul x)}|e^{-|\epsilon_{n}|f(a\ul x)}-e^{(\epsilon_{n}-|\epsilon_{n}|)f(a\ul x)}|.
\end{align*}
Since $\lim_{n\to\infty}|\epsilon_{n}|=0$, we know $(\Re z+|\epsilon_{n}|,\Re w)\in D$
when $n$ is large enough. Thus, there exists $0<\epsilon<<1$ such
that $(\Re z+\epsilon,\Re w)\in D$ and $|\epsilon_{n}|<\epsilon$
when $n$ is large. Therefore, by the Gibbs property of $(\Re z+|\epsilon_{n}|)\cdot f+\Re w\cdot g$,
we have 
\[
e^{\left((\Re z+|\epsilon_{n}|)\cdot f+\Re w\cdot g\right)(a\ul x)}\leq e^{\left((\Re z+\epsilon)\cdot f+\Re w\cdot g\right)(a\ul x)}\leq Q\cdot\mu_{(\Re z+\epsilon)\cdot f+\Re w\cdot g}([a])
\]
where $[a]$ is the cylinder generated by $a\in\mathcal A_{x_0}$. Set $F_{n}(\ul x):=e^{-|\epsilon_{n}|f(\ul x)}-e^{(\epsilon_{n}-|\epsilon_{n}|)f(\ul x)}.$
Since $f\geq0$, we can easily get $|F_{n}|\leq2$ and $\lim_{n\to\infty}|F_{n}(\ul x)|=0$.
Hence
\begin{alignat*}{1}
|R_{n}\phi(\ul x)|/\|\phi\|_{\beta} & \leq\sum_{a\in\mathcal{A}_{x_{0}}}Q\cdot\mu_{(\Re z+\epsilon)\cdot f+\Re w\cdot g}([a])|F_{n}(a\ul x)|\\
 & \leq Q\int_{\mathcal A_{x_0}}|F_{n}(\ul x)|\ d\mu_{(\Re z+\epsilon)\cdot f+\Re w\cdot g}(\ul x)<\infty.
\end{alignat*}
By the Bounded Convergence Theorem, we have as $n\to\infty$
\[
\sup\left\{ \frac{|R_{n}\phi(\ul x)|}{\|\phi\|_{\beta}}:\ \phi\in\mathcal{F}_{\beta}^{b}(\mathbb{C})\right\} \to0.
\]
The second step is to show $\sup\left\{ \frac{\mathrm{Lip}(R_{n}\phi(\ul x))}{\|\phi\|_{\beta}}:\ \phi\in\mathcal{F}_{\beta}^{b}(\mathbb{C})\right\} \to0$
as $n\to\infty$. 
\begin{align*}
|R_{n}\phi(\ul x)-R_{n}\phi(\ul y)|\leq & \sum_{a\in\mathcal{A}_{x_{0}}}|e^{U(a\ul x)}G_{n}(a\ul x)\phi(a\ul x)-e^{U(a\ul y)}G_{n}(a\ul y)\phi(a\ul y)|\\
\leq & \sum_{a\in\mathcal{A}_{x_{0}}}\left|e^{U(a\ul x)}\right|\cdot\left|\left(1-e^{U(a\ul y)-U(a\ul x)}\right)\right|\cdot\left|G_{n}(a\ul x)\right|\cdot\left|\phi(a\ul x)\right|\\
 & +\sum_{a\in\mathcal{A}_{x_{0}}}\left|e^{U(a\ul y)}\right|\cdot\left|G_{n}(a\ul x)-G_{n}(a\ul y)\right|\cdot\left|\phi(a\ul x)\right|\\
 & +\sum_{a\in\mathcal{A}_{x_{0}}}\left|e^{U(a\ul y)}\right|\cdot\left|G_{n}(a\ul y)\right|\cdot\left|\phi(a\ul x)-\phi(a\ul y)\right|.
\end{align*}
Notice that if $U(a\ul x)\neq U(a\ul y)$ then 
\begin{align*}
\left|\left(1-e^{U(a\ul y)-U(a\ul x)}\right)\right|\leq & \frac{\left|\left(1-e^{U(a\ul y)-U(a\ul x)}\right)\right|}{\left|U(a\ul y)-U(a\ul x)\right|}\left(\mathrm{Lip}(U)\right)d(a\ul x,a\ul y)^{\beta}<Kd(\ul x,\ul y)^{\beta}
\end{align*}
where, for example, $K=\max\left\{ \mathrm{Lip}(U)\cdot\sup\left\{ \frac{\left|1-e^{\delta}\right|}{\delta}:\ |\delta|\leq \mathrm{Lip}(U)\right\} ,2\right\} $.
Hence
\begin{align*}
\sum_{a\in\mathcal{A}_{x_{0}}}\left|e^{U(a\ul x)}\right|\cdot\left|\left(1-e^{U(a\ul y)-U(a\ul x)}\right)\right|\cdot\left|G_{n}(a\ul x)\right|\cdot\left|\phi(a\ul x)\right|\leq & Kd(\ul x,\ul y)^{\beta}\|\phi\|_{\mathcal{\beta}}\sum_{a\in\mathcal{A}_{x_{0}}}\left|e^{U(a\ul x)}\right|\cdot\left|G_{n}(a\ul x)\right|,
\end{align*}
\[
\sum_{a\in\mathcal{A}_{x_{0}}}\left|e^{U(a\ul y)}\right|\cdot\left|G_{n}(a\ul y)\right|\cdot\left|\phi(a\ul x)-\phi(a\ul y)\right|\leq Ae^{-1}d(\ul x,\ul y)^{\beta}\|\phi\|_{\mathcal{\beta}}\sum_{a\in\mathcal{A}_{x_{0}}}\left|e^{U(a\ul y)}\right|\cdot\left|G_{n}(a\ul y)\right|
\]
and 
\begin{align*}
\sum_{a\in\mathcal{A}_{x_{0}}}\left|e^{U(a\ul y)}\right|\cdot\left|G_{n}(a\ul x)-G_{n}(a\ul y)\right|\cdot\left|\phi(a\ul x)\right| & \leq\|\phi\|_{\mathcal{\beta}}\sum_{a\in\mathcal{A}_{x_{0}}}\left|e^{U(a\ul y)}\right|\left|e^{\epsilon_{n}f(a\ul x)}-e^{\epsilon_{n}f(a\ul y)}\right|\\
 & \leq\|\phi\|_{\mathcal{\beta}}\sum_{a\in\mathcal{A}_{x_{0}}}e^{\Re U(a\ul y)}\left|e^{\epsilon_{n}\left(f(a\ul x)\right)}\right|\left|1-e^{\epsilon_{n}\left(f(a\ul y)-f(a\ul x)\right)}\right|\\
 & \leq\|\phi\|_{\mathcal{\beta}}\left(\epsilon_{n}K'\right)d(\ul x,\ul y)^{\beta}\sum_{a\in\mathcal{A}_{x_{0}}}e^{\Re U(a\ul y)}\left|e^{\epsilon_{n}\left(f(a\ul x)\right)}\right|\\
 & \leq\epsilon_{n}K'\|\phi\|_{\mathcal{\beta}}d(\ul x,\ul y)^{\beta}\sum_{a\in\mathcal{A}_{x_{0}}}e^{\left((\Re z+\epsilon)\cdot f+\Re w\cdot g)\right)(a\ul x)}\\
 & \leq\epsilon_{n}K'\|\phi\|_{\mathcal{\beta}}d(\ul x,\ul y)^{\beta}\left(Q\int d\mu_{(\Re z+\epsilon)\cdot f+\Re w\cdot g)}\right)\\
 & \leq\epsilon_{n}\cdot K'Q\|\phi\|_{\mathcal{\beta}}d(\ul x,\ul y)^{\beta}
\end{align*}
where, for example, $K'=\max\left\{\mathrm{ Lip}(f)\cdot\sup\left\{ \frac{\left|1-e^{\delta}\right|}{\delta}:\ |\delta|\leq \mathrm{Lip}(f)\right\} ,2\right\} $.

Hence
\[
|R_{n}\phi(\ul x)-R_{n}\phi(\ul y)|\leq(Ae^{-1}+K)\|\phi\|_{\mathcal{\beta}}\sum_{a\in\mathcal{A}_{x_{0}}}\left|e^{U(a\ul x)}\right|\cdot\left|G_{n}(a\ul x)\right|+\epsilon_{n}\cdot K'Q\|\phi\|_{\beta}.
\]
As we have proved in the first step: $\text{\ensuremath{\sum_{a\in\mathcal{A}_{x_{0}}}\left|e^{U(a\ul x)}\right|\cdot\left|G_{n}(a\ul x)\right|}}\to0$
as $n\to\infty$, we have established the claim. 
\end{proof}	

\section{Proof of the Saddle-point method}\label{app:saddle-point}
Let $F:\Omega=\{\bf \Theta\colon \|\bf \Theta\|<\eee\}\subset\mathbb{R}^{2}\to\mathbb{R}$ be an analytic
function such that $\nabla F(\bf 0)=0$ and $\nabla^{2}F(\bf 0)$ is positive definite. By analyticity, we know $F$ has an absolutely convergent Taylor expansion $F(\mathbf{x})=F(\bf 0)+\nabla^{2}F(\bf 0)\langle\mathbf{x},\mathbf{x}\rangle+R(\mathbf{x})$
for $\mathbf{x}\in \Omega$. Thus, $F$ is analytic over $\Omega_{\bb C}=\{\bf \Theta\in\mathbb{C}^{2}:\|\bf \Theta\|<\epsilon\}$
with the same Taylor expansion. We will abuse notations and let $F$ denote the corresponding analytic function over $\mathbb{C}^{2}$. 

To fix our notation, we write $M_{u}=\sup\{|u(x)|\}$, $v(t)=O_{F}(t)$
if and only if there exists $t_{0}>0$ and two constants $C_{1}$and $C_{2}$
only depending on $F$ such that $C_{1}t\leq v(t)\leq C_{2}t$ for $t>t_{0}$,
and $L_{v}$ is the Lipschitz constant of a Lipschitz function $v,$
i.e., $|v(z)-v(w)|\leq L{}_{v}|z-w|.$
\begin{lem}
	Let $G:\mathbb{C}^{2}\to\mathbb{C}$ be a Lipschitz function with
	compact support. Suppose $F:\Omega\subset\mathbb{R}^{2}\to\mathbb{R}$
	is an analytic function such that $\nabla F(\bf 0)=0$ and $\nabla^{2}F(\bf 0)$
	is positive definite. Then 
	\[
	\int_{\Omega}G(i\bf \Theta)e^{nF(i\bf \Theta)}\ d\bf \Theta=\frac{e^{nF(\bf 0)}G(\bf 0)}{n}\cdot\frac{2\pi}{\sqrt{\det \nabla^{2}F(\bf 0)}}+\frac{e^{nF(\bf 0)}}{n}\left(L_G\cdot O_{G}\left(\frac{1}{\sqrt{n}}\right)+M_{G}\cdot O_{F}\left(\frac{1}{\sqrt{n}}\right)\right).
	\]
\end{lem}

\begin{proof}
	Let $\widetilde {\bf \Theta}=\sqrt{n}\bf \Theta\in\mathbb{R}^{2}$, we
	have 
	\[
	\int_{\Omega}G(i\bf \Theta)e^{nF(i\bf \Theta)}\ d\bf \Theta=\frac{e^{nF(\bf 0)}}{n}\int_{\{\|\widetilde {\bf \Theta}\|<\sqrt{n}\epsilon\}}G\left(\frac{i\widetilde {\bf \Theta}}{\sqrt{n}}\right)e^{n\left(F\left(\frac{i\widetilde {\bf \Theta}}{\sqrt{n}}\right)-F(\bf 0)\right)}\ d\widetilde {\bf \Theta}.
	\]
	Since $R(\bf x)=F(\bf x)-F(\bf 0)-\nabla^2F(\bf 0)(\xbf,\xbf )/2$, we have
	\[
	\int_{\Omega}G(i\bf \Theta)e^{nF(i\bf \Theta)}\ d\bf \Theta=\frac{e^{nF(\bf 0)}}{n}\int_{\{\|\widetilde {\bf \Theta}\|<\sqrt{n}\epsilon\}}G\left(\frac{i\widetilde {\bf \Theta}}{\sqrt{n}}\right)e^{-\frac{1}{2}\nabla^2F(\bf 0)(\widetilde {\bf \Theta},\widetilde {\bf \Theta})+nR\left(\frac{i\widetilde {\bf \Theta}}{\sqrt{n}}\right)}\ d\widetilde {\bf \Theta}
	\]
    and, for $\|\widetilde {\bf \Theta}\|<\sqrt{n}\epsilon$, there exists $K>0$ such that
    \begin{align*}
		\left|nR\left(\frac{i\widetilde {\bf \Theta}}{\sqrt{n}}\right)\right| \leq\frac{K}{\sqrt{n}}\|\widetilde {\bf \Theta}\|^{3} \leq K\epsilon\|\widetilde {\bf \Theta}\|^{2}.
	\end{align*}

	Since $\left|e^{z}-1-z\right|\leq e^{|z|}\frac{|z|^{2}}{2}$, we can compute
	\begin{align*}
		\left|\int_{\{\|\widetilde {\bf \Theta}\|<\sqrt{n}\epsilon\}}\left|G\left(\frac{i\widetilde {\bf \Theta}}{\sqrt{n}}\right)e^{-\frac{1}{2}\nabla^2F(\bf 0)(\widetilde {\bf \Theta},\widetilde {\bf \Theta})}\left(e^{nR\left(\frac{i\widetilde {\bf \Theta}}{\sqrt{n}}\right)}-1-nR\left(\frac{i\widetilde {\bf \Theta}}{\sqrt{n}}\right)\right)\right|d\widetilde {\bf \Theta}\right|\\&\hspace{-8cm}\leq\frac{M_G}{2}\left|\int_{\{\|\widetilde {\bf \Theta}\|<\sqrt{n}\epsilon\}}e^{-\frac{1}{2}\nabla^2F(\bf 0)(\widetilde {\bf \Theta},\widetilde {\bf \Theta})}e^{\left|nR\left(\frac{i\widetilde {\bf \Theta}}{\sqrt{n}}\right)\right|}\left|nR\left(\frac{i\widetilde {\bf \Theta}}{\sqrt{n}}\right)\right|^{2}d\widetilde {\bf \Theta}\right|\\
        &\hspace{-8cm}\leq\frac{M_GK^{2}}{2n}\left|\int_{\{\|\widetilde {\bf \Theta}\|<\sqrt{n}\epsilon\}}e^{-\frac{1}{2}\nabla^2F(\bf 0)(\widetilde {\bf \Theta},\widetilde {\bf \Theta})+K\epsilon\|\widetilde {\bf \Theta}\|^{2}}\|\widetilde {\bf \Theta}\|^{6}d\widetilde {\bf \Theta}\right|\\
		&\hspace{-8cm}\leq C_1(F,\epsilon)\cdot\frac{M_GK^{2}}{2n}
	\end{align*}
	where $C_{1}(F,\epsilon):=\int_{\bb R^2}e^{-\frac{1}{2}\nabla^2F(\bf 0)(\widetilde {\bf \Theta},\widetilde {\bf \Theta})+K\epsilon\|\widetilde {\bf \Theta}\|^{2}}\|\widetilde {\bf \Theta}\|^{6}d\widetilde {\bf \Theta}$
	is finite when $\epsilon$ is small. So we get 
	\begin{align*}
	\int_{\{\|\bf \Theta\|<\epsilon\}}G(i\bf \Theta)e^{nF(i\bf \Theta)}\ d\bf \Theta&\\&\hspace{-3cm}=\frac{e^{nF(\bf 0)}}{n}\left(\int_{\{\|\widetilde {\bf \Theta}\|<\sqrt{n}\epsilon\}}G\left(\frac{i\widetilde {\bf \Theta}}{\sqrt{n}}\right)e^{-\frac{1}{2}\nabla^2F(\bf 0)(\widetilde {\bf \Theta},\widetilde {\bf \Theta})}\left(1+nR\left(\frac{i\widetilde {\bf \Theta}}{\sqrt{n}}\right)\right)\ d\widetilde {\bf \Theta}+M_GO_{F}\left(\frac{1}{n}\right)\right).
	\end{align*}
	Since $G$ is Lipschitz with compact support, we have 
	\begin{align*}
		\left|\int_{\{\|\widetilde {\bf \Theta}\|<\sqrt{n}\epsilon\}}\left|G\left(\frac{i\widetilde {\bf \Theta}}{\sqrt{n}}\right)-G(\bf0)\right|e^{-\frac{1}{2}\nabla^2F(\bf 0)(\widetilde {\bf \Theta},\widetilde {\bf \Theta})}\ d\widetilde {\bf \Theta}\right| & \leq\frac{L_{G}}{\sqrt{n}}\int_{\mathbb R^{2}}\|\widetilde {\bf \Theta}\|e^{-\frac{1}{2}\nabla^2F(\bf 0)(\widetilde {\bf \Theta},\widetilde {\bf \Theta})}\ d\widetilde {\bf \Theta}=\colon\frac{L_{G}}{\sqrt{n}}\cdot C_{2}(F)
	\end{align*}
	and observe that $C_{2}(F)<\infty.$ We also have 
	\begin{align*}
		\left|\int_{\{\|\widetilde {\bf \Theta}\|<\sqrt{n}\epsilon\}}\left|G\left(\frac{i\widetilde {\bf \Theta}}{\sqrt{n}}\right)\cdot nR\left(\frac{i\widetilde {\bf \Theta}}{\sqrt{n}}\right)\right|e^{-\frac{1}{2}\nabla^2F(\bf 0)(\widetilde {\bf \Theta},\widetilde {\bf \Theta})
        }\ d\widetilde {\bf \Theta}\right|&\\ &\hspace{-3cm} \leq M_{G}\int_{\mathbb R^{2}}\frac{K\|\widetilde {\bf \Theta}\|^{3}}{\sqrt{n}}e^{-\frac{1}{2}\nabla^2F(\bf 0)(\widetilde {\bf \Theta},\widetilde {\bf \Theta})}\ d\widetilde {\bf \Theta}=: \frac{KM_{G}}{\sqrt{n}}\cdot C_{3}(F)
	\end{align*}
	and observe $C_{3}(F)<\infty.$
	Therefore, summing up all the estimates we have 
	\begin{align*}
		\int_{\{\|\bf \Theta\|<\epsilon\}}G(i\bf \Theta)e^{nF(i\bf \Theta)}\ d\bf \Theta &\\
        &\hspace{-3cm}=\frac{e^{nF(\bf 0)}G(\bf0)}{n}\left(\int_{\{\|\widetilde {\bf \Theta}\|<\sqrt{n}\epsilon\}}e^{-\frac{1}{2}\nabla^2F(\bf 0)(\widetilde {\bf \Theta},\widetilde {\bf \Theta})}\ d\widetilde {\bf \Theta}\right)+\frac{e^{nF(\bf 0)}}{n}\left(L_G\cdot O_{F}\left(\frac{1}{\sqrt{n}}\right)+M_{G}\cdot O_{F}\left(\frac{1}{\sqrt{n}}\right)\right).
	\end{align*}
	Lastly, it is routine to see that 
	\begin{align*}
		\int_{\{\|\widetilde {\bf \Theta}\|<\sqrt{n}\epsilon\}}e^{-\frac{1}{2}\nabla^2F(\bf 0)(\widetilde {\bf \Theta},\widetilde {\bf \Theta})}\ d\widetilde {\bf \Theta} & =\frac{2\pi}{\sqrt{\det \nabla^2F(\bf 0)}}+O_{F}(e^{-n\epsilon^{2}}).
	\end{align*}
\end{proof}
\bibliographystyle{plain}
\bibliography{BIB}

\begin{thebibliography}{10}

\bibitem{Anantharaman:2000jq}
Nalini Anantharaman.
\newblock {Precise counting results for closed orbits of Anosov flows}.
\newblock {\em Annales Scientifiques de l'\'Ecole Normale Sup\'erieure.
  Quatri\`eme S\'erie}, 33(1):33--56, 2000.

\bibitem{Babillot:1998kh}
Martine Babillot and Fran\c{c}ois Ledrappier.
\newblock Lalley's theorem on periodic orbits of hyperbolic flows.
\newblock {\em Ergodic Theory Dynam. Systems}, 18(1):17--39, 1998.

\bibitem{Babillot:2000fv}
Martine Babillot and Marc Peign{\'e}.
\newblock {Homologie des g\'eod\'esiques ferm\'ees sur des vari\'et\'es
  hyperboliques avec bouts cuspidaux}.
\newblock {\em Annales Scientifiques de l'\'Ecole Normale Sup\'erieure.
  Quatri\`eme S\'erie}, 33(1):81--120, 2000.

\bibitem{BCK-2019}
Harrison Bray, Richard Canary, and Lien-Yung Kao.
\newblock Pressure metrics for deformation spaces of quasifuchsian groups with
  parabolics.
\newblock {\em Algebr. Geom. Topol.}, 23(8):3615--3653, 2023.

\bibitem{BCKMcounting}
Harrison Bray, Richard Canary, Lien-Yung Kao, and Giuseppe Martone.
\newblock Counting, equidistribution and entropy gaps at infinity with
  applications to cusped {H}itchin representations.
\newblock {\em J. Reine Angew. Math. (Crelle's Journal)}, 791:1--51, 2022.

\bibitem{BCKMpressure}
Harrison Bray, Richard Canary, Lien-Yung Kao, and Giuseppe Martone.
\newblock Pressure metrics for cusped {H}itchin components.
\newblock {\em Adv. Math.}, 435:Paper No. 109352, 24, 2023.

\bibitem{Burger:1993wb}
Marc Burger.
\newblock {Intersection, the {M}anhattan curve, and {P}atterson-{S}ullivan
  theory in rank {$2$}}.
\newblock {\em Internat. Math. Res. Notices}, (7):217--225, 1993.

\bibitem{Canary:2024}
Richard Canary.
\newblock Hitchin representations of {F}uchsian groups.
\newblock {\em EMS Surv. Math. Sci.}, 9(2):355--388, 2022.

\bibitem{CZZcusped}
Richard Canary, Tengren Zhang, and Andrew Zimmer.
\newblock Cusped {H}itchin representations and {A}nosov representations of
  geometrically finite {F}uchsian groups.
\newblock {\em Adv. Math.}, 404:Paper No. 108439, 67, 2022.

\bibitem{Chae_Hartog_Thm}
Soo~Bong Chae.
\newblock {\em Holomorphy and calculus in normed spaces}, volume~92 of {\em
  Monographs and Textbooks in Pure and Applied Mathematics}.
\newblock Marcel Dekker, Inc., New York, 1985.
\newblock With an appendix by Angus E. Taylor.

\bibitem{chow2023jordan}
Michael Chow and Hee Oh.
\newblock Jordan and cartan spectra in higher rank with applications to
  correlations.
\newblock {\em arXiv e-prints}, 2023.

\bibitem{chow2024multiplecor}
Michael Chow and Hee Oh.
\newblock Multiple correlations of spectra for higher rank anosov
  representations.
\newblock {\em arXiv e-prints}, 2024.

\bibitem{Dai}
Xian Dai and Giuseppe Martone.
\newblock Correlation of the renormalized {H}ilbert length for convex
  projective surfaces.
\newblock {\em Ergodic Theory Dynam. Systems}, 43(9):2938--2973, 2023.

\bibitem{Dalbo:1996vs}
Fran{\c c}oise Dal'bo and Marc Peign{\'e}.
\newblock {\em {Comportement asymptotique du nombre de g{\'e}od{\'e}siques
  ferm{\'e}es sur la surface modulaire en courbure variable}}.
\newblock Ast\'erisque, 1996.

\bibitem{Epstein87homology}
Charles~L. Epstein.
\newblock Asymptotics for closed geodesics in a homology class, the finite
  volume case.
\newblock {\em Duke Math. J.}, 55(4):717--757, 1987.

\bibitem{Fock:2006da}
Vladimir Fock and Alexander Goncharov.
\newblock {Moduli spaces of local systems and higher {T}eichm\"uller theory}.
\newblock {\em Publ. Math. Inst. Hautes \'Etudes Sci.}, 103(103):1--211, 2006.

\bibitem{Glorieux17}
Olivier Glorieux.
\newblock Counting closed geodesics in globally hyperbolic maximal compact
  {A}d{S} 3-manifolds.
\newblock {\em Geom. Dedicata}, 188:63--101, 2017.

\bibitem{Huber:1959ed}
Heinz Huber.
\newblock {Zur analytischen Theorie hyperbolischer Raumformen und
  Bewegungsgruppen}.
\newblock {\em Mathematische Annalen}, 138(1):1--26, 1959.

\bibitem{Kao:2018th}
Lien-Yung Kao.
\newblock {Manhattan Curves for Hyperbolic Surfaces with Cusps}.
\newblock {\em Ergodic Theory Dynam. Systems}, 40(7):1843--1874, 2020.

\bibitem{Kao:2019tr}
Lien-Yung Kao.
\newblock Pressure metrics and {M}anhattan curves for {T}eichm\"{u}ller spaces
  of punctured surfaces.
\newblock {\em Israel J. Math.}, 240(2):567--602, 2020.

\bibitem{Kato95-pertubation}
Tosio Kato.
\newblock {\em Perturbation theory for linear operators}.
\newblock Classics in Mathematics. Springer-Verlag, Berlin, 1995.
\newblock Reprint of the 1980 edition.

\bibitem{Katsuda88homology}
Atsushi Katsuda and Toshikazu Sunada.
\newblock Homology and closed geodesics in a compact {R}iemann surface.
\newblock {\em Amer. J. Math.}, 110(1):145--155, 1988.

\bibitem{Kessebohmer:2017ke}
Marc Kesseb\"{o}hmer and Sabrina Kombrink.
\newblock A complex {R}uelle-{P}erron-{F}robenius theorem for infinite {M}arkov
  shifts with applications to renewal theory.
\newblock {\em Discrete Contin. Dyn. Syst. Ser. S}, 10(2):335--352, 2017.

\bibitem{Benoit19Perturb}
Beno\^{i}t~R. Kloeckner.
\newblock Effective perturbation theory for simple isolated eigenvalues of
  linear operators.
\newblock {\em J. Operator Theory}, 81(1):175--194, 2019.

\bibitem{Lalley:1987df}
Steven Lalley.
\newblock Distribution of periodic orbits of symbolic and {A}xiom {A} flows.
\newblock {\em Adv. in Appl. Math.}, 8(2):154--193, 1987.

\bibitem{Lalley:1989eh}
Steven Lalley.
\newblock {Closed geodesics in homology classes on surfaces of variable
  negative curvature}.
\newblock {\em Duke Math. J.}, 58(3):795--821, 1989.

\bibitem{Lalley:1989jm}
Steven Lalley.
\newblock Renewal theorems in symbolic dynamics, with applications to geodesic
  flows, non-{E}uclidean tessellations and their fractal limits.
\newblock {\em Acta Math.}, 163(1-2):1--55, 1989.

\bibitem{Ledrappier:2008wq}
Fran{\c c}ois Ledrappier and Omri Sarig.
\newblock {Fluctuations of ergodic sums for horocycle flows on {$\Bbb
  Z^d$}-covers of finite volume surfaces}.
\newblock {\em Discrete Contin. Dyn. Syst.}, 22(1-2):247--325, 2008.

\bibitem{Margulis:1970gj}
Grigory Margulis.
\newblock {Applications of ergodic theory to the investigation of manifolds of
  negative curvature}.
\newblock {\em Functional Analysis and Its Applications}, 3(4):335--336, 1969.

\bibitem{Mauldin:2003dn}
Daniel Mauldin and Mariusz Urba\'nski.
\newblock {\em {Graph directed Markov systems}}, volume 148 of {\em Cambridge
  Tracts in Mathematics}.
\newblock Cambridge University Press, Cambridge, Cambridge, 2003.

\bibitem{Parry:1990tn}
William Parry and Mark Pollicott.
\newblock {Zeta functions and the periodic orbit structure of hyperbolic
  dynamics}.
\newblock {\em Ast\'erisque}, 187-188:1--268, 1990.

\bibitem{Pollicott:2012ud}
Fr{\'e}d{\'e}ric Paulin, Mark Pollicott, and Barbara Schapira.
\newblock {\em {Equilibrium states in negative curvature}}.
\newblock Number 373. Ast\'erisque, 2015.

\bibitem{Phillips87homology}
Ralph Phillips and Peter Sarnak.
\newblock Geodesics in homology classes.
\newblock {\em Duke Math. J.}, 55(2):287--297, 1987.

\bibitem{Pollicott91homology}
Mark Pollicott.
\newblock Homology and closed geodesics in a compact negatively curved surface.
\newblock {\em Amer. J. Math.}, 113(3):379--385, 1991.

\bibitem{Pollicott06correlation}
Mark Pollicott and Richard Sharp.
\newblock Correlations for pairs of closed geodesics.
\newblock {\em Invent. Math.}, 163(1):1--24, 2006.

\bibitem{Pollicott13correlation}
Mark Pollicott and Richard Sharp.
\newblock Correlations of length spectra for negatively curved manifolds.
\newblock {\em Comm. Math. Phys.}, 319(2):515--533, 2013.

\bibitem{Rockafellar98_Legendre}
R.~Tyrrell Rockafellar and Roger J.-B. Wets.
\newblock {\em Variational analysis}, volume 317 of {\em Grundlehren der
  mathematischen Wissenschaften [Fundamental Principles of Mathematical
  Sciences]}.
\newblock Springer-Verlag, Berlin, 1998.

\bibitem{Sambarino:2024za}
Andr\'es Sambarino.
\newblock Infinitesimal {Z}ariski closures of positive representations.
\newblock {\em J. Differential Geom.}, 128(2):861--901, 2024.

\bibitem{Sarig06phase}
Omri Sarig.
\newblock Continuous phase transitions for dynamical systems.
\newblock {\em Comm. Math. Phys.}, 267(3):631--667, 2006.

\bibitem{Sarig:2009wta}
Omri Sarig.
\newblock {\em {Lecture notes on thermodynamic formalism for topological Markov
  shifts}}.
\newblock 2009.

\bibitem{Schwartz:1993ty}
Richard Schwartz and Richard Sharp.
\newblock {The correlation of length spectra of two hyperbolic surfaces}.
\newblock {\em Comm. Math. Phys.}, 153(2):423--430, 1993.

\bibitem{Sharp93homology}
Richard Sharp.
\newblock Closed orbits in homology classes for {A}nosov flows.
\newblock {\em Ergodic Theory Dynam. Systems}, 13(2):387--408, 1993.

\bibitem{Sharp:1998if}
Richard Sharp.
\newblock {The Manhattan curve and the correlation of length spectra on
  hyperbolic surfaces}.
\newblock {\em Math. Z.}, 228(4):745--750, 1998.

\bibitem{Stadlbauer:2004va}
Manuel Stadlbauer.
\newblock {The return sequence of the {B}owen-{S}eries map for punctured
  surfaces}.
\newblock {\em Fund. Math.}, 182(3):221--240, 2004.

\end{thebibliography}

\end{document}